\documentclass[12pt]{article}
\usepackage {epsfig,amsmath,euscript}
\usepackage[latin1]{inputenc}
\usepackage[english]{babel}
\usepackage{color}
\usepackage[scr=boondoxo]{mathalfa}
\usepackage[amsmath,thmmarks,standard,thref]{ntheorem}
\usepackage{ae}
\usepackage{subfigure}
\usepackage[T1]{fontenc}
\usepackage{graphicx}
\usepackage{epstopdf}
\usepackage{soul}
\usepackage{color}
\sethlcolor{yellow}
\definecolor{rred}{rgb}{0.7,0.0,0.2}
\definecolor{bblue}{rgb}{0.2,0.0,0.7}

\newcommand{\secref}[1]{Section~\ref{sec:#1}}

\newcommand{\seclab}[1]{\label{sec:#1}}
\newcommand{\eqlab}[1]{\label{eq:#1}}
\renewcommand{\eqref}[1]{(\ref{eq:#1})}
\newcommand{\eqsref}[2]{(\ref{eq:#1}) and~(\ref{eq:#2})}

\newcommand{\figref}[1]{Fig.~\ref{fig:#1}}
\newcommand{\figlab}[1]{\label{fig:#1}}
\newcommand{\propref}[1]{Proposition~\ref{proposition:#1}}
\newcommand{\proplab}[1]{\label{proposition:#1}}

\newcommand{\lemmaref}[1]{Lemma~\ref{lemma:#1}}
\newcommand{\lemmalab}[1]{\label{lemma:#1}}

\newcommand{\thmref}[1]{Theorem~\ref{theorem:#1}}
\newcommand{\thmlab}[1]{\label{theorem:#1}}

\newcommand{\appref}[1]{Appendix~\ref{app:#1}}

\newcommand{\applab}[1]{\label{app:#1}}
\usepackage{enumitem}
\usepackage{tikz}

\title{{Relaxation oscillations in substrate-depletion oscillators close to the nonsmooth limit}} 
\author {Kristiansen, K. Uldall and Szmolyan, P.} 
\date {}
\begin{document}
\maketitle

%

\begin{abstract}
In this paper, we describe a novel type of relaxation oscillations occurring in a model of substrate-depletion oscillators. Using geometric singular perturbation theory, with blow-up as a key technical tool, we show that the oscillations in this planar model are produced by a complicated interplay between two stable nodes and a discontinuity set in the singular limit $\varepsilon\rightarrow 0$. This interplay produces a new mechanism for producing relaxation-type oscillations, which we also describe in a more general setting. 
\end{abstract}
%


\section{Introduction}
Biochemical and biophysical rhythms are ubiquitous characteristics of living organisms, from rapid oscillations of membrane potential in nerve cells to slow cycles of ovulation in mammals, \cite{Goldbeter1997,Tyson2002}.
In ODE models these oscillations correspond to limit cycles. Due to time scale separation 
 the governing  ODEs often have the form of slow-fast systems 
\begin{align}
 \varepsilon \dot x &=f(x,y,\varepsilon),\eqlab{xy0Eps}\\
 \dot y&=g(x,y,\varepsilon).\nonumber
\end{align}
and the periodic behaviour is of ``relaxation type''
consisting of long periods of ``in-activity'' interspersed with short periods of ``rapid transitions''. Mathematically, relaxation oscillations 
are defined as limit cycles $\Gamma_\varepsilon$ of \eqref{xy0Eps} that for $\varepsilon\rightarrow 0$  approach a closed loop $\Gamma_0$ consisting of a union of (i) \textit{slow orbits} of the reduced problem \eqref{xy0Eps}$_{\varepsilon=0}$ and (ii) \textit{fast orbits} of the associated layer problem:
\begin{align}
 x' &=f(x,y,0),\eqlab{layerhere}\\
 y' &=0.\nonumber
\end{align}
Here $()'=\frac{d}{d\tau}$, $\dot{()}=\frac{d}{dt}$ with $\tau = \varepsilon^{-1}t$ the fast time and $t$ the slow one. During the last decades these oscillations have been studied intensively using the framework of Geometric Singular Perturbation Theory (GSPT), \cite{jones_1995,fen3}. 
 The prototypical example for planar systems is the van der Pol system \cite{pol1926a,dumortier1996a,krupa_relaxation_2001} where the critical manifold $\{f(x,y,0)=0\}$ is $S$-shaped. Another interesting and widely studied phenomenon in these systems, commonly referred to as a canard explosion,
 is the  the extremely rapid growth of a small amplitude limit cycle generated in a Hopf bifurcation to large amplitude relaxation oscillations as a system parameter varies in an exponentially small intervall  \cite{dumortier1996a,krupa_relaxation_2001}. Under certain generic conditions, the canard explosion occurs, to leading order, when the $y$-nullcline transversally intersects a fold of the critical manifold \cite{dumortier1996a,krupa_relaxation_2001}.
 
The importance of slow fast dynamics and relaxation oscillations is well recognized in many areas of mathematical biology, perhaps
most notably in mathematical neuroscience, see e.g. \cite{terman2010a,izhi}.

Oscillations in molecular regulatory networks typically arise due to several nonlinear positive or negative feedback loops.
In simulations such systems also often show clear signs of slow-fast behaviour, their mathematical analysis 
based on slow-fast dynamics is however less developed. 
The influential paper \cite{Tyson2003} presents several simple but important mathematical models for molecular regulatory networks.
 In particular, \cite[Box 1]{Tyson2003} and \cite[Fig. 2]{Tyson2003} show equations, molecular wiring diagrams and bifurcation diagrams for three different models: negative feedback control loop \cite[Fig. 2(a)]{Tyson2003}, activator-inhibitor \cite[Fig. 2(b)]{Tyson2003} and substrate-depletion oscillators \cite[Fig. 2(c)]{Tyson2003}. These oscillators have all been proposed as the basis for oscillations in many different biological contexts, see \cite{Tyson2003} for further description and additional references.

 Mathematically, each model in \cite[Fig. 2]{Tyson2003} is described by ordinary differential equations which undergo two Hopf bifurcations as the ``signal'' parameter $S$ is varied. The attracting limit cycles, emerging in between these bifurcations, display oscillations with different phases, similar to relaxation oscillations in slow-fast systems.
However, the three models in \cite{Tyson2003} cannot all be described as slow-fast systems of the form \eqref{xy0Eps}. In fact, only the relaxation oscillations in the activator-inhibitor system in \cite[Fig. 2(b)]{Tyson2003} can be explained by this framework based upon \eqref{xy0Eps}, see e.g. \cite[Fig. 2(b), center column]{Tyson2003} where the red nullcline is precisely $S$-shaped as in the van der Pol system. The bifurcation diagram in \cite[Fig. 2(b), right column]{Tyson2003} can also be explained by GSPT: The two canard explosions, where the amplitudes of the limit cycles undergo a rapid increase within a small parameter regime, occur approximately at the parameter values $S$ when  the blue nullcline in \cite[Fig. 2(b), center column]{Tyson2003} intersects the red nullcline at the folds.

These observations, along with previous work on the subject, provide motivation for referring to periodic solutions with different, clearly separated phases as relaxation oscillations in a wider sense, perhaps without giving a general mathematical definition. Several examples of such oscillations have been studied recently, see e.g. \cite{kosiuk2015a,kristiansen2019d,2019arXiv190312232U}. In \cite{kosiuk2015a}, for example, a minimal model of cell division is considered which has a nonsmooth limit as $\varepsilon\rightarrow 0$.  The nonsmoothness of the system in \cite{kosiuk2015a} is due to the discontinuous pointwise limit of Michaelis-Menten terms $x/(\epsilon +x)$ as the Michaelis-Menten constant $\epsilon\rightarrow 0$.  In such systems, where the notions of slow and fast orbits have to be generalised, the blow-up method \cite{dumortier_1996,krupa_extending_2001} has proven extremely useful \cite{kosiuk2015a,kristiansen2018a,kristiansen2019d,kosiuk2011a,Gucwa2009783}. By this method, the authors of \cite{kosiuk2015a} show the existence of a limit cycle $\Gamma_\varepsilon$ that sticks to the discontinuity set for a fraction of its period and in this sense resembles classical relaxation oscillations. 
The model in \cite{kosiuk2015a} is also closely related to the negative feedback model in \cite[Fig. 2(a)]{Tyson2003}, compare with \cite[Box 1]{Tyson2003} and terms like $Y_P/(K_{m4}+Y_P)$ for $K_{m4}=0.01$ small, and the oscillations of this system can therefore be described by the same methods. 

Now, whereas the models in \cite{kosiuk2015a} and  \cite[Fig. 2(a)]{Tyson2003} are naturally viewed as smooth models approaching nonsmooth ones, the inverse process where a piecewise smooth system is regularized has also been an active area of research \cite{Sotomayor96,reves_regularization_2014,kristiansen2018a,krihog,krihog2,kaklamanos2019a}. Here blow-up has also proven very useful, see e.g. \cite{kristiansen2018a,2019arXiv190806781U} where the piecewise smooth fold singularities are analyzed using this method. 

In this paper, we provide a complete description of the relaxation oscillations in the last model in \cite{Tyson2003}, the substrate-depletion oscillator \cite[Fig. 2(c)]{Tyson2003}, by perturbing away from its very degenerate nonsmooth limit. We find that the main mechanism for the oscillations is based on  two piecewise smooth ``boundary node bifurcations'', see \cite{Kuznetsov2003,hogan2016a}, where a node intersects the discontinuity set. We use blow-up to desingularize the system, similar to applications in GSPT, and prove existence and non-existence of the relaxation oscillations in this model. 

\subsection{The substrate-depletion oscillator}

The molecular wiring diagram for the substrate-depletion oscillator is shown in \figref{wiring}. The system involves two chemical species, substrate $y$ and product $x$. The substrate $y$ is converted into $x$ by an autocatalytic process, i.e. a process which is further activated by $x$ itself.
 A possible mechanism for this would be an enzyme (not explicitly modelled) which is activated by $x$.
This autocatalytic reaction therefore accelerates the production  of $x$ until the concentration of the substrate $y$ is depleted. 
\begin{figure}[h!] 
\begin{center}
{\includegraphics[width=.49\textwidth]{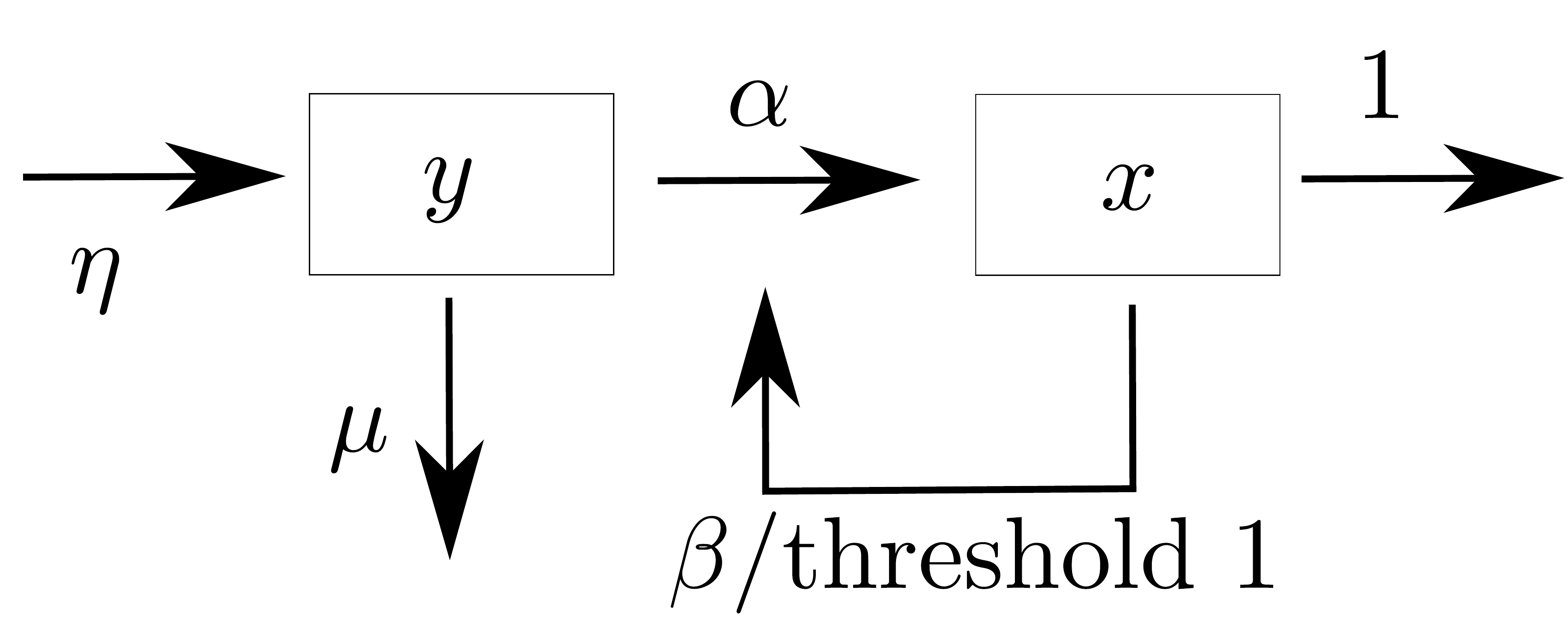}}
\end{center}
 \caption{Wiring diagram for the substrate-depletion oscillator. }
\figlab{wiring}
\end{figure}

In further details, we follow \cite{Tyson2003,casso} and assume that the substrate $y$ is
produced at a constant rate $\eta$,  is degraded at a constant rate $\mu$,  and converted into the product $x$ with a basic reaction rate $\alpha$
which is increased up to $\alpha + \beta$ by the autocatalytic process if $x$ is well above a critical threshold (which we normalize to $1$).
We assume that the product $x$ itself is degraded at a rate $k$ which we also normalise to $1$. For further reference, see the corresponding system of differential
equations in \eqref{xy}.

 The critical threshold for the autocatalytic process is in \cite{Tyson2003} modelled by the Goldbeter-Koshland, sigmoidal function:
\begin{align}
G_\varepsilon(x) = \frac{2x\varepsilon}{1-x+\varepsilon(1+x)+\sqrt{(1-x+\varepsilon(1+x))^2-4(1-x)x\varepsilon}},\quad x\ge 0,\eqlab{GK}
\end{align}
see \cite[Box 1, Fig. 2(c)]{Tyson2003}. The graph of $G_\varepsilon$ is shown in \figref{GK}(a) for three different values of $\varepsilon$: $\varepsilon=10^{-k}$, $k=1,2,3$.   Algebraic manipulations show the pointwise convergence
\begin{align*}
 G_\varepsilon(x) \rightarrow \left\{\begin{array}{cc}
                                        0& \quad \text{for}\quad x\in [0,1),\\
                                        1&\quad \text{for}\quad x>1,
                                       \end{array}\right.
\end{align*}
for $\varepsilon\rightarrow 0^+$.
\begin{figure}[h!] 
\begin{center}
\subfigure[]{\includegraphics[width=.49\textwidth]{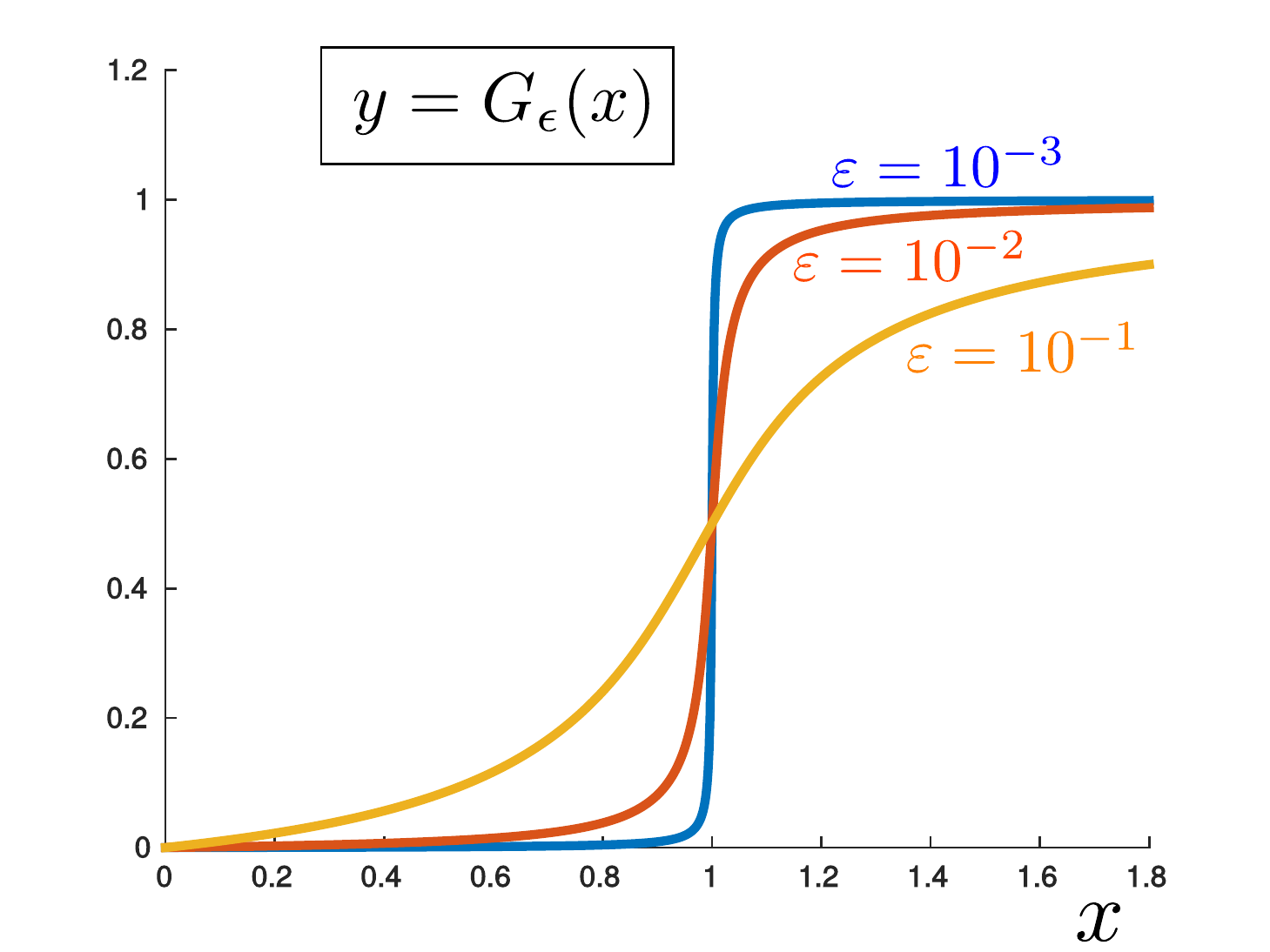}}
\subfigure[]{\includegraphics[width=.49\textwidth]{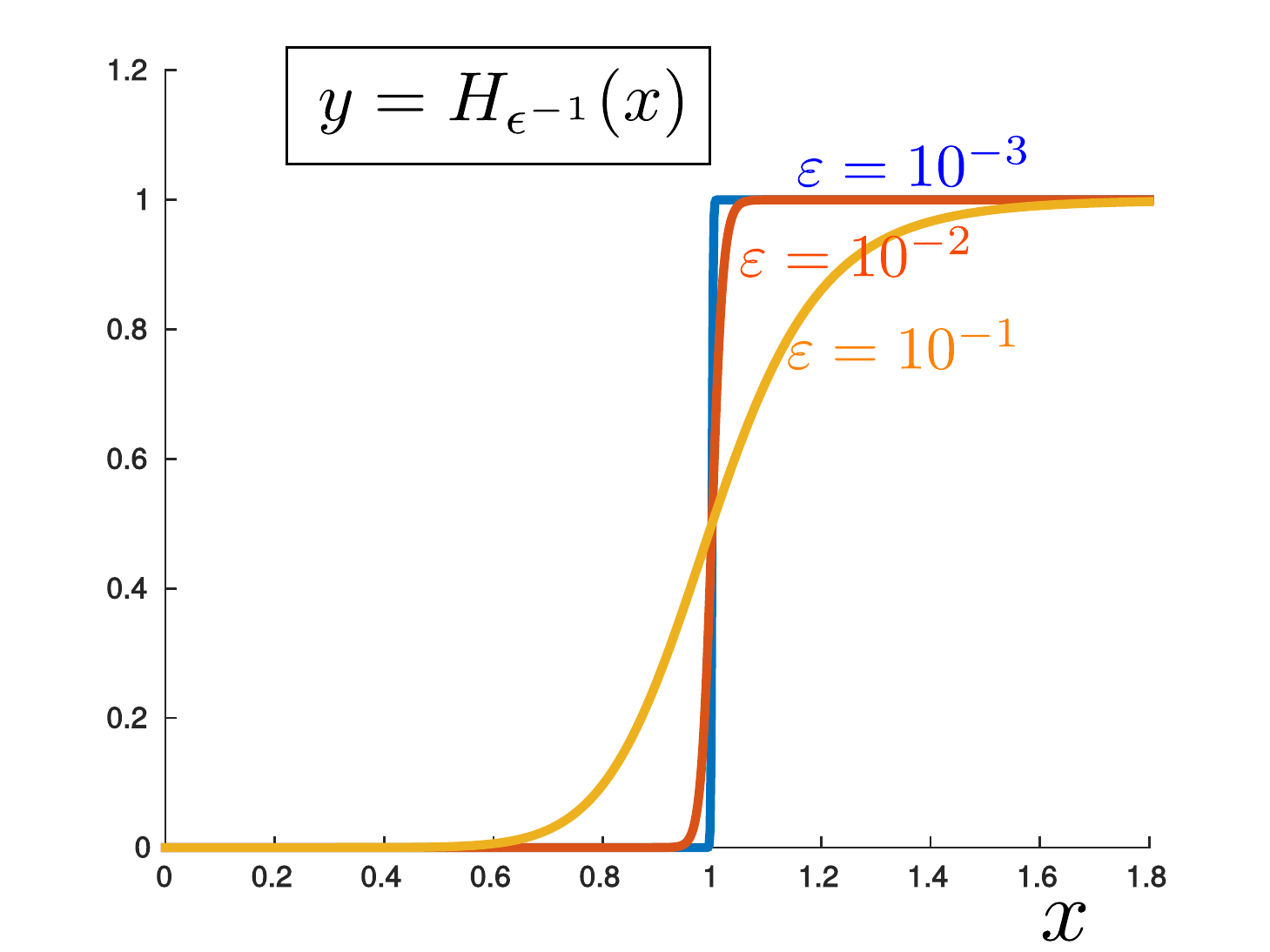}}
\subfigure[]{\includegraphics[width=.49\textwidth]{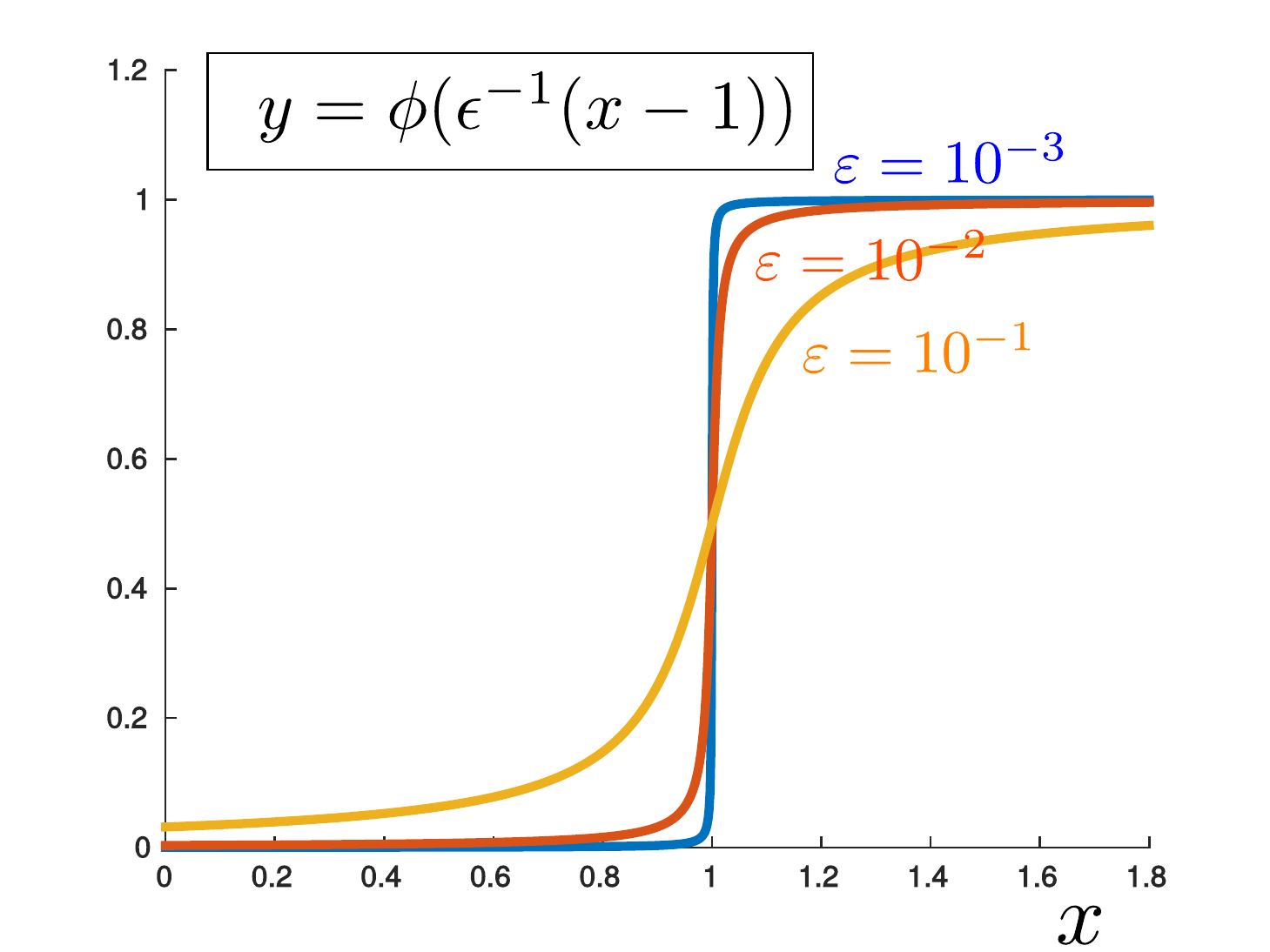}}
\end{center}
 \caption{Graphs of three different sigmoidal functions. In (a): $G_\varepsilon$ is the Goldbeter-Koshland function \eqref{GK}. In (b): $H_n$ is the Hill's function \eqref{hill}. In (c): $\phi$ is $\frac12 + \frac{1}{\pi}\arctan$. }
\figlab{GK}
\end{figure}
In \cite[Eq. (2.1.1)]{casso} the autocatalytic regulation is modelled using a Hill's function:
\begin{align}
 H_n(x) = \frac{x^n}{1+x^n},\eqlab{hill}
\end{align}
with $n\gg 1$. See the graph of $H_n$ in \figref{GK}(b) for $n=10^k$, $k=1,2,3$. Abstracting from this, we shall in this paper consider a general sigmoidal function of the following form:
\begin{align}
 x\mapsto  \phi(\varepsilon^{-1} (x-1)),\eqlab{phi}
\end{align}
with $\phi:\mathbb R \rightarrow \mathbb R$ a smooth function satisfying: 
\begin{align*}
\phi(z) &\in (0,1),\, \phi'(z) >0,\quad \text{for all}\quad z\in \mathbb R,
\end{align*}
and 
\begin{align}
 \phi(z) &\rightarrow 1^-\quad \text{for}\quad z\rightarrow \infty,\eqlab{philimit}\\
 \phi(z) &\rightarrow 0^+\quad \text{for}\quad z\rightarrow -\infty.\nonumber
 \end{align}
The function \eqref{phi} therefore gives the desirable sigmoidal function with threshold at $x=1$ for $0<\varepsilon\ll 1$, similar to $G_\varepsilon$ and $H_n$. Hence, we consider the following equations for the substrate-depletion oscillator:
\begin{align}
 \dot x &=(\alpha + \beta \phi(\varepsilon^{-1}(x-1)))y-x,\eqlab{xy}\\
 \dot y &=\eta -(\mu+\alpha+\beta \phi(\varepsilon^{-1}(x-1)))y.\nonumber
\end{align}

 We also make the following technical assumption
\begin{itemize}
 \item[(A)] The function $\phi$  has algebraic decay at $\pm \infty$: There exists a $k\in \mathbb N$ and smooth functions $\phi^{L} :[0,\infty)\rightarrow [0,\infty)$ and $\phi^{R} :[0,\infty)\rightarrow [0,\infty)$ (compactly: $\phi^{L/R}$, a notation we adapt frequently in the following)
\begin{align*}
 \phi(z) &= (-z)^{-k} \phi^L((-z)^{-1}),\quad \text{for}\quad z<0,\\
 \phi(z) &=1-z^{-k}\phi^R(z^{-1}),\quad \text{for}\quad z>0,
  \end{align*}
 and 
 \begin{align*}
  \phi^{L/R} (0) >0.
  \end{align*}
  \end{itemize}
  The number $k$ is the (algebraic) decay rate of $\phi(z)$ as $z\rightarrow \pm \infty$. It is also possible to consider different decay rates at $\pm \infty$ but we find that this complicates the notation further and only leads to minor technical and somewhat unessential changes to our approach. 
 We shall use $\phi(z) = \frac12 + \frac{1}{\pi}\arctan(z)$ in our numerical computations below. Here $k=1$ and $\phi^R(0)=\phi^L(0)=\frac{1}{\pi}$. See \figref{GK}(c). As opposed to $G_\varepsilon$ and $H_n$, the graph of $x\mapsto \phi(\varepsilon^{-1}(x-1)))$ does not pass through the origin. This, however, has no biological significance and is not important for our approach. The algebraic decay in (A) excludes a function like $\frac12 \tanh(z)+\frac12$ from consideration. Here $\phi^{L/R}(0)=0$ for any $k\in \mathbb N$ due to the exponential decay of $\tanh$. Similar issues arise when studying Hill's function $H_n$ for $n\rightarrow \infty$. To study these functions we would have to use different methods, see e.g. \cite{kristiansen2017a}. Assumption (A) makes the analysis substantially easier.
 
 \begin{remark}
  A simple calculation shows that the Goldbeter-Koshland function $G_\varepsilon$ can be written in the following form
  \begin{align}
   G_\varepsilon(x) = \psi(\varepsilon^{-1}(x-1),\varepsilon),\eqlab{GKNew}
  \end{align}
where the function
\begin{align}
\psi(z,\epsilon) = {\frac {2+\epsilon\,\sqrt {4+z^2+2\,\epsilon\,z^2+4\,\epsilon\,{z}+ {\epsilon}^{2}{{z
}}^{2}}+2\,\epsilon+\epsilon\,{z}+{\epsilon}^{2}{z}}{ 
 \left(2-z+ \epsilon\,{z}+\sqrt {4+z^2+2\,
\epsilon\,z^2+4\,\epsilon\,{z}+{\epsilon}^{2}z^2} \right)\left( 1+\epsilon \right)  }}.\eqlab{gkfunc}
\end{align}
is
smooth in each of its arguments. Here we have normalized $G_\epsilon$ appropriately (without loss of generality) such that $\psi(\cdot,\epsilon)$ satisfies \eqref{philimit} for all $0\le \epsilon\ll 1$. Notice in particular that 
\begin{align*}
                         \psi(z,0) = \frac{2}{2-z + \sqrt{z^2+4}}.
                        \end{align*}
Direct calculation then also shows that the function $\psi$ satisfies $\psi'_z(z,\varepsilon)>0$ for all $0\le \epsilon\ll 1$ and the following modified version of (A):
\begin{align*}
 \psi(\epsilon_1^{-1},r_1\epsilon_1) = 1-\epsilon_1\psi^R(r_1,\epsilon_1), \\
 \psi(-\epsilon_3^{-1},r_3\epsilon_3) = \epsilon_3\psi^L(r_3,\epsilon_3),
\end{align*}
with $\psi^{L/R}$ smooth and satisfying $\psi^R(0,0)=\psi^L(0,0)=1$. As a result, we can also study the Goldbeter-Koshland function by our methods (setting $k=1$), see  \cite{2019arXiv190806781U}. But since the notation is slightly more complicated for regularization functions of the type \eqref{GKNew} we will henceforth only focus on the simpler functions $\phi$, see \eqref{phi}, that are independent of $\varepsilon$. 
 \end{remark}

 We suppose that 
 $\alpha$ and $\beta$ in \eqref{xy} satisfy
\begin{align}
 \alpha\in (0,1),\,\alpha+\beta>1.\eqlab{abasymption}
\end{align}
Otherwise relaxation-type oscillations can easily be excluded.
In this paper we will fix $\alpha$ and $\beta$ satisfying \eqref{abasymption} and use $\varepsilon$, $\mu$ and $\eta$ as our bifurcation parameters. In \cite{Tyson2003} $\mu=0$, but we shall see that $\mu>0$ changes the bifurcation diagram significantly.  In particular, we shall study \eqref{xy} near $\varepsilon=\mu=0,\,\eta=1$, restricting attention to the biological meaningful regime where $\varepsilon> 0$, $\mu\ge 0$. In fact, we will most frequently think of $\eta$ as the primary bifurcation parameter and then study how the bifurcation diagram changes with $\mu\ge 0$ for $0<\varepsilon\ll 1$. We therefore consider \eqref{xy} as a ``nonsmooth'' perturbation problem, being singular along $x=1$ for $\varepsilon=0$.
%



\subsection{Numerics}
Using the bifurcation-software AUTO we obtain bifurcation diagrams for $\alpha=0.5$, $\beta=2$, $\varepsilon=0.0064$ and $\phi = \frac12 +\frac{1}{\pi}\arctan$ similar to those in \cite{Tyson2003}. In \figref{muEq0} for example, we take $\mu=0$, the scenario considered in \cite[Fig. 2 (c)]{Tyson2003}. In \figref{muEq0} (a) we present a bifurcation diagram, using $\text{max}\,x$ as a norm on the $y$-axis. There are two sub-critical Hopf bifurcations around $\eta\approx 0.93$ and $\eta\approx 1.02$. The repelling limit cycles born in these local bifurcations are observed to belong to the same family of periodic orbits. More precisely, each local limit cycle bifurcates along near vertical segments ($\eta\approx 0.92$ and $\eta\approx 1.025$) into attracting limit cycles, in a phenomenon that resembles the canard explosion phenomenon \cite{krupa_relaxation_2001,broens1991a,dumortier1996a} known from e.g. the van der Pol oscillator \cite{pol1926a}. Through this analogy, the attracting limit cycles appearing as the almost straight line for $\eta\in (0.92,1.025)$ with amplitude around $\text{max}\,x\approx 1.6-1.8$, would be of relaxation-type. \figref{muEq0} (b), shows an example of an attracting limit cycle (in red) on this branch for $\eta=1$ using a phase portrait. Figure (c) shows the corresponding functions $x(t)$, $y(t)$ (dashed). For comparison, figure (d) shows $x(t)$, $y(t)$ (dashed) for a smaller value of $\varepsilon=10^{-5}$. The effect of decreasing $\epsilon$ is shown to extend the period of the periodic orbit. In particular, the time spent close to the discontinuity line $x=1$ increases by decreasing the value of $\varepsilon$. Consequently, we can clearly view the limit cycles as relaxation oscillations in a broader sense of the word, where a period of ``inactivity'' near $x=1$ is interspersed by periods of ``activity'' for $x>1$ and $x<1$. 

\begin{figure}[h!] 
\begin{center}
\subfigure[Bifurcation diagram]{\includegraphics[width=.49\textwidth]{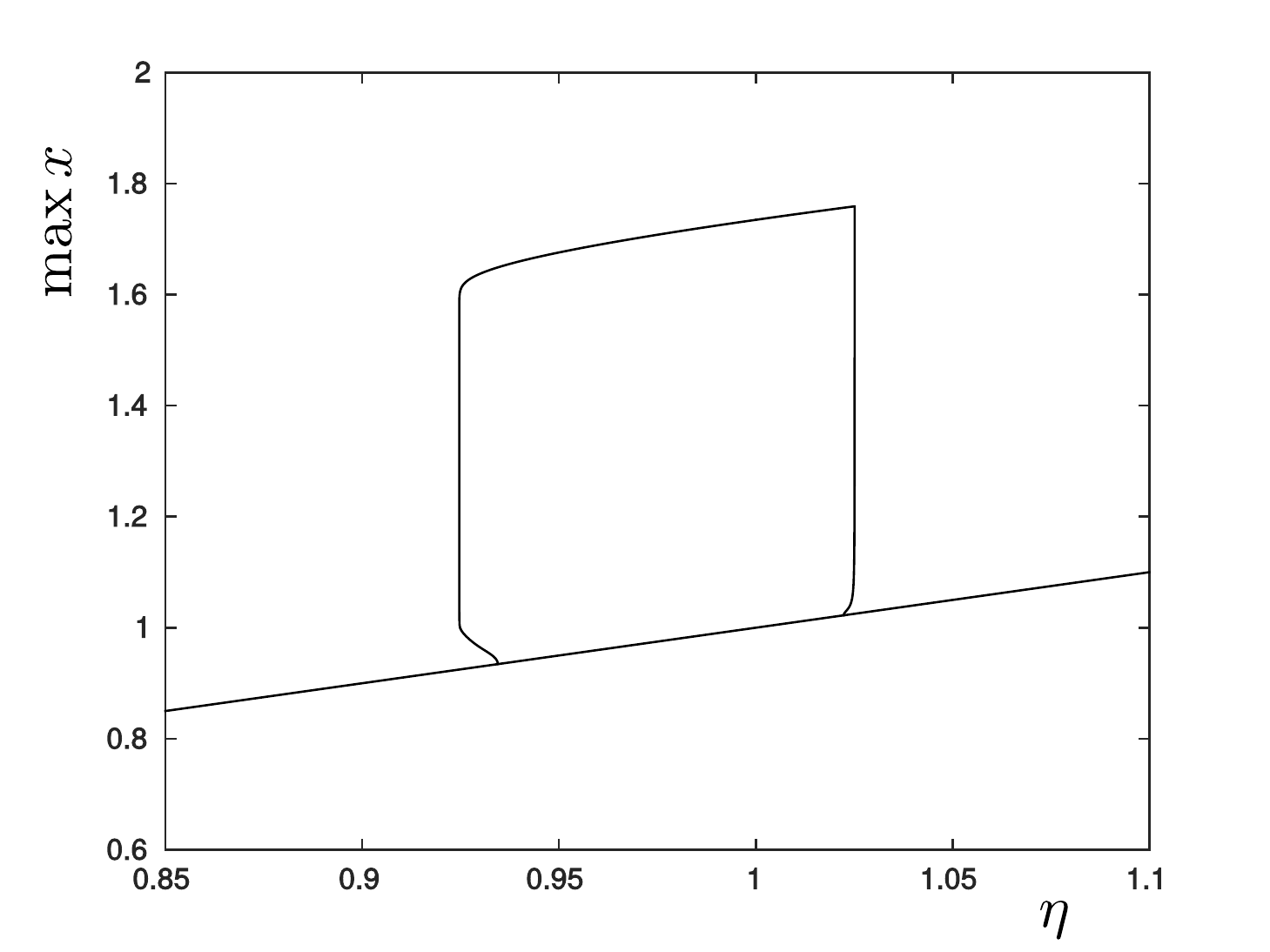}}
\subfigure[Phase portrait for $\eta=1$]{\includegraphics[width=.49\textwidth]{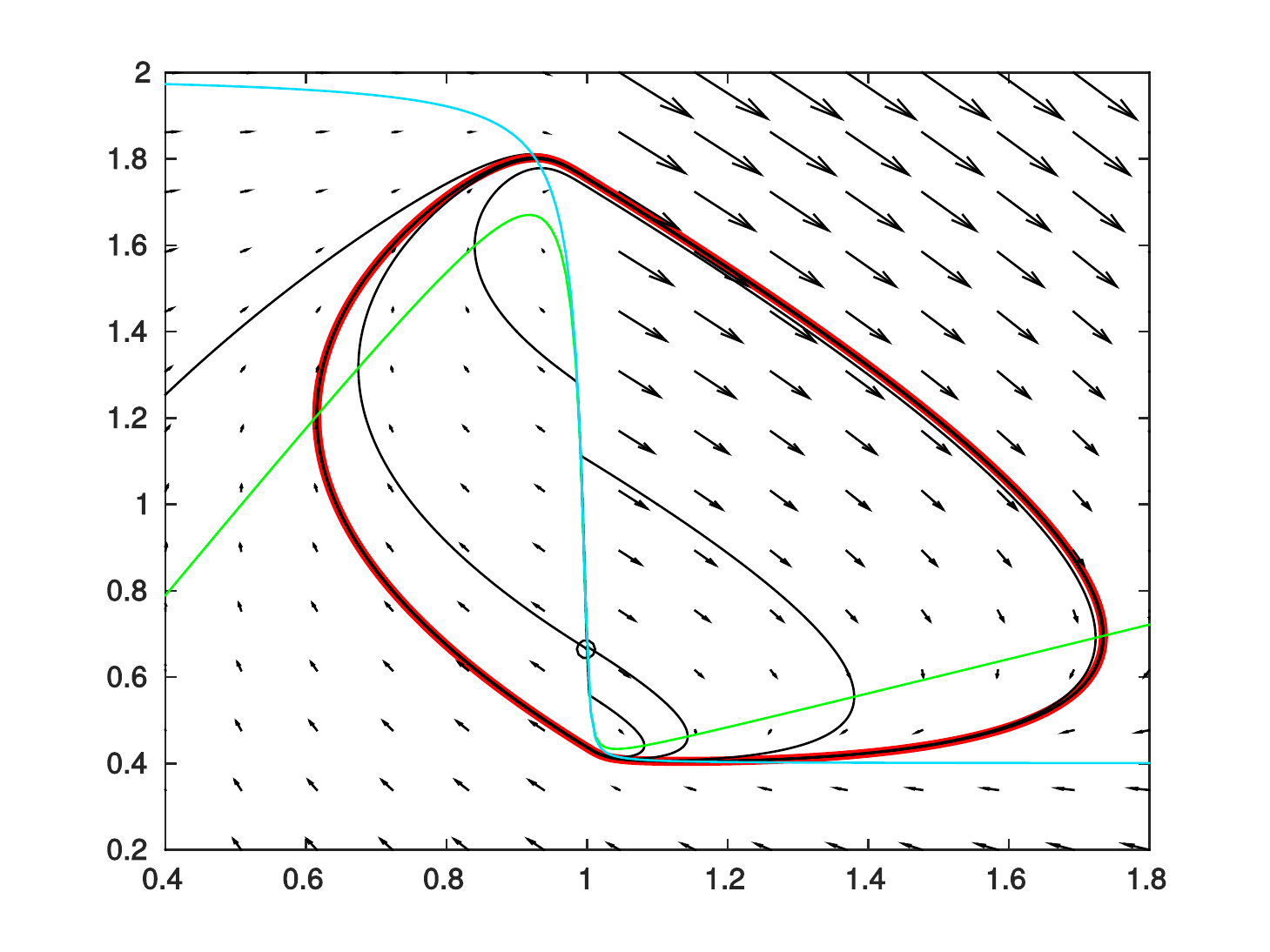}}
\subfigure[$x(t)$, $y(t)$ (dashed) for $\eta=1$ and $\varepsilon = 0.0064$]{\includegraphics[width=.49\textwidth]{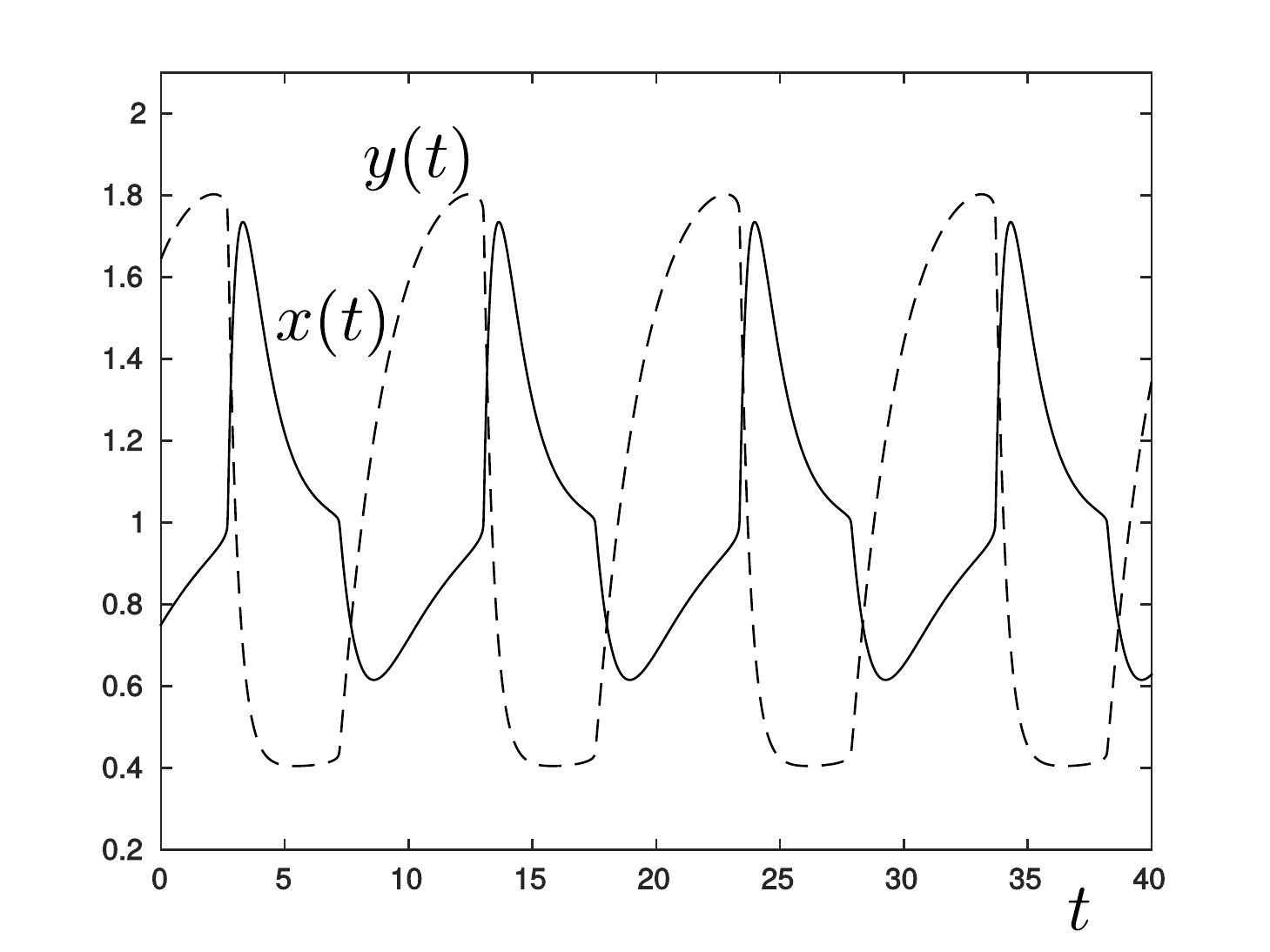}}
\subfigure[$x(t)$, $y(t)$ (dashed) for $\eta=1$ and $\varepsilon = 10^{-5}$]{\includegraphics[width=.49\textwidth]{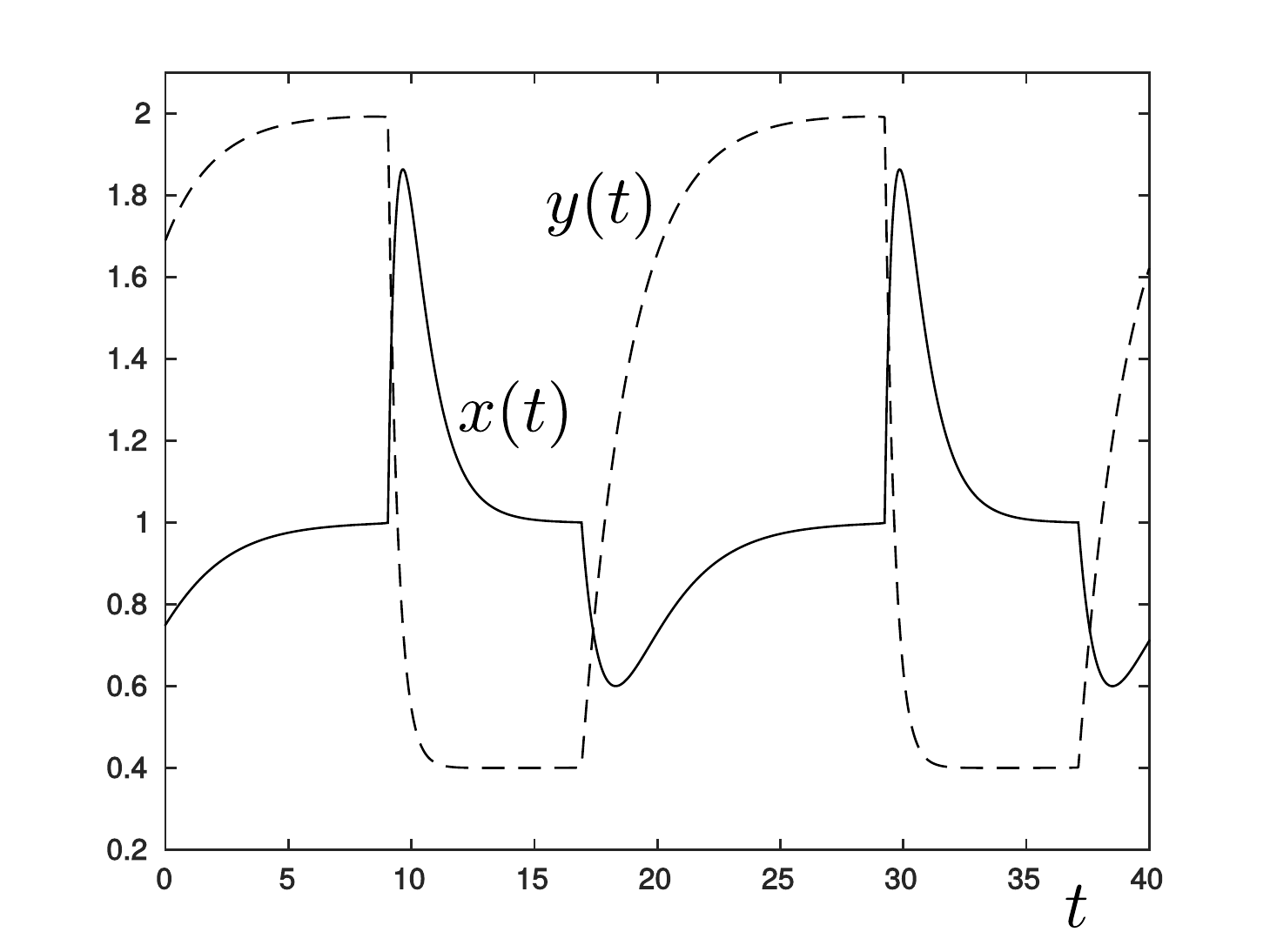}}
\end{center}
 \caption{Computations of relaxation oscillations using the bifurcation-software AUTO for $\alpha=0.5,\beta=2,\,\phi = \frac12 + \frac{1}{\pi}\arctan$ and $\mu=0$. Also in (a)-(c): $\varepsilon = 0.0064$ whereas $\varepsilon=10^{-5}$ in (d). Figure (a) shows the bifurcation diagram using $\text{max}\, x$ as a norm on the $y$-axis and $\eta$ as the bifurcation parameter. AUTO detects two sub-critical Hopf bifurcations around $\eta\approx 0.93$ and $\eta\approx 1.02$. The repelling limit cycles born in these local bifurcations are observed to belong to the same family of periodic orbits. Each local limit cycle bifurcates along near vertical segments ($\eta\approx 0.92$ and $\eta\approx 1.025$) into attracting limit cycles, in a phenomenon that resembles the canard explosion phenomenon known from e.g. the van der Pol oscillator. (b) shows the unique limit cycle, which is attracting, for $\eta=1$ in red. The green and blue curves are the nullclines. Black curves are different orbits of the system that all are asymptotic to the stable limit cycle. Points within the region confined by the red, closed curve are all backwards asymptotic to the unstable node indicated by a circle. (c) shows the limit cycle in (b) for $\eta=1$ and $\varepsilon=0.0064$ as functions $x(t)$, $y(t)$ (dashed) of time. For comparison, figure (d) shows $x(t)$, $y(t)$ (dashed) for a smaller value of $\varepsilon=10^{-5}$. }
\figlab{muEq0}
\end{figure}

However, if we increase $\mu$ slightly to $0.08$ then the relaxation oscillations disappear altogether. See \figref{muEq008}. Now the local repelling limit cycles born in the Hopf bifurcations further bifurcate in two separate homoclinics (see Figs. (a) and (b)). It is possible to combine the phenomena in \figref{muEq0} and \figref{muEq008}. For this we only have to reduce $\mu$ slightly from $0.08$ to $0.07936$. We illustrate this in \figref{muEq0079} using a phase portrait for $\eta=1.05940$. In this case, three limit cycles co-exist: two repelling ones (dashed and red) and one attracting, relaxation-like oscillation (full red). 
\begin{figure}[h!] 
\begin{center}
\subfigure[Bifurcation diagram using $\text{max}$-norm]{\includegraphics[width=.49\textwidth]{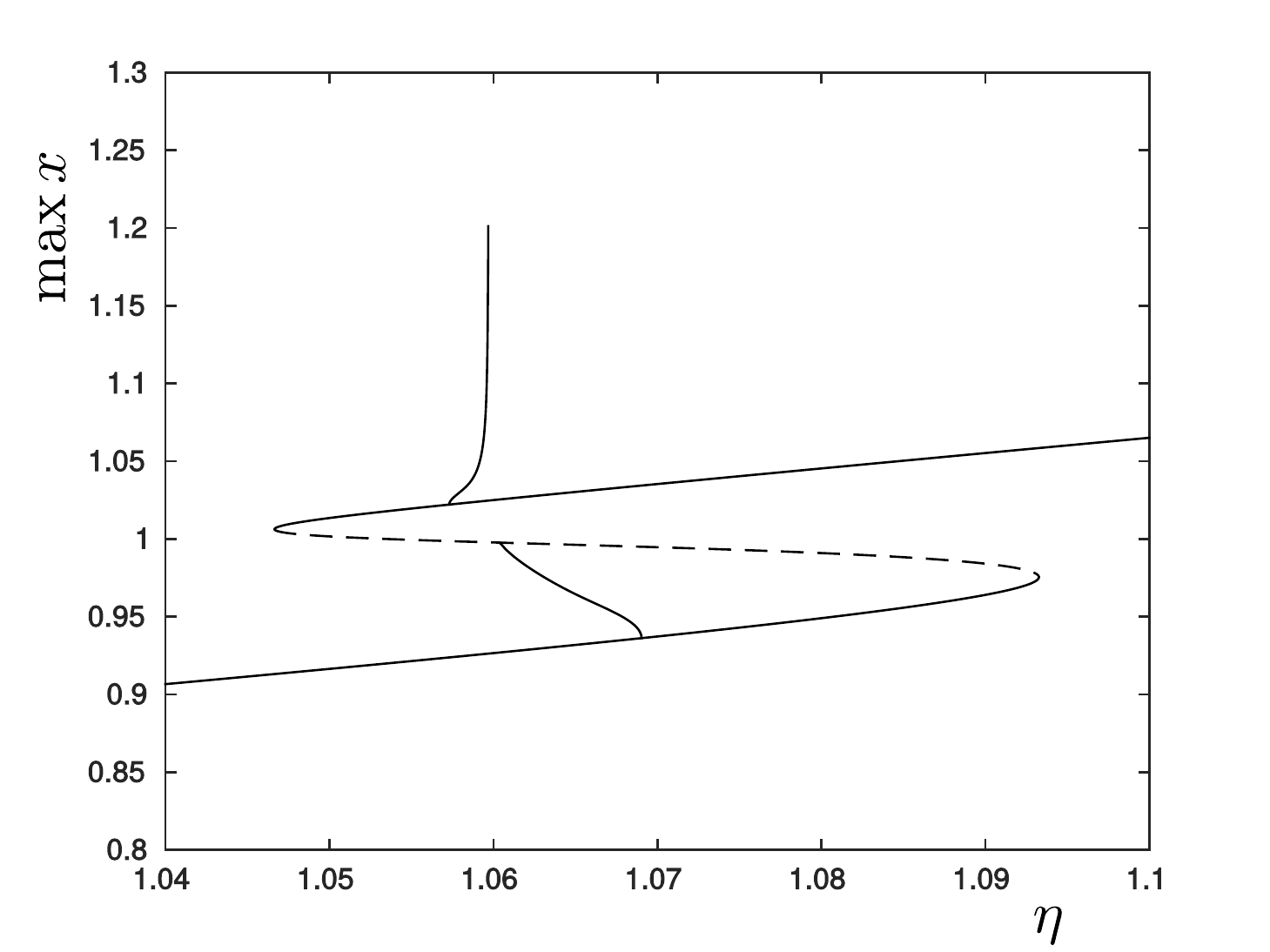}}
\subfigure[Bifurcation diagram using $L^2$-norm]{\includegraphics[width=.49\textwidth]{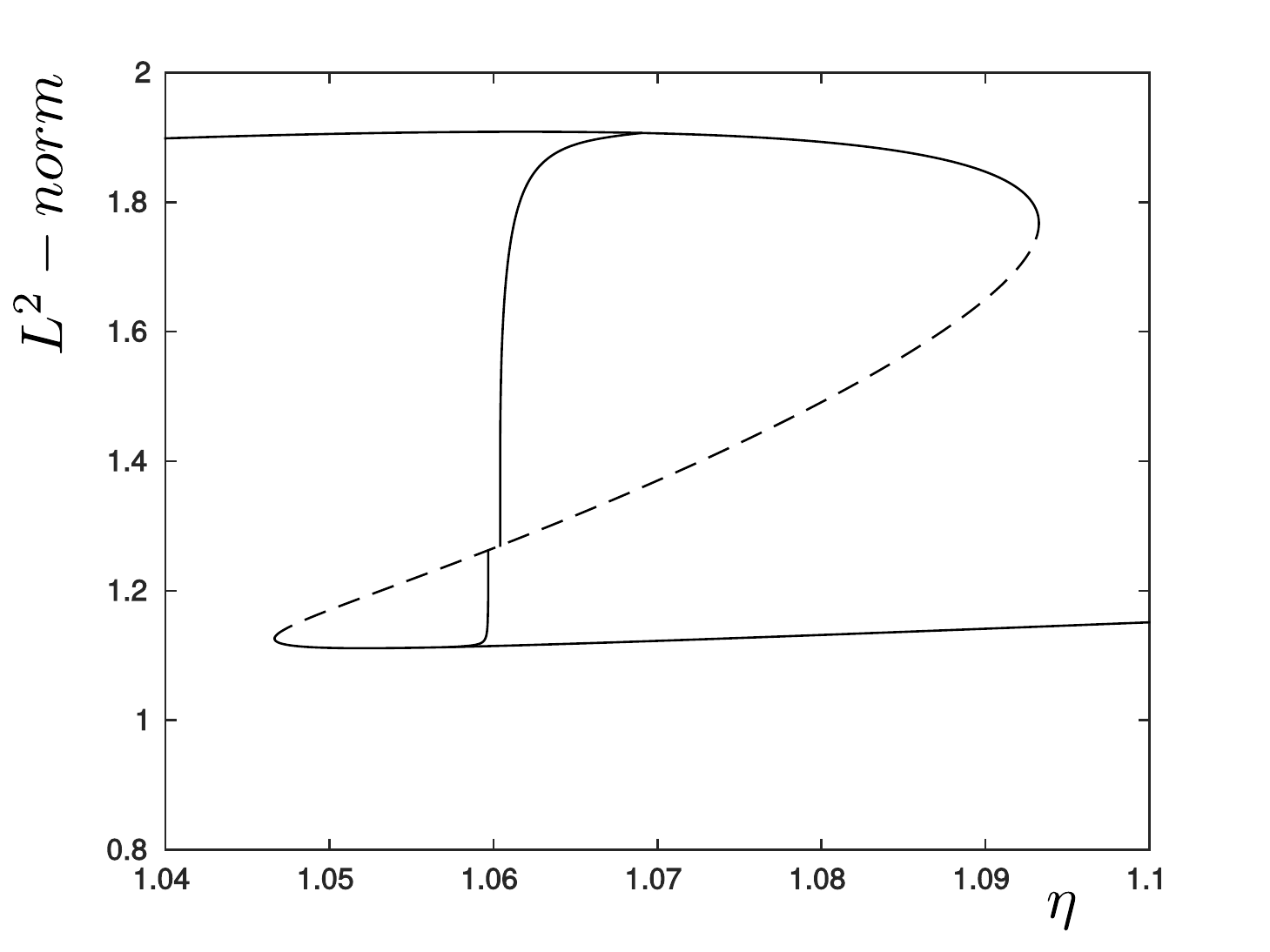}}
\subfigure[Phase portrait for $\eta\approx 1.0597$]{\includegraphics[width=.49\textwidth]{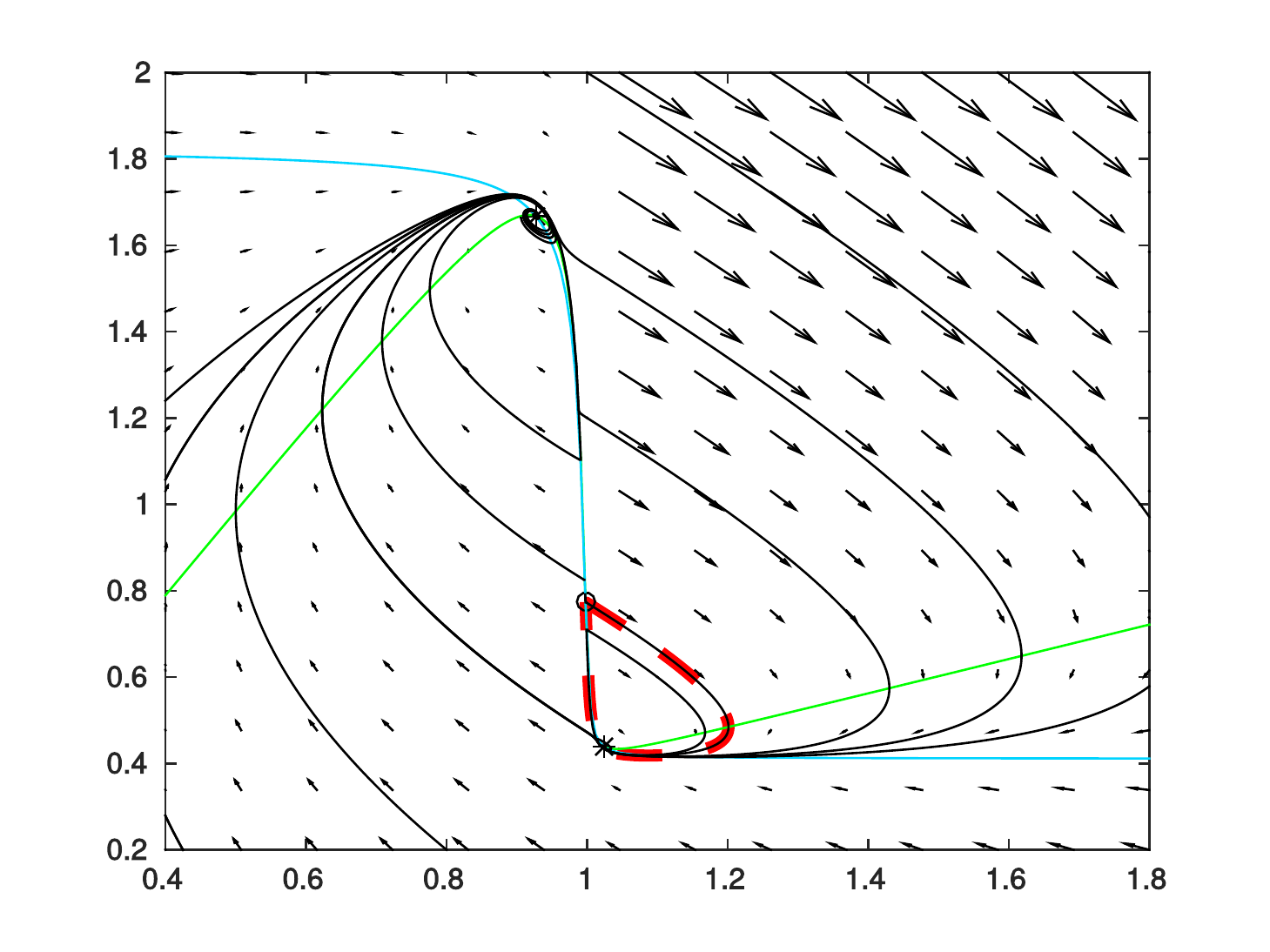}}
\subfigure[Phase portrait for $\eta\approx 1.0604$]{\includegraphics[width=.49\textwidth]{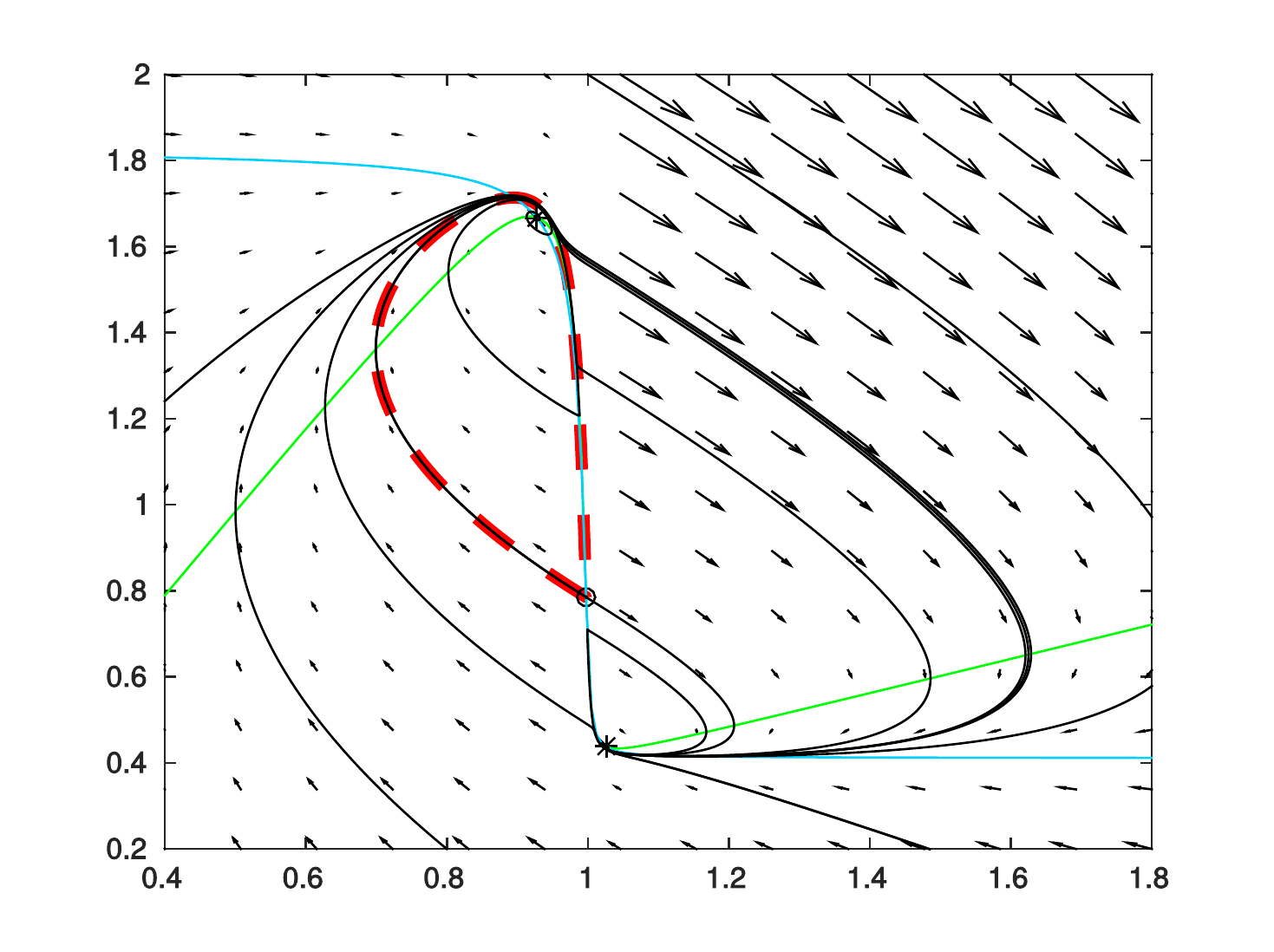}}
\end{center}
\caption{Computations of relaxation oscillations using bifurcation-software AUTO for the same parameter values as in \figref{muEq0}(a), (b) and (c) but now with $\mu=0.08$. (a) shows the bifurcation diagram using (as in \figref{muEq0}) $\text{max}\,x$ as a norm on the $y$-axis. $\eta$ is again the bifurcation parameter. In comparison with \figref{muEq0}, the branch of equilibria now bifurcates in an $S$-shaped fashion, given rise to a segment (dotted) of unstable (saddles) equilibria. As a result, the local limit cycles now bifurcate in two separate homoclinic bifurcations. This is more clearly visualized in (b) where we use AUTO's ``$L^2$-norm'' on the $y$-axis. Here the near-vertical branches of limit cycles end along the line of saddles, an indication of two homoclinic bifurcations. (c) and (d) show repelling limit cycles (dashed and red) near the two homoclinics. The orbits in black in (c) and (d) also show that generic points are forward asymptotic to two stable foci near $y\approx 1.6$ and $y\approx 0.4$ indicated by black stars. Therefore relaxation oscillations do not exist.}
\figlab{muEq008}
\end{figure}

%

\begin{figure}[h!] 
\begin{center}
{\includegraphics[width=.49\textwidth]{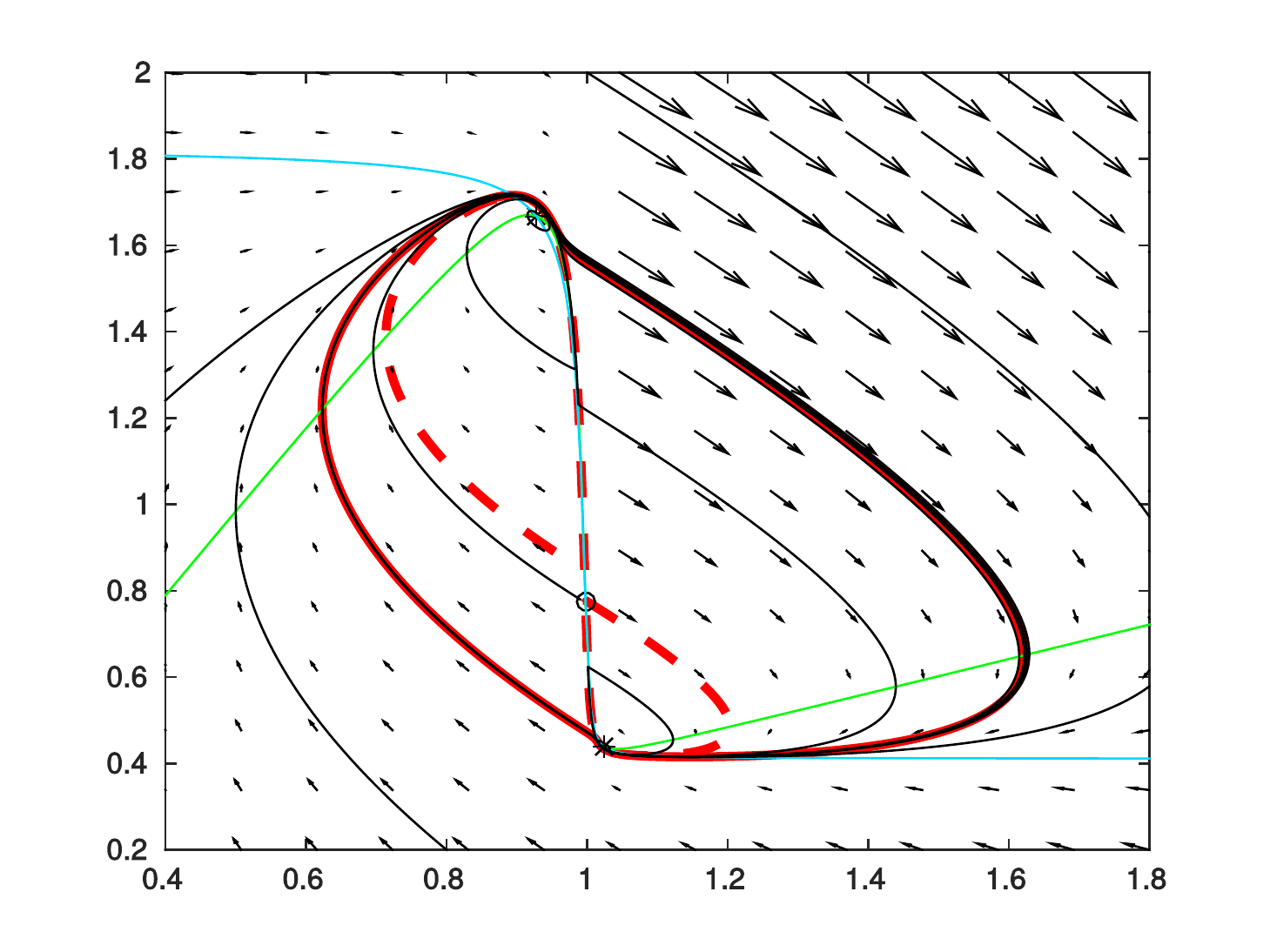}}
\end{center}
\caption{Phase portrait for $\mu=0.07936$ and $\eta=1.05940$. In this case three limit cycles co-exist: two repelling ones (dashed and red) and one attracting, relaxation-like oscillation (full red). The description of the more complicated bifurcation diagram is part of future work.  }
\figlab{muEq0079}
\end{figure}

\subsection{Aim of paper}
In this paper, we aim to describe a mathematical mechanism leading to the dynamics and the bifurcations in \figref{muEq0}, \figref{muEq008}, \figref{muEq0079}. We will focus on the existence and non-existence of the relaxation-type oscillations, leaving the full description of the complete bifurcation diagrams as part of future work. The latter will include an explanation of the explosive change in amplitude that is visible in \figref{muEq0} through a novel canard phenomena that occurs without the presence of attracting slow manifolds.  In the following, we will assume some familiarity with piecewise smooth systems, see e.g. \cite{Bernardo08,filippov1988differential,jeffrey_geometry_2011}, and the blow-up method in geometric singular perturbation theory, see e.g. \cite{krupa_extending_2001,kuehn2015,kristiansen2018a}.

\section{A preliminary blow-up analysis of \eqref{xy}}\seclab{singular0}

Setting $\varepsilon=0$ in \eqref{xy}$_{x\ne 1}$ gives by (A) a piecewise linear system where
\begin{align}
\dot x &= \alpha y-x,\eqlab{xN1system}\\
                 \dot y&=\eta-(\mu + \alpha)y,\nonumber
                 \end{align}
                 within $x<1$ and 
                 \begin{align}
 \dot x &=(\alpha+\beta)y-x,\eqlab{xP1system}\\
                               \dot y&=
                  \eta-(\mu + \alpha +\beta)y,\nonumber                
\end{align}
within $x>1$. The set $x=1$ is called a switching manifold in the literature of piecewise smooth dynamical systems \cite{Bernardo08}. However, to describe \eqref{xy} as a perturbation problem, it is useful to append a trivial equation for $\varepsilon$. We therefore consider the extended system
\begin{align}
\dot x &=(\alpha + \beta \phi(\varepsilon^{-1}(x-1)))y-x,\eqlab{xySlowExt}\\
 \dot y &=\eta -(\mu+\alpha+\beta \phi(\varepsilon^{-1}(x-1)))y,\nonumber\\
 \dot \varepsilon&=0,\nonumber
 \end{align}
  in the following. We let $y\in J$ with $J$ a sufficiently large interval throughout. Since \eqref{xySlowExt} is singular along $x=1$ for $\varepsilon=0$ it is useful to consider a separate time, the fast time $\tau = \varepsilon^{-1} t$. With respect to this time, \eqref{xySlowExt} becomes
 \begin{align}
x' &=\varepsilon((\alpha + \beta \phi(\varepsilon^{-1}(x-1)))y-x),\eqlab{xyFastExt}\\
 y' &=\varepsilon(\eta -(\mu+\alpha+\beta \phi(\varepsilon^{-1}(x-1)))y),\nonumber\\
 \varepsilon'&=0,\nonumber
 \end{align}
  Notice that for $\varepsilon=0$ this produces a vector-field which vanishes everywhere. Since \eqref{xyFastExt} has a lack of smoothness at $x=1$, the equilibria at $x=1$ are more singular and must be treated by blow-up. We therefore blow-up the singular line $x=1$, $y\in J$, $\varepsilon=0$ by setting 
\begin{align}
 \left. \begin{matrix} x&=& 1+r\bar x\\ y&=&\bar y\\ 
\varepsilon &= &r \bar \epsilon         
        \end{matrix}\right\}\quad r\ge 0,\,\bar y\in J,\,(\bar x,\bar \epsilon)\in S^1. \eqlab{initialBlowup}
\end{align}
The associated transformation $(r,\bar y,(\bar x,\bar \epsilon)) \mapsto (x,y,\varepsilon)$ defined by \eqref{initialBlowup} is a polar blow-up, its inverse blowing up $x=1$, $y\in J$ to a cylinder in the extended phase space $(x,y,\varepsilon)$-space. See \figref{xEq1Blowup}. 
Clearly, we can study any point $(x,\varepsilon)$, $\varepsilon\le \varepsilon_0$, by studying $(r,(\bar x,\bar \epsilon))$ with $r\ge 0$. Since $\varepsilon\ge 0$, only the points of the circle with $\bar \epsilon\ge 0$ are relevant.

We could describe \eqref{initialBlowup} using polar coordinates 
\begin{align}
r\ge 0,\,(\bar x,\bar \epsilon) = (\cos \theta,\sin \theta),\eqlab{polar}
\end{align}
but it is more useful to consider directional charts. We define these charts by requiring that \eqref{initialBlowup} in local coordinates $(r_1,y_1,\epsilon_1)$, $(x_2,r_2,y_2)$, $(r_3,y_3,\epsilon_3)$, takes the form that corresponds to setting $\bar x=-1$, $\bar \epsilon=1$, $\bar x=1$ in \eqref{initialBlowup}:
\begin{align}
 \bar x=-1:\quad &x = 1-r_1,&y&=y_1,&\varepsilon&=r_1 \epsilon_1,&r_1&\ge 0,\,\epsilon_1\ge 0,\eqlab{barXN1}\\
 \bar \epsilon=1:\quad &x = 1+r_2x_2,&y&=y_2,&\varepsilon&=r_2,&x_2&\in \mathbb R,\,r_2\ge 0,\eqlab{barEps1}\\
 \bar x =1:\quad &x = 1+r_3,&y&=y_3&\varepsilon&=r_3 \epsilon_3,& r_3&\ge 0,\,\epsilon_3\ge 0,\eqlab{barX1}
\end{align}
respectively.
We will henceforth refer to these charts as $\bar x=-1$, $\bar \epsilon=1$ and $\bar x=1$, respectively. Notice by \eqref{polar} that 
\begin{align*}
 \epsilon_1 = -\tan \theta,\,x_2= \cot \theta ,\, \epsilon_3 = \tan \theta,
\end{align*}
for $\theta  \in (\pi/2,3\pi/2)$, $\theta\in (0,\pi)$ and $\theta \in (-\pi/2,\pi/2)$, respectively.
These charts therefore describe parts of the circle $(\bar x,\bar \epsilon)\in S^1$ with $\bar x<0$, $\bar \epsilon>0$ and $\bar x>0$, respectively, by central projections from the respective planes at $\bar x=-1$, $\bar \epsilon=1$, $\bar x=1$. We will consider $(x_2,r_2)\in I_2\times [0,\varepsilon_0]$ with $I_2$ large but fixed in chart $\bar \epsilon=1$. We then fix small $U_1$ and $U_3$ accordingly so that in the charts $\bar x=\pm 1$, we have $(r_1,\epsilon_1)\in U_1$ and $(r_3,\epsilon_3)\in U_3$ and the charts $\bar x=-1,\,\bar \epsilon=1,\,\bar x=1$ cover a full neighborhood of $x=1,\,\varepsilon=0$. Since $y$ does not transform we will henceforth drop the subscript on $y$ in each of the charts. As is typical in singular perturbed systems of the form \eqref{xy}, see \cite{kristiansen2017a,kristiansen2018a}, we find that the pullback of the vector-field in \eqref{xyFastExt} by the blow-up transformation $(r,y,(\bar x,\bar \epsilon))\mapsto (x,y,\varepsilon)$ given by \eqref{initialBlowup} is well-defined, even for $r=0$. In fact, in the charts $\bar x=\mp 1$, the local forms of the vector-field, obtained from \eqref{xyFastExt} by inserting the expressions in \eqsref{barXN1}{barX1}, have $\epsilon_1$ and $\epsilon_3$, respectively, as common factors. This is a consequence of $\{\varepsilon=0\}$ being a set of equilibria for \eqref{xyFastExt}. We will therefore apply desingularization and divide the corresponding right hand sides by the common factors $\epsilon_1$ and $\epsilon_3$ in the local charts. In this way, we will recover the piecewise linear flows \eqref{xN1system} and \eqref{xP1system} within  $\epsilon_1=0$ and $\epsilon_3=0$, respectively, which allow for application of perturbation techniques.  Notice that, in contrast to the more usual blow-up approach of nonhyperbolic points \cite{dumortier2006a,dumortier1996a,krupa_extending_2001}, we do not divide by $r$. 

We will adopt the usual convention for blow-up \cite{krupa_extending_2001}: Objects, i.e points/orbits/manifolds, will in charts be enumerated by subscripts, say $M_i$, where $i$ coincides with the subscripts used in the corresponding coordinates. For example, an object in the chart $\bar x=-1$ \eqref{barXN1} will (typically) be written as $M_1$. Also, 
if $M_1$, under the application of the change of coordinates $K_{21}:(r_1,y_1,\epsilon_1)\mapsto (x_2,y_2,r_2)$, is (partially) covered by the chart $\bar \epsilon=1$, then we will denoted by $M_2$ there.
We adopt this convention throughout. Similarly, we will by $\overline M$ denote the global object obtained from the local versions $M_i$, $i=1,2,3$, in terms of blow-up variables $r\ge 0,\bar y\in J,(\bar x,\bar \epsilon)\in S^1$. 
\begin{figure}[h!] 
\begin{center}
{\includegraphics[width=.69\textwidth]{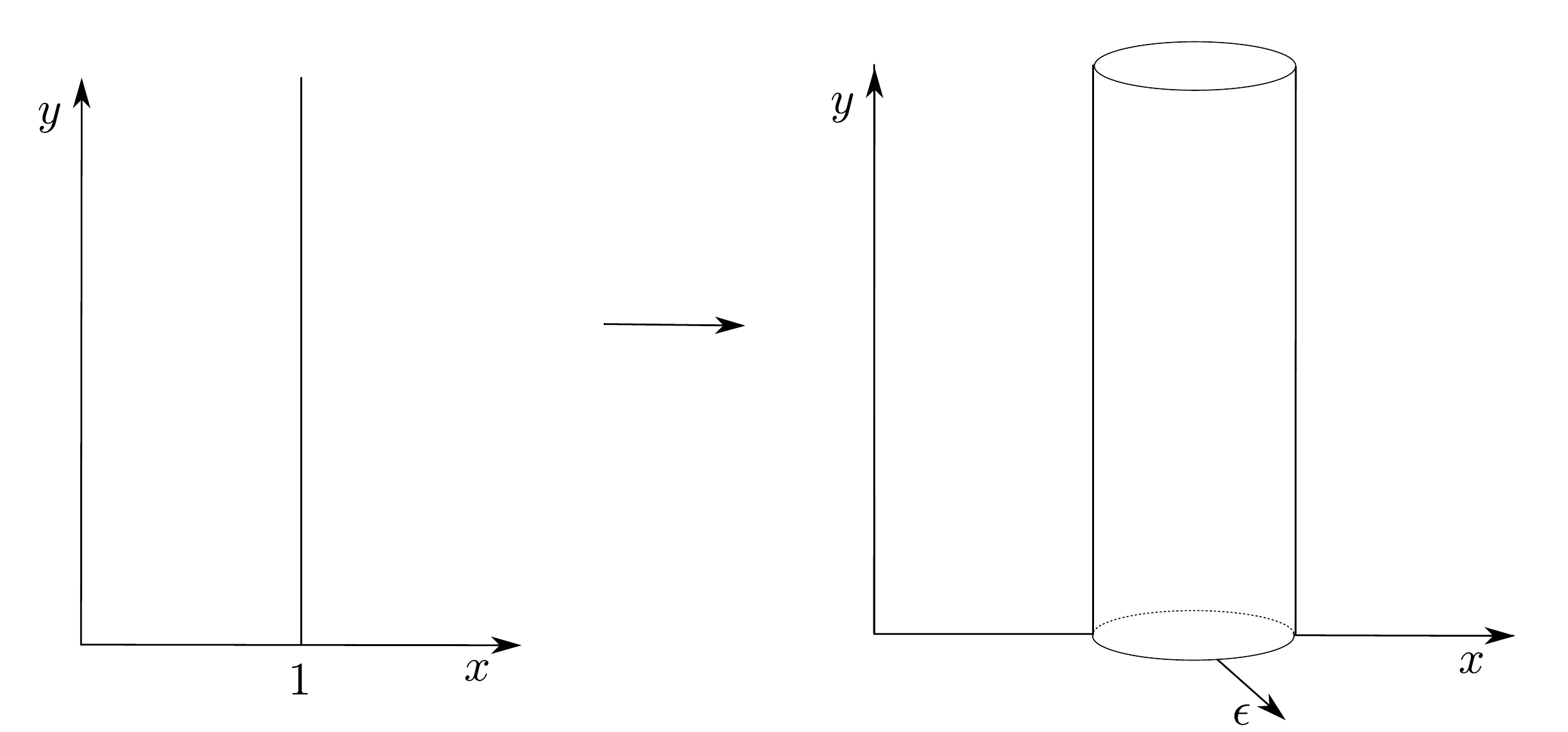}}
\end{center}
\caption{Blowup of $x=1$.}
\figlab{xEq1Blowup}
\end{figure}

In the following section, we will describe the dynamics in each of the directional charts. 
\subsection{Chart $\bar \epsilon=1$}
Inserting \eqref{barEps1} into \eqref{xyFastExt} gives a slow-fast system \cite{jones_1995}
\begin{align}
 x_2' &=(\alpha + \beta \phi(x_2))y-1-r_2x_2,\eqlab{fast}\\
 y '&=r_2(\eta -(\mu+\alpha+\beta \phi(x_2))y),\nonumber
\end{align}
and $r_2'=0$.
With respect to the slow time $t$, we obtain the equivalent form
\begin{align}
r_2 \dot x_2 &=(\alpha + \beta \phi(x_2))y-1-r_2x_2,\eqlab{slow}\\
 \dot y&=\eta -(\mu+\alpha+\beta \phi(x_2))y.\nonumber
 \end{align}
 This system is called the slow formulation of \eqref{fast}. Here $()'=r_2 \dot {()}$.
$x_2$ is called the fast variable while $y$ is called the slow variable. 

Setting $r_2=0$ in \eqref{fast} gives the layer problem
\begin{align}
 x_2'&=(\alpha + \beta \phi(x_2))y-1,\eqlab{layerHere}\\
 y'&=0.\nonumber
\end{align}
We illustrate the result of the following lemma in \figref{Sr2}.
\begin{lemma}
The critical manifold of \eqref{layerHere}
\begin{align*}
 S_{r,2}=\{(x_2,y)\in [0,\infty)\times J\vert y = (\alpha+\beta \phi(x_2))^{-1}\},
\end{align*}
is normally hyperbolic and repelling but noncompact. 

Let
\begin{align}
y^L = \alpha^{-1},\quad y^R = (\alpha+\beta)^{-1}.\eqlab{yNP}
\end{align}
Then the manifold $S_{r,2}$ has horizontal asymptotes
\begin{align*}
 y\rightarrow y^{L/R},
\end{align*}
for $x_2\rightarrow \mp \infty$, respectively, 
and it is therefore by assumption (A) contained within the strip $y\in (y^R,y^L)$.
\end{lemma}
\begin{proof}
 Simple calculations. 
\end{proof}

Now, setting $r_2=0$ in the slow system \eqref{slow} gives the reduced problem
\begin{align}
 0 &=(\alpha + \beta \phi(x_2))y-1,\eqlab{reduced}\\
 \dot y&=\eta -(\mu+\alpha+\beta \phi(x_2))y.\nonumber
\end{align}
For this system, we obtain the following.
\begin{lemma}
$S_{r,2}$ carries a reduced slow flow, described by the reduced problem:
\begin{align}
 \dot y &=\eta-1-\mu y,\quad (x_2,y)\in S_{r,2}.\eqlab{reducedNew}
\end{align}
For $\mu>0$, let
\begin{align*}
 y^S &= \mu^{-1}(\eta-1),
\end{align*}
and 
\begin{align*}
x_2^S = \phi^{-1}\left(-\beta^{-1}\left(\alpha-1/y^S\right)\right).
\end{align*}
Furthermore, we set
 \begin{align}
 \eta^L(\mu) = 1+\frac{\mu}{\alpha},\quad \eta^R(\mu) = 1+\frac{\mu}{\alpha+\beta}.\eqlab{etaN}
\end{align}
%
\begin{enumerate}[label=(\alph*)]
%
\item For any $\mu>0$ and
\begin{align*}
 \eta \in (\eta^R,\eta^L), 
\end{align*}
there exists a stable node of \eqref{reducedNew}
at 
\begin{align}
(x_2,y)=(x_2^S,y^S)\in S_{r,2}\eqlab{saddle}.
\end{align}
\item For $\eta^L(\mu)<\eta$ then $\dot y>0$ and $(x_2,y)\rightarrow (-\infty,y^L)$ in finite forward time and $(x_2,y)\rightarrow (+\infty,y^R)$ in finite backward time. 
\item For $\eta<\eta^R(\mu)$ then $\dot y<0$ and $(x_2,y)\rightarrow (+\infty,y^R)$ in finite forward time and $(x_2,y)\rightarrow (-\infty,y^L)$ in finite backward time. 
\item For $\mu=0$, $\eta^L=\eta^R=1$ and for $\eta=1$ the critical manifold $S_{r,2}$ is also a manifold of equilibria for the reduced problem. 
\end{enumerate}
\end{lemma}
\begin{proof}
Simple calculations. In particular, simplifying \eqref{reduced} gives \eqref{reducedNew}. The point defined by $y=y^S$ is a hyperbolic and attracting for \eqref{reducedNew} if $\mu\ne 0$. It is contained within $(y^R,y^L)$ if and only if $\eta\in (\eta^R,\eta^L)$.
\end{proof}
Now, suppose $\eta \in (\eta^R,\eta^L)$ and consider a compact neighborhood $U$ of \eqref{saddle} on $S_{r,2}$. Then by Fenichel's theory \cite{fen1,fen2,fen3,jones_1995}, $S_{r,2}\cap U$ perturbs smoothly to a locally invariant slow manifold of \eqref{slow} for all $0<\varepsilon\ll 1$. On the slow manifold, the stable node \eqref{saddle} of \eqref{reducedNew} becomes a saddle
\begin{align*}
p_2^S = \left(x_2^S,y^S\right),
\end{align*}
with $S_{r,2}\cap U $ as the local stable manifold and the critical fiber $\{x_2\in \mathbb R,\,y=y^S\}\cap U$ as the local unstable manifold as $r_2\rightarrow 0$. The perturbation of these objects for $0<r_2\ll 1$ is smooth by Fenichel's theory. We set 
\begin{align*}
 p^S = (1,y^S),
\end{align*}
in the $(x,y)$-variables, henceforth. We illustrate the results obtained in the scaling chart in \figref{Sr2}.

\begin{figure}[h!] 
\begin{center}
{\includegraphics[width=.49\textwidth]{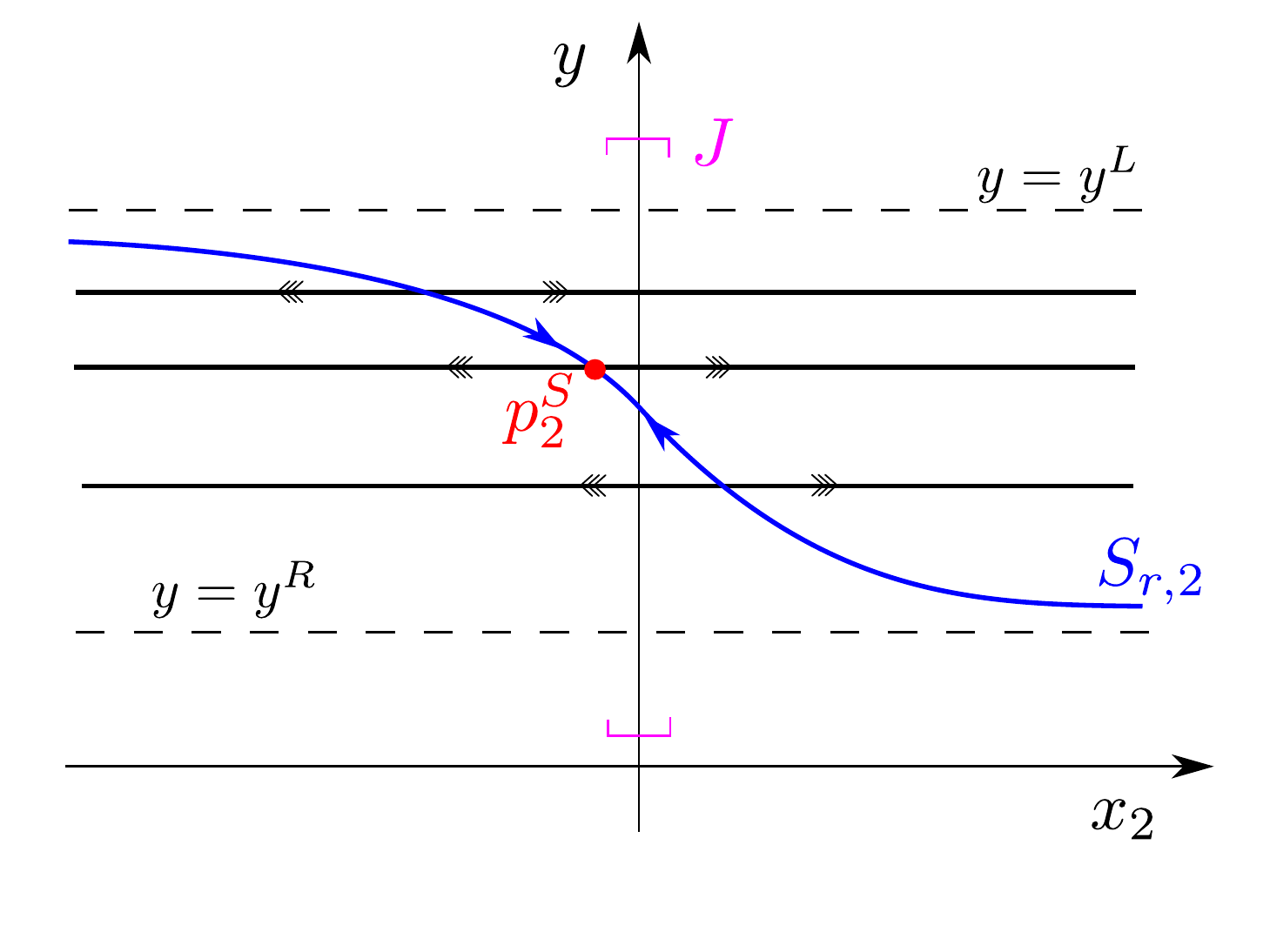}}
\end{center}
 \caption{Slow-fast dynamics in the scaling chart $\bar \epsilon=1$. The critical manifold $S_{r,2}$ is normally hyperbolic and repelling everywhere but noncompact. The lines $y=y^{L/R}$ are asymptotes. The reduced flow on $S_{r,2}$ has a stable node at $p_2^S$ for any $\mu>0$ and $\eta \in (\eta^R,\eta^L)$.}
\figlab{Sr2}
\end{figure}
\subsection{Chart $\bar x=-1$}\seclab{xbarN1sec}
Inserting \eqref{barXN1} into \eqref{xyFastExt} gives
\begin{align}
\dot r_1 &=-r_1 \epsilon_1\left((\alpha+\beta\epsilon_1^k \phi^L(\epsilon_1))y-1+r_1\right),\nonumber\\
\dot y&=r_1 \epsilon_1\left[\eta-(\mu+\alpha+\beta \epsilon_1^k\phi^L(\epsilon_1))y\right],\nonumber\\
\dot \epsilon_1&=\epsilon_1^2 \left((\alpha+\beta\epsilon_1^k \phi^L(\epsilon_1))y-1+r_1\right).\nonumber
\end{align}
As promised, $\epsilon_1\ge 0$ is a common factor on the right hand side. Division of the right hand side by this common factor produces the following desingularized system
\begin{align}
\dot r_1 &=-r_1\left((\alpha+\beta\epsilon_1^k \phi^L(\epsilon_1))y-1+r_1\right),\eqlab{barxN1Eqns}\\
\dot y&=r_1\left[\eta-(\mu+\alpha+\beta \epsilon_1^k\phi^L(\epsilon_1))y\right],\nonumber\\
\dot \epsilon_1&=\epsilon_1\left((\alpha+\beta\epsilon_1^k \phi^L(\epsilon_1))y-1+r_1\right),\nonumber
\end{align}
which we study in the following. 
Here $\{r_1=0\}$ and $\{\epsilon_1=0\}$ are two invariant sets. Their intersection $\{r_1=\epsilon_1=0\}$ corresponds to the edge $\overline E^L = \{(\bar x,\bar \epsilon)=(-1,0),\,y\in J\}$ of the blow-up cylinder. A simple calculation shows that it is a normally hyperbolic set of equilibria for \eqref{barxN1Eqns} for all $y\ne y^L$. Recall \eqref{yNP}. The point 
\begin{align}
 p^L_1 = (0,y^L,0),\eqlab{pN1}
\end{align}
is, however, fully nonhyperbolic, the linearization having only zero eigenvalues, and this point will therefore play an important role in the following. In the $(x,y)$-plane, it becomes 
\begin{align*}
 p^L = (1,y^L).
\end{align*}

On $r_1=0$, which corresponds to points on the cylinder, we obtain
\begin{align}
\dot y&=0,\eqlab{barxN1Eqnsr1Eq0}\\
\dot \epsilon_1&=\epsilon_1 \left((\alpha+\beta\epsilon_1^k \phi^L(\epsilon_1))y-1\right).\nonumber
\end{align}
Here we re-discover $S_{r,2}\cap \{x_2<0\}$ from chart $\bar \epsilon=1$ as a set of critical points of \eqref{barxN1Eqnsr1Eq0}:
\begin{align*}
 S_{r,1}= \{(r_1,y,\epsilon_1)\in \{0\}\times U_1 \vert y = (\alpha+\beta\epsilon_1^k \phi^L(\epsilon_1))^{-1},\,\epsilon_1>0\}.
\end{align*}
It retains its hyperbolicity properties and carries a reduced slow flow described by 
\begin{align*}
 \dot y = \eta-1-\mu y.
\end{align*}
The manifold $S_{r,1}$ ends at $p^L_1$ on $\{r_1=\epsilon_1=0\}$.


Within $\{\epsilon_1=0\}$ we just re-discover the $x<1$ PWL system \eqref{xN1system}. Indeed, if we divide the right hand side by $r_1$ then 
\begin{align*}
 \dot r_1 &= -(\alpha y-1+r_1),\\
 \dot y &=\eta-(\mu+\alpha)y.
 \end{align*}
Setting $r_1=x-1$ then gives \eqref{xN1system}. In terms of the PWS system, the point $p^L$ is a tangency point in the $(x,y)$-plane for the $x<1$ system with the discontinuity set $x=1$, \cite{Bernardo08,jeffrey_geometry_2011}. System \eqref{xN1system} has a unique stable node at a point $z^L$ defined by
\begin{align}
(x,y) = \left(\frac{\alpha \eta}{\mu+\alpha},\frac{\eta}{\mu+\alpha}\right),\eqlab{nodeLeft}
\end{align}
with associated eigensolutions $(\lambda_i,v_i)$, $i=1,2$, of the linearization:
\begin{align*}
 \lambda_1 &=-1,v_1 = (1,0)^T,\\
 \lambda_2 &=-(\mu+\alpha),v_2 = (\alpha,1-(\mu+\alpha))^T.
\end{align*}
By \eqref{abasymption}, $\lambda_1<\lambda_2<0$ for $\mu$ sufficiently small and the eigenvector $v_2$ is therefore weak while $v_1$ is strong. 
Furthermore, $z^L$ \eqref{nodeLeft} is contained within $x<1$ if and only if 
\begin{align}
 \eta < \eta^L(\mu),
\end{align}
recall \eqref{etaN}.
It is on $x=1$ if $\eta=\eta^L(\mu)$ (in which case $z^L=p^L$) and within $x>1$ (and therefore ``virtual'') if $\eta>\eta^L(\mu)$.
\subsection{Chart $\bar x=1$}
The analysis in this chart is similar to chart $\bar x=-1$. $\epsilon_3\ge 0$ is a common factor on the right hand side. As in chart $\bar x=-1$ we therefore apply a desingularization through the division of the right hand sides by this common factor. Then the edge $E^R_3=\{r_3=\epsilon_3=0,y\in J\}$ is normally hyperbolic for all $y\ne y^R$. Recall \eqref{yNP}. The point 
\begin{align*}
 p^R_3 = (0,y^R,0),
\end{align*}
in the $(r_3,y,\epsilon_3)$-variables, is, however, fully nonhyperbolic, the linearization having a tripple zero eigenvalue. It corresponds to
\begin{align*}
 p^R = (1,y^R),
\end{align*}
in the $(x,y)$-plane. For the PWS system, it is a tangency point of the $x>1$ system with the discontinuity set. 
Similarly, we re-discover $S_{r,2}$ from the chart $\bar \epsilon=1$ as a set of normally hyperbolic and repelling critical points.  It is a graph over $\epsilon_3>0$ and ends at $p^R_3$.  Within $\{\epsilon_3=0\}$ we re-discover the $x>1$ PWL system \eqref{xP1system} setting $r_3=x-1$. This system has a unique stable node at a point $z^R$ defined by
\begin{align}
(x,y)=\left(\frac{(\alpha+\beta) \eta}{\mu+\alpha+\beta},\frac{\eta}{\mu+\alpha+\beta}\right),\eqlab{nodeRight}
\end{align}
with associated eigensolutions $(\lambda_i,v_i)$, $i=1,2$, of the linearization:
\begin{align*}
 \lambda_1 &=-1,v_1 = (1,0)^T,\\
 \lambda_2 &=-(\mu+\alpha+\beta),v_2 = (\alpha+\beta,1-(\mu+\alpha+\beta))^T.
\end{align*}
By \eqref{abasymption}, $\lambda_1<\lambda_2<0$ for $\mu\ge 0$ and the eigenvector $v_2$ is therefore weak while $v_1$ is strong. 
$z^R$ \eqref{nodeLeft} is contained within $x>1$ if and only if 
\begin{align}
 \eta > \eta^R(\mu),
\end{align}
recall \eqref{etaN}. 
It is on $x=1$ if $\eta=\eta^R(\mu)$ (in which case $z^R=p^R$) and within $x>1$ (and therefore ``virtual'') if $\eta>\eta^R(\mu)$. 

\subsection{Singular picture for $\varepsilon=0$}
Following the analysis in the charts we can now present the dynamics obtained by blow-up. See \figref{pwsblowup}. Upon blow-up we obtain a hyperbolic but noncompact critical manifold $\overline S_r$. It is asymptotic to fully nonhyperbolic points $\overline p^{L}$ and $\overline p^R$, respectively. These points lie on the edges $\overline E^L$ and $\overline E^R$, which away from $\overline p^L$ and $\overline p^R$, consist of partially hyperbolic equilibria. Within $\bar \epsilon=0$ we recover the piecewise linear flows. In \figref{bif} we also present a ``singular'' bifurcation diagram of the equilibria using $\eta$ as a bifurcation parameter. (a) is for $\mu>0$ and (b) is for $\mu=0$.  Notice that for $\mu>0$, two stable equilibria on either side of $x=1$ co-exist for $\eta \in(\eta^R,\eta^L)$. For these values of $\eta$, the reduced flow on $\overline S_r$ also has a fix-point, being of saddle-type for the full system, see \figref{pwsblowup}(b). In the piecewise smooth literature, the piecewise smooth system \eqref{xN1system} and \eqref{xP1system} is said to undergo ``boundary node bifurcations'' at $\eta=\eta^R$ and $\eta=\eta^L$, see \cite{Kuznetsov2003}.
\begin{figure}[h!] 
\begin{center}
\subfigure[$\eta<\eta^R<\eta^L$, $\mu>0$]{\includegraphics[width=.49\textwidth]{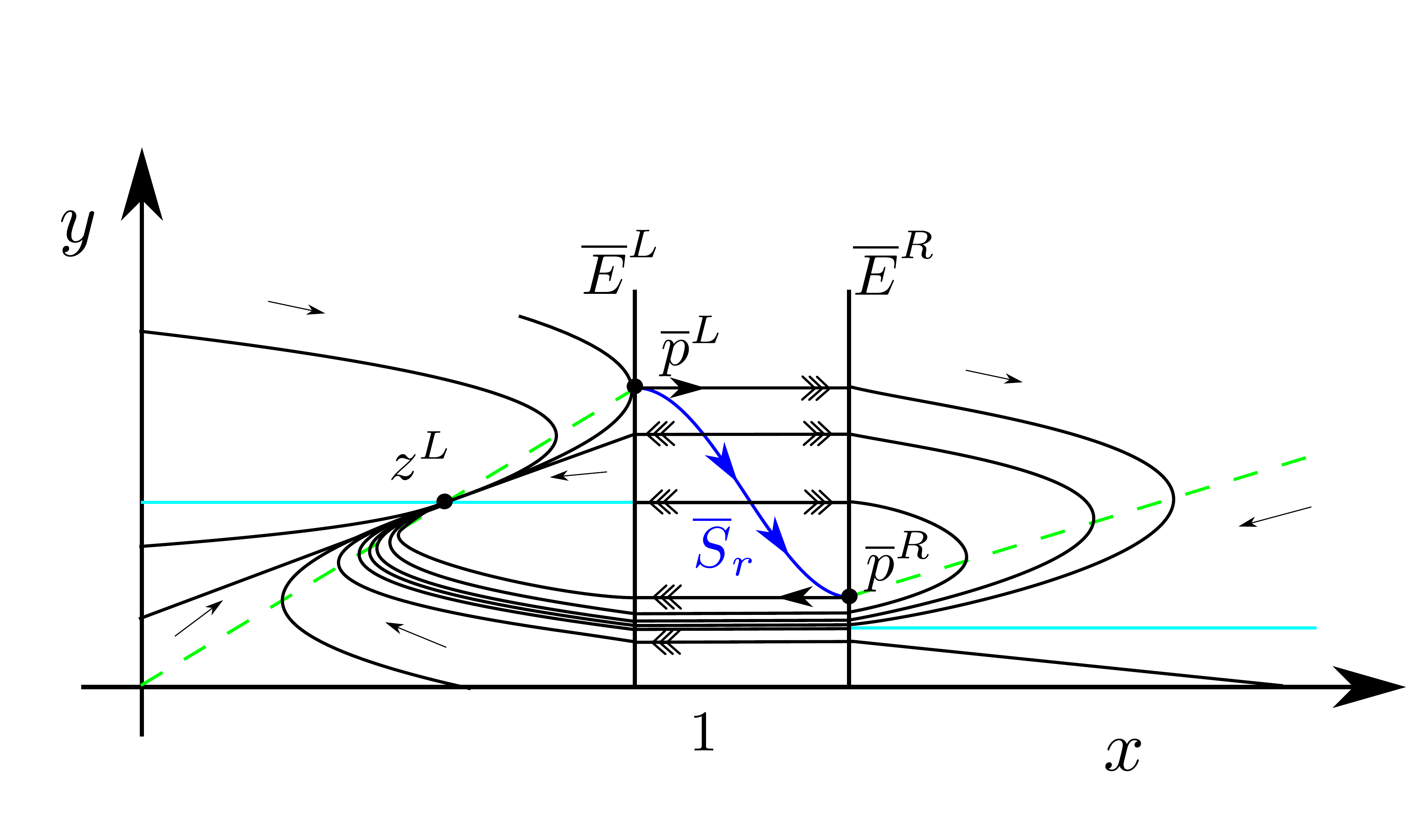}}
\subfigure[$\eta^R<\eta<\eta^L$, $\mu>0$]{\includegraphics[width=.49\textwidth]{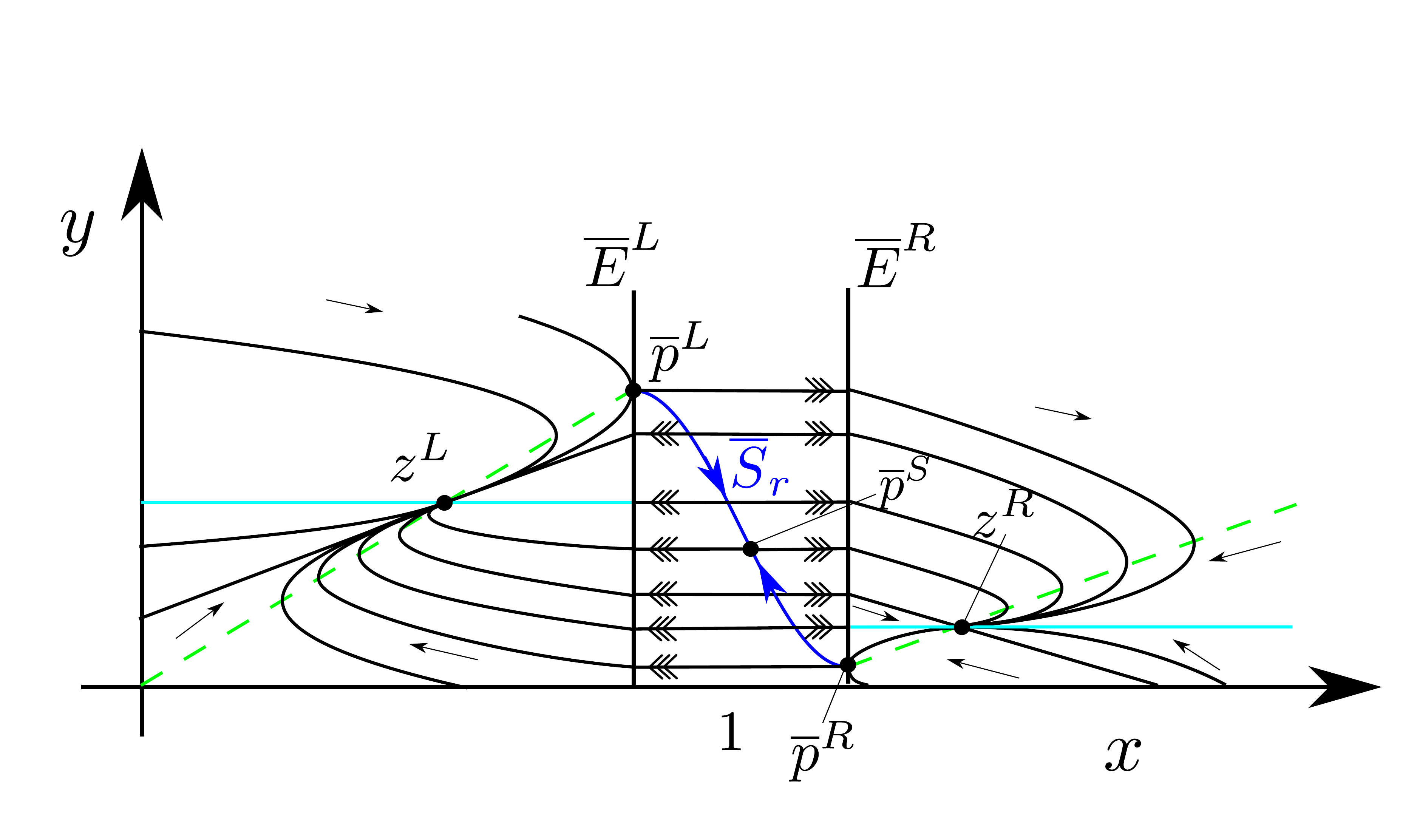}}
\subfigure[$\eta^R<\eta^L<\eta$, $\mu>0$]{\includegraphics[width=.49\textwidth]{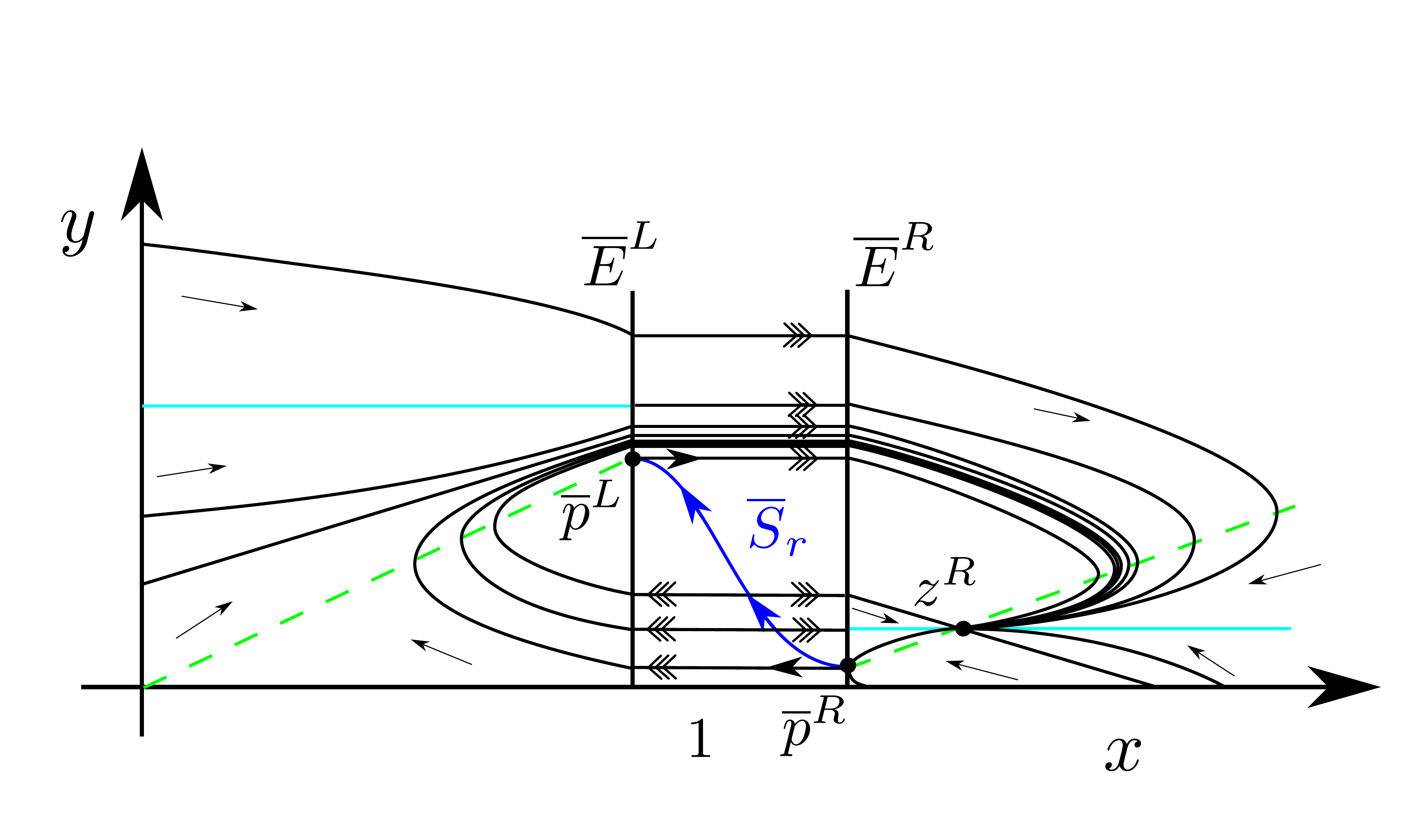}}
\subfigure[$\eta=\eta^R=\eta^L=1$, $\mu=0$]{\includegraphics[width=.49\textwidth]{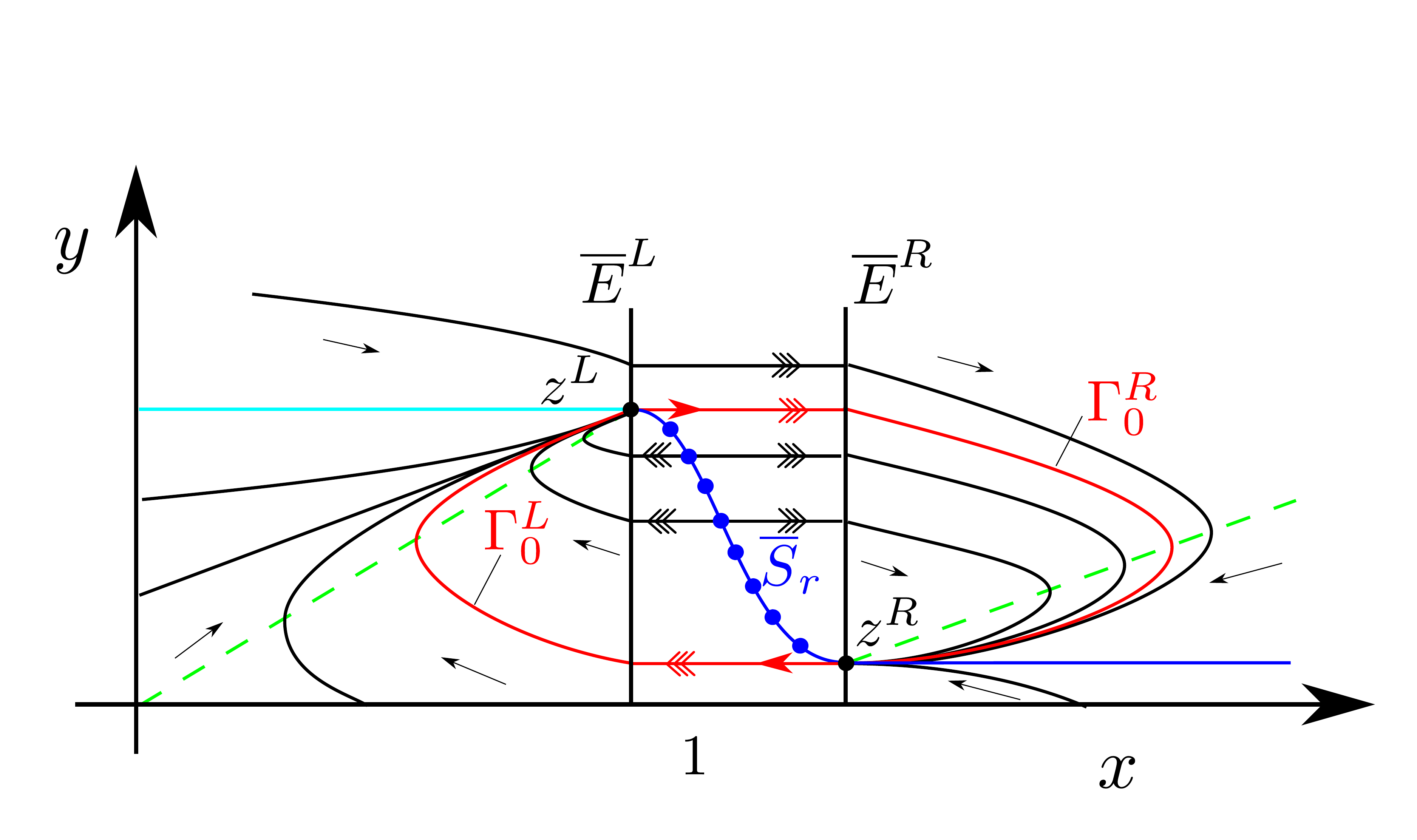}}
\end{center}
 \caption{Dynamics of the blow-up of \eqref{xyFastExt} for different values of $\eta$ and $\mu$. The normally hyperbolic and repelling critical manifold $\overline S_r$ is in blue. The edges $\overline E^{L}$ and $\overline E^R$ are lines of normally hyperbolic equilibria away from the degenerate points at $\overline p^L$ and $\overline p^R$, respectively. In (a), (b) and (c) at least one stable node ($z^L$ or $z^R$) exists within $x\ne 1$. The dashed green curve is the $x$-nullcline whereas the cyan line is the $y$-nullcline. Notice that for $\mu=0$, $z^L$ and $z^R$ coincide with $\overline p^L$ and $\overline p^R$ on $\overline E^{L}$ and $\overline E^R$, respectively, simultaneously for $\eta=\eta^R=\eta^L=1$. This case is shown in (d). Here we also illustrate the singular cycle (in red) which we perturb to an actual limit cycles in \thmref{mainThm}, depending on further conditions on $\eta$ and $\mu$, for $0<\varepsilon\ll 1$.}
\figlab{pwsblowup}
\end{figure}

\begin{figure}[h!] 
\begin{center}
\subfigure[$\mu>0$]{\includegraphics[width=.45\textwidth]{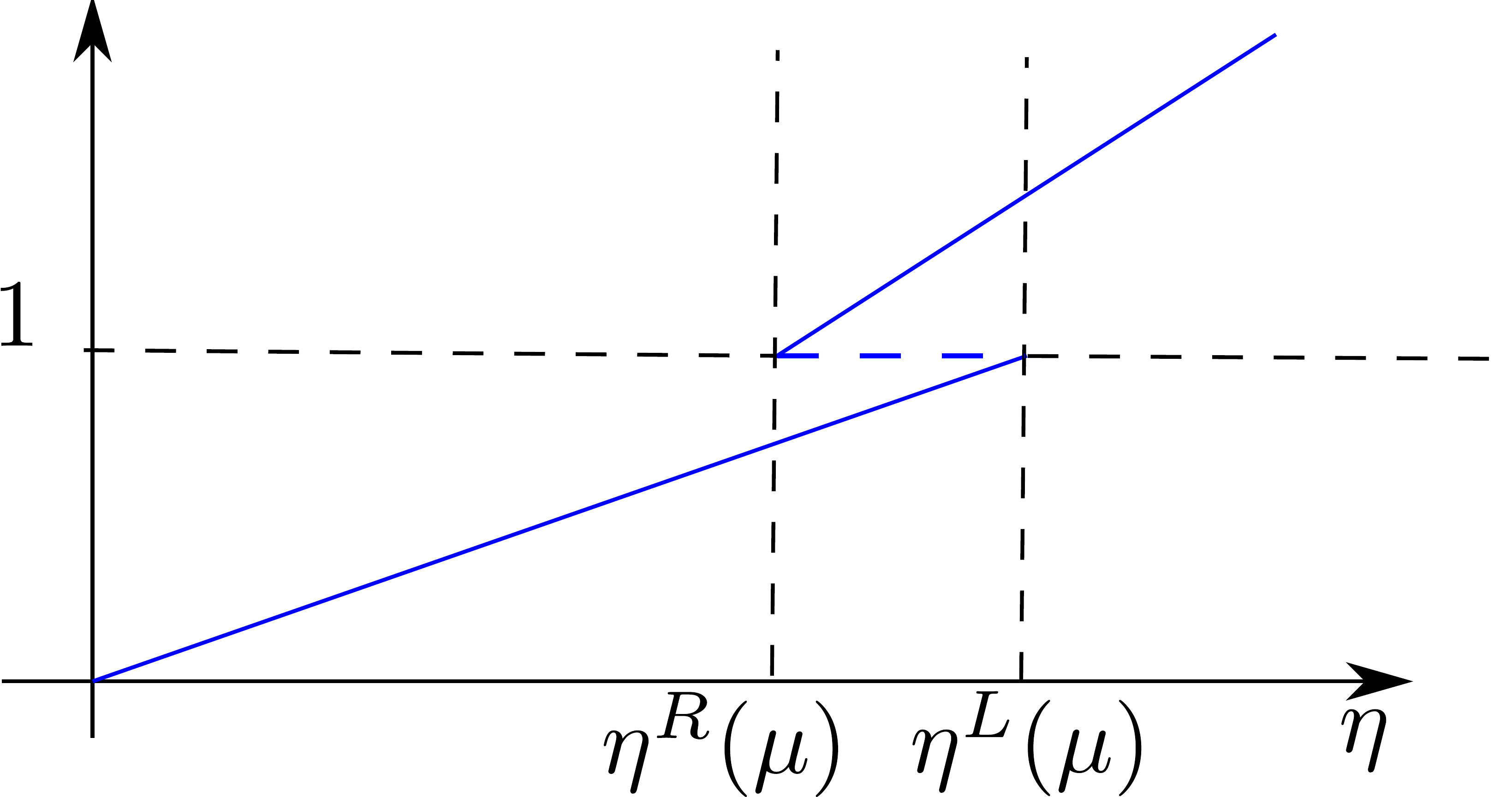}}
\subfigure[$\mu=0$]{\includegraphics[width=.45\textwidth]{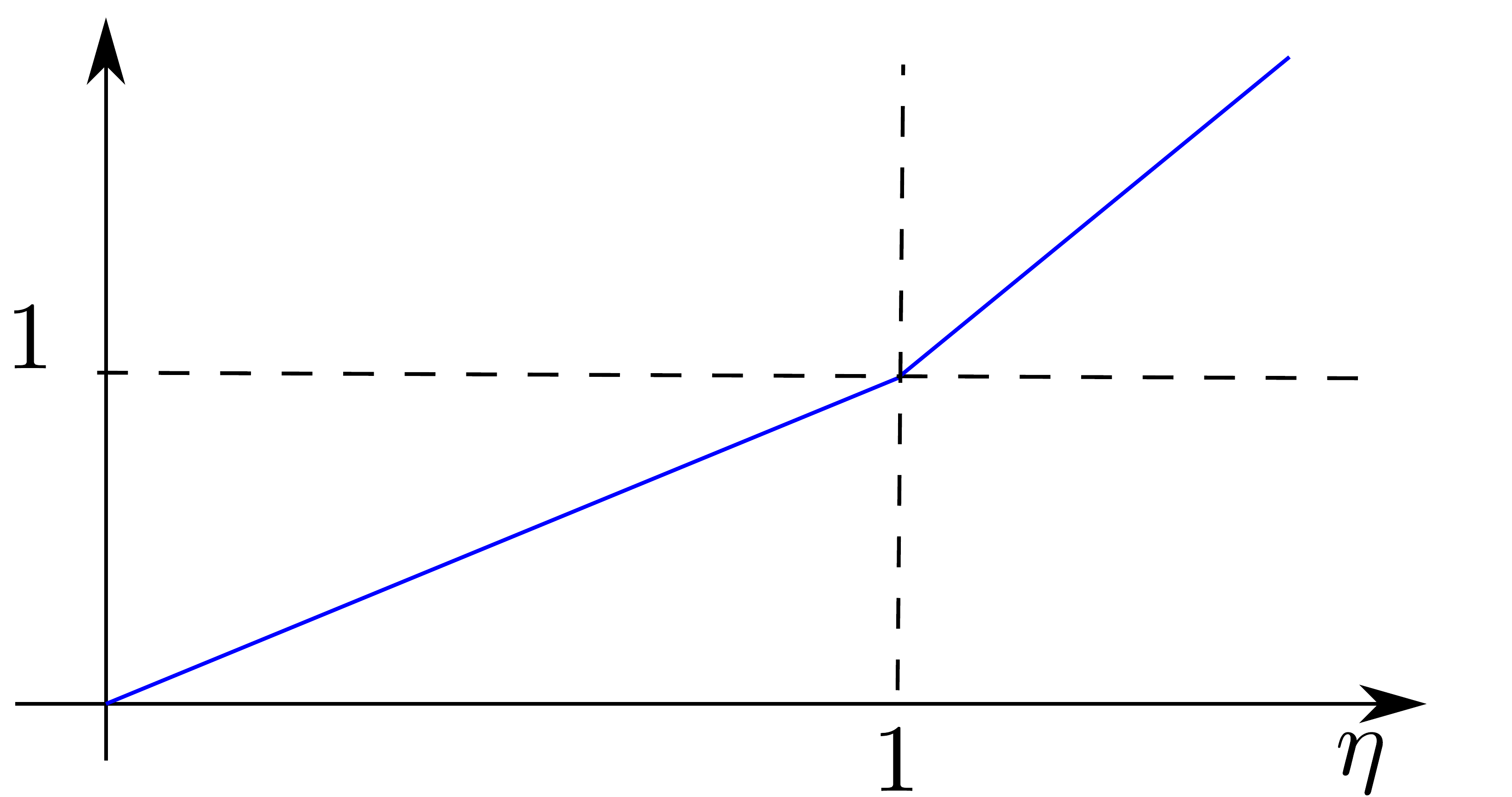}}
\end{center}
 \caption{Bifurcation of equilibria for $\varepsilon=0$ using $\eta$ as the bifurcation parameter. In (a): $\mu>0$. In (b): $\mu=0$.}
\figlab{bif}
\end{figure}
One cannot perturb away from the singular picture in \figref{pwsblowup} directly due to the loss of hyperbolicity at the points $\overline p^{L/R}$. For this we have to apply a further blow-up of these points, including $\eta=\eta^{L/R}$ in the blow-up when the equilibria intersect the cylinder at $x=1$. However, away from $\eta=\eta^{L/R}$ the global picture is not surprising.
\begin{proposition}\proplab{noLimit}
 Suppose $\eta \ne \eta^L(\mu)$, $\eta\ne \eta^R(\mu)$ and $\mu\ge 0$. Then there exists an $\varepsilon_0>0$ sufficiently small such that \eqref{xyFastExt} has no limit cycles for any $\varepsilon \in (0,\varepsilon_0)$ and the omega limit set consists entirely of equilibria. 
\end{proposition}
\begin{proof}
 First, by the assumption on $\eta$, we realise that no equilibria exist in sufficiently small neighborhoods $V^{L/R}$ of the nonhyperbolic points $\overline p^{L/R}$ for all $0<\varepsilon\ll 1$. This excludes local limit cycles near $V^{L/R}$. To exclude global limit cycles we use perturbation arguments; basically the  structure away from $\overline p^{L/R}$ in \figref{pwsblowup} persists for all $0<\varepsilon\ll 1$. Hence we can follow the flow of any point outside $V^{L}$ and $V^R$ towards a stable node, see \figref{pwsblowup} (a),(b) and (c) for all $0<\varepsilon\ll 1$. 
\end{proof}
For $\mu>0$, local limit cycles can exists near $\eta =\eta^{L}$ and $\eta=\eta^R$. But the existence of a stable node on either side of $x=1$ prevents existence of global limit cycles, of the relaxation-type in \figref{muEq0}, for all $0<\varepsilon\ll 1$.  Following this analysis, we can therefore conclude that relaxation oscillations can only occur near $(\eta,\mu)=(1,0)$ as $\varepsilon\rightarrow 0$.

\section{Main result}
We are now ready to present our main result on the existence of relaxation oscillations of \eqref{xy} for $0<\varepsilon \ll 1$. For this, consider first the $x<1$ ($x>1$) system in \eqref{xN1system}$_{\eta=1,\,\mu=0}$ (\eqref{xP1system}$_{\eta=1,\,\mu=0}$, respectively) and define $\Gamma_0^L$ ($\Gamma_0^R$) as the forward orbit of $p^R$ ($p^L$, respectively) for the parameter values 
\begin{align}
\eta = 1,\quad \mu=0.\eqlab{parameter}
\end{align}
Then $\Gamma_0 \equiv \Gamma_0^L\cup \Gamma_0^R$ is a closed curve. See \figref{pws}.  In particular, the trajectories $\Gamma_0^{L/R}$ are asymptotic to the (one-sided) stable nodes $z^{L/R}$, respectively. These points coincide with $p^{L/R}$ on $x=1$ for the parameter values of $\eta$ and $\mu$ in \eqref{parameter}.

We know from \propref{noLimit} and the discussion proceeding it that relaxation oscillations only exist close to $(\eta,\mu)=(1,0)$. It turns out that it is useful to define $\eta_1$ and $\mu_1$ as follows
\begin{align}
 \eta = 1+\varepsilon^{k/(k+1)} \eta_1,\quad 
 \mu =\varepsilon^{k/(k+1)}\mu_1.\eqlab{eta1mu1}
\end{align}
and set
\begin{align}
 \eta_1^L(\mu_1) = \frac{\mu_1}{\alpha},\quad \eta_1^R(\mu_1) = \frac{\mu_1}{\alpha+\beta}.\eqlab{eta1Leta1R}
\end{align}
Here $k\in \mathbb N$ is the order of decay to $0$ and $1$ of the sigmoidal function $\phi$ at $\mp \infty$, respectively, see assumption (A).
Notice that (a): $\varepsilon=0$ in \eqref{eta1mu1} gives \eqref{parameter} and (b): $\eta=\eta^{L/R}(\mu) \Leftrightarrow \eta_1=\eta_1^{L/R}(\mu_1)$, respectively, for any $\varepsilon>0$, recall \eqref{etaN}. 
For simplicity, we also group all of the numbers $k$, $\alpha$, $\beta$, $\phi^L(0)$ and $\phi^R(0)$ into a single parameter vector
\begin{align}
 \gamma = (k,\alpha,\beta,\phi^L(0),\phi^R(0)).\eqlab{gamma}
\end{align}

We then prove the following.
\begin{figure}[h!] 
\begin{center}
{\includegraphics[width=.496\textwidth]{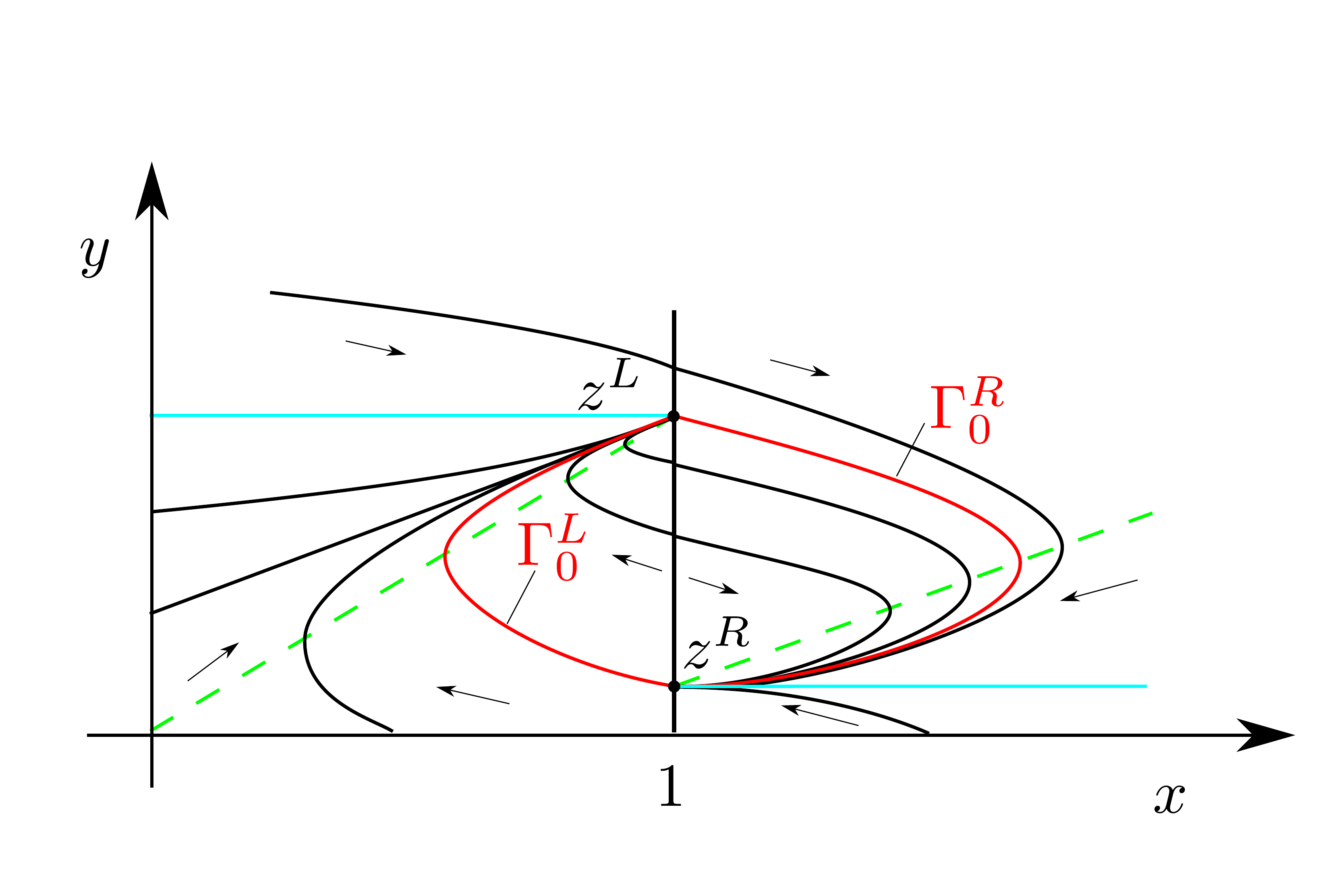}}
\end{center}
 \caption{Phase portrait of the PWL system \eqsref{xN1system}{xP1system} for $\mu=0$ and $\eta=\eta^R=\eta^L=1$.}
\figlab{pws}
\end{figure}
\begin{theorem}\thmlab{mainThm}
Consider \eqref{xy} and suppose \eqref{abasymption} and (A).  Then there exist two numbers 
\begin{align}
\eta^L_{\text{Het},0}<0,\quad \eta^R_{\text{Het},0}>0,\eqlab{numbers}
\end{align} 
 such that the following holds for any $\gamma$.
\begin{enumerate}[label=(\alph*)]
%
\item \label{item3} Let 
\begin{align*}
  \eta^L_{\text{Het}}(\mu_1) &:= \eta^L_{\text{Het},0}+\eta^L_1(\mu_1),\\
  \eta^R_{\text{Het}}(\mu_1) &:=\eta^R_{\text{Het},0}+\eta^R_1(\mu_1).
  \end{align*}
Then there exists a unique $\mu_{1*}>0$ such that 
\begin{enumerate}[label=(\alph{enumi}.\arabic*)]
\item  \label{item3a} $\eta^L_{\text{Het}}(\mu) < \eta^R_{\text{Het}}(\mu)$ for all $\mu \in [0,\mu_{1*})$.
\item $\eta^L_{\text{Het}}(\mu_{1*})=\eta^R_{\text{Het}}(\mu_{1*})$.
\item \label{item3c} $\eta^L_{\text{Het}}(\mu) > \eta^R_{\text{Het}}(\mu)$ for all $\mu >\mu_{1*}$.
\end{enumerate}
\end{enumerate}
Furthermore,
\begin{enumerate}[resume*]
\item \label{item4} Fix $\mu_1\in [0,\mu_{1*})$ so that $\eta^L_{\text{Het}}(\mu_1)<\eta^R_{\text{Het}}(\mu_1)$ by \ref{item3a} and let $\eta_1\in (\eta^L_{\text{Het}}(\mu_1),\eta^R_{\text{Het}}(\mu_1))$. Then there exists an $\varepsilon_0$ such that for all $\varepsilon\in (0,\varepsilon_0)$, there exists an attracting limit cycle $\Gamma_{\varepsilon}^{\text{Relax}}$ of the system \eqref{xy} with $(\eta,\mu)$ as in \eqref{eta1mu1}. Here $\Gamma_{\varepsilon}^{\text{Relax}}\rightarrow \Gamma_0$ in Hausdorff distance as $\varepsilon\rightarrow 0$. On the other hand, if $\eta_1\notin [\eta^L_{\text{Het}}(\mu_1),\eta^R_{\text{Het}}(\mu_1)]$ then there exist (i) an $\varepsilon_0>0$ and (ii) a constant $c>0$ such that for all $\varepsilon\in (0,\varepsilon_0)$ and $(\eta,\mu)$ as in \eqref{eta1mu1}, there exist no limit cycles for the system \eqref{xy}  closer to $\Gamma_0$ than $c$ in Hausdorff distance.
 \item \label{item5} Fix any $\mu_1>\mu_{1*}$ so that $\eta^L_{\text{Het}}(\mu_1)>\eta^R_{\text{Het}}(\mu_1)$ by \ref{item3c}. Then there exist (i) an $\varepsilon_0>0$, (ii) a constant $c>0$ and (iii) a neighborhood $I$ of $\eta=1$ such that for all $\varepsilon\in (0,\varepsilon_0)$, any $\eta \in I$ and $\mu$ as in \eqref{eta1mu1}, there exist no limit cycles for the system \eqref{xy} closer to $\Gamma_0$ than $c$ in Hausdorff distance.
\end{enumerate}

%
%

\end{theorem}
To prove this theorem we will blow-up the points $\overline p^{L/R}$ to spheres.
The numbers $\eta^{L/R}_{\text{Het},0}$ will then appear as heteroclinic bifurcation values of $\eta_1-\eta^{L/R}_1(\mu_1)$ for two separate connection problems (that only depend on $\eta_1$ and $\mu_1$ through $\eta_1-\eta^{L/R}_1(\mu_1)$) on these spheres, respectively. In \figref{etaLRHet}, we have computed these heteroclinics numerically (using a simple shooting method) for $\alpha=0.5,\beta=1$ and $\phi^{L/R}(0)=\frac{1}{\pi}$. The shaded region where $\eta^L_{\text{Het}}(\mu_1)<\eta^R_{\text{Het}}(\mu_1)$ corresponds to the region where relaxation oscillations exists.

It follows from \eqref{numbers} that for $\mu_1=0$ we have $\eta^L_{Het}(0)<0<\eta^R_{Het}(0)$.  Therefore relaxation oscillations always exist when $\mu=0$, in agreement with \cite{Tyson2003} and \figref{muEq0}. Finally, we note that since the nullclines lie very close near $p^{L/R}$, $x(t)$ and $y(t)$ slow down as they come close to these points. This gives rise to the relaxation-type oscillations seen in e.g. \figref{muEq0}(c) and (d).  

In contrast, for the system in \figref{muEq008} where $\mu=0.08$, $\varepsilon=0.006$ we find that $\mu_1=1.0$. Consequently, by \figref{etaLRHet}, $\mu_1>\mu_{1*}\approx 0.8$ and the nonexistence of relaxation oscillations in \figref{muEq008} is therefore in agreement with \thmref{mainThm} \ref{item4}. 
\begin{remark}
Near the lines $\eta=\eta^{L/R}_{Het}(\mu)$ for $\mu<\mu_*$ a canard-like explosion occurs where the Hopf cycles grow to the relaxation oscillations described in \thmref{mainThm} within a tiny parameter regime by following the repelling manifold $S_r$. For $\mu>\mu_*$ incomplete canard explosions occur where the Hopf cycles terminate in homoclinics to $p_S$. The description of this is part of future work. 
\end{remark}

\begin{figure}[h!] 
\begin{center}
{\includegraphics[width=.7\textwidth]{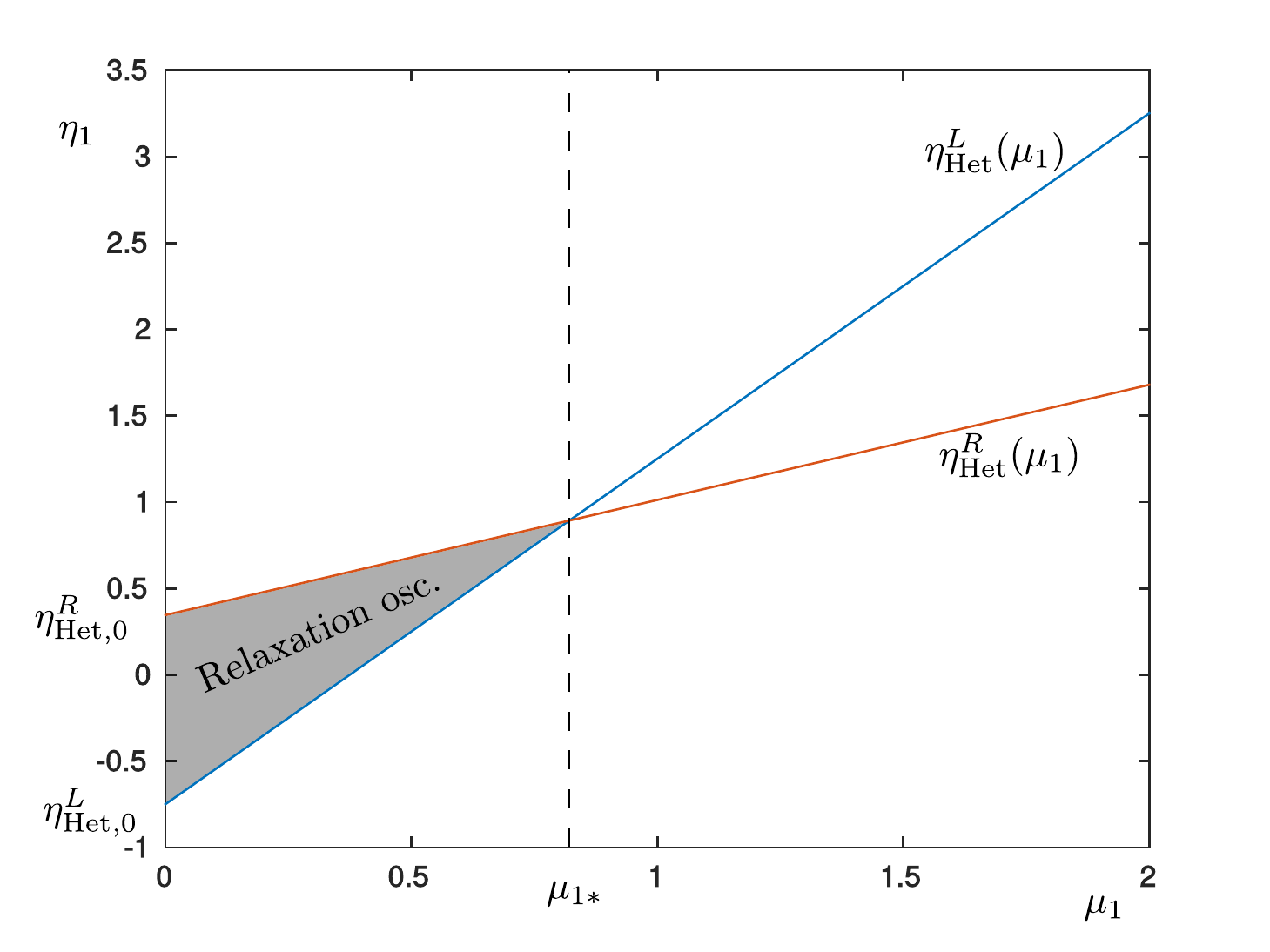}}
\end{center}
\caption{Computed functions $\eta_{Het}^{L/R}(\mu_1)$ for $\alpha=0.5,\beta=1$ and $\phi^{L/R}(0)=\frac{1}{\pi}$. By our main theorem, relaxation oscillations only exists within the shaded region for $0<\varepsilon\ll 1$.}\figlab{etaLRHet}
\end{figure}

\subsection{Outline}
The remainder of the paper is mainly devoted to the proof of \thmref{mainThm}. In the following \secref{blowup} we describe several blow-ups used in our analysis to prove \thmref{mainThm}. In \secref{scalingParameters}, for example, we use the scaling \eqref{eta1mu1} of parameters used to zoom $(\eta,\mu)$ in on the degenerate parameter value $(\eta,\mu)=(1,0)$ in \eqref{parameter}. Within this scaled system, using now $\eta_1$ and $\mu_1$ as the main bifurcation parameters,  we then redo the blow-up of $x=1$, $y\in J$, $\varepsilon=0$, see \secref{blowupXEq1New}. As above, we describe this blow-up using local directional charts $\bar x=-1$, $\bar \varepsilon=1$ and $\bar x=1$. But now in \secref{blowupPLR}, we then further blow-up the points $p^{L}$ and $p^R$ in the charts $\bar x=\pm 1$ to spheres $\overline S^L$ and $\overline S^R$. In \secref{sphere}, we describe the dynamics on these sphere in two propositions, see \propref{prop1} and \propref{prop2}, that we prove in \appref{proofprop1}. There we demonstrate that the numbers $\eta^L_{\text{Het}}$ and  $\eta^R_{\text{Het}}$ in \thmref{mainThm} are values of the scaled $\eta$ that give rise to heteroclinic connections on the spheres $\overline S^L$ and $\overline S^R$, respectively. Using this blow-up analysis we complete the proof of the theorem in \secref{proof}. In \secref{discussion}, we present a generalization of \thmref{mainThm} to a more general scenario of a single boundary node bifurcation. 

\section{Blowup}\seclab{blowup}

\subsection{Blowup of parameters}\seclab{scalingParameters}
Following \eqref{eta1mu1} we now scale the parameters $(\epsilon,\mu,\eta)$ by setting 
\begin{align}
(\varepsilon,\mu,\eta) = (\sigma_1^{k+1},\sigma_1^k \mu_1,1+\sigma_1^k\eta_1),\quad \sigma_1\in I_1,\,(\mu_1,\eta_1)\in V_1,\eqlab{blowupParameters1}
\end{align}
Here $I_1=[0,\nu]$ with $\nu$ sufficiently small and $V_1$ is a sufficiently large but fixed compact neighborhood of $(0,0)$ in $\mathbb R^2$. Notice also that $\sigma_1 = \varepsilon^{1/(k+1)}$ and it is therefore possible to write \eqref{blowupParameters1} as $\mu = \varepsilon^{k/(k+1)}\mu_1$ and $\eta = 1+\varepsilon^{k/(k+1)}\eta_1$, in agreement with \eqref{eta1mu1}.  Inserting \eqref{blowupParameters1} into \eqref{xyFastExt} gives the following system
 \begin{align}
x' &=\sigma_1^{k+1}((\alpha + \beta \phi(\sigma_1^{-(k+1)}(x-1)))y-x),\eqlab{xyFastExtDelta}\\
 y' &=\sigma_1^{k+1}(1+\sigma_1^k \eta_1 -(\sigma_1^k \mu_1+\alpha+\beta \phi(\sigma_1^{-(k+1)}(x-1)))y),\nonumber\\
 \sigma_1'&=0,\nonumber
\end{align}
that we shall study henceforth. Let $X$ denote the corresponding vector-field. 
It is natural to think of \eqref{blowupParameters1} as a chart ($\bar \epsilon=1$) associated with the blow-up (in parameter space)
\begin{align}
 (\varepsilon,\mu,\eta) = (\sigma^{k+1}\bar \epsilon,\sigma^k\bar \mu,\sigma^k\bar \eta),\quad \sigma\ge 0,\,(\bar \epsilon,\bar \mu,\bar \eta)\in S^2.\eqlab{blowupParameters}
\end{align}
We can study a small neighborhood of $(\varepsilon,\mu,\eta)=0$ by studying each $(\sigma,(\bar \epsilon,\bar \mu,\bar \eta))$ with $\sigma\ge 0$ small. 
\eqref{blowupParameters1} then describes parameter values on $S^2$ with $\bar \epsilon>0$. To describe points near the equator $\bar \epsilon=0$ one could use the chart $\bar \mu=1$, which would correspond to
\begin{align}
(\varepsilon,\mu,\eta) = (\sigma_2^{k+1}\epsilon_2,\sigma_2^k,\sigma_2^k\eta_2),\quad \sigma_2\in I_2,\,(\epsilon_2,\eta_2)\in U_2.\nonumber
\end{align}
But following the previous analysis (recall also \propref{noLimit}), all oscillations that we wish to describe will all be visible in \eqref{blowupParameters1} and we shall therefore focus on this chart henceforth. We therefore also drop the subscripts in \eqref{blowupParameters1}.

\begin{remark}
 The weights/exponents of $\sigma$ in \eqref{blowupParameters1}, are found by the blowup approach, the defining condition being that the transformed vector-field, working in directional charts, has a power of $\sigma$ (or more precisely: a power of $\rho$ below in \secref{blowupPLR}, see also \eqref{widehatXi}), as a common factor. The desingularization, obtained by division of the vector-field by this common factor, gives rise to improved hyperbolicity properties. This is central to our application of hyperbolic methods from dynamical systems theory \cite{perko2001a,wiggins2003a}. 
\end{remark}

\subsection{Blowup of $x=1$, $y\in J$, $\sigma=0$}\seclab{blowupXEq1New}
Now, $0\le \sigma\ll 1$ is our new small parameter for \eqref{xyFastExtDelta}. Similar to system \eqref{xyFastExt} for $\varepsilon=0$, system \eqref{xyFastExtDelta} is singular along $x=1$ for $\sigma=0$. In line with \eqref{initialBlowup}, we therefore also blow-up $x=1$, $y\in J$, $\sigma=0$ as 
 \begin{align}
 x= 1+r^{k+1}\bar x,\quad y=\bar y,\quad  \sigma = r \bar \delta,\quad r\ge 0,\,\bar y \in J,\,(\bar x,\bar \delta)\in S^1. \eqlab{initialBlowup2}
\end{align}
The transformation $\Phi:\,(r,\bar y,(\bar x,\bar \delta))\mapsto (x,y,z,\sigma)$, defined by \eqref{initialBlowup2}, gives rise to a vector-field $\overline X=\Phi^{*}(X)$ on $r\ge 0$, $\bar y\in J$, $(\bar x,\bar \delta)\in S^1$ by pull-back. 

As with \eqref{initialBlowup}, we describe \eqref{initialBlowup2} and the vector-field $\overline X$ using three different charts
\begin{align}
 \bar x=-1:\quad &x = 1-r_1^{k+1},&y&=y_1,&\sigma&=r_1 \delta_1,&r_1&\ge 0,&\delta_1&\ge 0,\eqlab{barXN1New}\\
 \bar \delta=1:\quad &x = 1+r_2^{k+1}x_2,&y&=y_2,&\sigma& =r_2,& x_2&\in \mathbb R,&r_2&\ge 0,\eqlab{baru1}\\
 \bar x =1:\quad &x = 1+r_3^{k+1},&y&=y_2,&\sigma&=r_3 \delta_3,&r_3&\ge 0,&\delta_3&\ge 0,\eqlab{barX1New}
\end{align}
As above, we will consider $(x_2,r_2)\in I_2\times [0,\nu]$ with $I_2$ large but fixed in the chart $\bar \delta=1$. We then fix small compact sets $U_1$ and $U_3$ accordingly so that in charts $\bar x=\pm 1$, we have $(r_1,\delta_1)\in U_1$ and $(r_3,\delta_3)\in U_3$ and the charts $\bar x=-1,\,\bar \epsilon=1,\,\bar x=1$ cover a full neighborhood of $x=1,\,\sigma=0$. We let $y_i\in J$, $i=1,2,3$, with $J$ a sufficiently large interval throughout. We keep the subscripts on $y_i$, $i=1,2,3$ in the following, because the $y$-variables are treated slightly different in the charts $\bar x=\pm 1$.

The local versions of $\overline X$ in $\bar x=\mp 1$ will have $\delta_1^{k+1}$ and $\delta_3^{k+1}$, respectively, as common factors. We will therefore again divide the right hand sides, that appear by substituting \eqsref{barXN1New}{barX1New} into \eqref{xyFastExtDelta}, by these factors, obtaining desingularized local vector-fields. In the following, we will re-use many of the symbols introduced in \secref{singular0} for related objects appearing from the blow-up of \eqref{xyFastExtDelta}.

\subsection{Chart $\bar x=-1$}
By inserting \eqref{barXN1New} into \eqref{xyFastExtDelta}, we obtain the following equations 
\begin{align*}
 \dot r_1 &=-\frac{1}{k+1}r_1\delta_1^{k+1}\left(r_1^{k+1}-F_1(\delta_1^{k+1},y_1)\right),\\
 \dot y_1&=r_1^{k+1}\delta_1^{k+1} \left(F_1(\delta_1^{k+1},y_1)+r_1^k\delta_1^k \left(\eta-\mu y_1\right)\right),\\
 \dot \delta_1 &=\frac{1}{k+1} \delta_1^{k+2} \left(r_1^{k+1}-F_1(\delta_1^{k+1},y_1)\right),
 \end{align*}
 where
 \begin{align*}
  F_1(\delta_1^{k+1},y_1) = 1-(\alpha+\beta \delta_1^{k(k+1)}\phi^L(\delta_1^{k+1}))y_1.
 \end{align*}
As promised, we apply a desingularization through the division of the right side by the common factor $\delta_1^{k+1}$ and shall henceforth study the following system:
\begin{align}
 \dot r_1 &=-\frac{1}{k+1}r_1\left(r_1^{k+1}-F_1(\delta_1^{k+1},y_1)\right),\eqlab{X1}\\
 \dot y_1&=r_1^{k+1} \left(F_1(\delta_1^{k+1},y_1)+r_1^k\delta_1^k \left(\eta-\mu y_1\right)\right),\nonumber\\
 \dot \delta_1 &=\frac{1}{k+1} \delta_1\left(r_1^{k+1}-F_1(\delta_1^{k+1},y_1)\right).\nonumber
\end{align}
As above, the ``edge'' 
\begin{align*}
 E^L_1 = \{(r_1,y_1,\delta_1)\vert r_1=\delta_1 =0,\,y_1\in J\},
\end{align*}
of the cylinder, is a set of normally hyperbolic equilibria for \eqref{X1} for $y_1\ne y^L$. The point 
\begin{align*}
 p_1^L:\,(r_1,y_1,\delta_1) = (0,y^L,0),
\end{align*}
is fully nonhyperbolic for \eqref{X1}. 
\subsection{Chart $\bar \delta=1$}
Similarly, in chart \eqref{baru1} we have
\begin{align}
 \dot x_2 &=-r_2^{k+1}x_2-F_2(x_2,y_2),\eqlab{X2}\\
 \dot y_2 &=r_2^{k+1}\left(F_2(x_2,y_2)+r_2^k(\eta-\mu y_2)\right),\nonumber
\end{align}
where $r_2\ge 0$ is a parameter: $\dot r_2=0$, and
\begin{align*}
 F_2(x_2,y_2)=1-(\alpha+\beta \phi(x_2))y_2.
\end{align*}
For $r_2=0$ the set
\begin{align*}
 C_{r,2}=\{(x_2,y_2)\in [0,\infty)\times J\vert y_2 = (\alpha+\beta \phi(x_2))^{-1}\},
\end{align*}
is a normally hyperbolic and repelling critical manifold. We have the following. 
\begin{lemma}
Any compact submanifold (with boundary) $S_{r,2}$ of $C_{r,2}$ perturbs into a repelling slow manifold
\begin{align*}
 S_{r,r_2,2}=\{(x_2,y_2)\vert x_2 = \phi^{-1}\left(-\beta^{-1}(\alpha-y_2^{-1})\right)\left(1 + r_2^{k+1}m_2(y_2,r_2,\eta,\mu)\right),\quad y_2\in I_2\},
\end{align*}
with $I_2\subset(y^R,y^L)$ the sufficiently large closed interval in \secref{blowupPLR}, for $r_2\le r_{20}$ sufficiently small. Here $m_2$ is a scalar-valued smooth function. 

$S_{r,r_2,2}$ carries the reduced flow 
\begin{align}
 \dot y_2 &=r_2^{2k+1}\left(\eta-\mu y_2+r_2 n_2(y_2,r_2,\eta,\mu)\right),\eqlab{slowFlow}
\end{align}
for some smooth function $n_2$. Hence, there exists a saddle for $\mu>0$, $\eta \in (\eta^R,\eta^L)$ where 
\begin{align*}
\eta^R = \frac{\mu}{\alpha+\beta},\quad \eta^L = \frac{\mu}{\alpha},
\end{align*}
recall also (with subscripts) \eqref{eta1Leta1R}, 
at $p_2^S\in S_{r,r_2,2}$ with $y_2$-coordinate close to $\mu^{-1}\eta$ for any $0<r_2\le r_{20}$. 

\end{lemma}
\begin{proof}
 Simple calculation.
\end{proof}

With respect the (new) slow time $\tau = r_2^{2k+1} t$, where $t$ is the fast time in \eqref{slowFlow}, \eqref{slowFlow} becomes 
\begin{align}
 y_2' = \eta-\mu y_2,\eqlab{slowFlow2}
\end{align}
for $r_2=0$. Hence, we will refer to the flow of \eqref{slowFlow2} as the slow flow on $C_{r,2}$. 

\subsection{Chart $\bar x=1$}
Finally, in chart \eqref{barX1New}
\begin{align}
\dot r_3 &=-\frac{1}{1+k}r_3 \left(r_3^{k+1} +F_3(\delta_3^{k+1},y_3)\right),\eqlab{X3}\\
\dot y_3 &=r_3^{k+1} \left(F_3(\delta_3^{k+1},y_3) +r_3^k\delta_3^k (\eta-\mu y_3)\right),\nonumber\\
\dot \delta_3 &=\frac{1}{1+k}\delta_3 \left(r_3^{k+1} +F_3(\delta_3^{k+1},y_3)\right),\nonumber
\end{align}
after division by the common factor $\delta_3^{k+1}$ of the right hand side. Here
\begin{align*}
 F_3(\delta_3^{k+1},y_3) = 1-(\alpha+\beta(1-\delta_3^{k(k+1)} \phi^R(\delta_3^{k+1})))y_3.
\end{align*}
As above, the ``edge'' 
\begin{align*}
 E^R_3 = \{(r_3,y_3,\delta_3)\vert r_3=\delta_3 =0,\,y_3\in J\},
\end{align*}
of the cylinder is a set of normally hyperbolic equilibria for \eqref{X3} for $y_3\ne y^R$. The point 
\begin{align*}
 p_3^R:\,(r_3,y_3,\delta_3) = (0,y^R,0),
\end{align*}
is fully nonhyperbolic for \eqref{X1}. 
We combine the results in  \figref{pwsblowupNew}. In comparison with \figref{pwsblowup}, we now see closed cycles for all values of the (scaled) $\eta$ and $\mu$. However, we still cannot perturb away from the singular limit due to the degeneracies at $p^L$ and $p^R$. 

\begin{figure}[h!] 
\begin{center}
\subfigure[$\eta<\eta^R<\eta^L$, $\mu>0$]{\includegraphics[width=.49\textwidth]{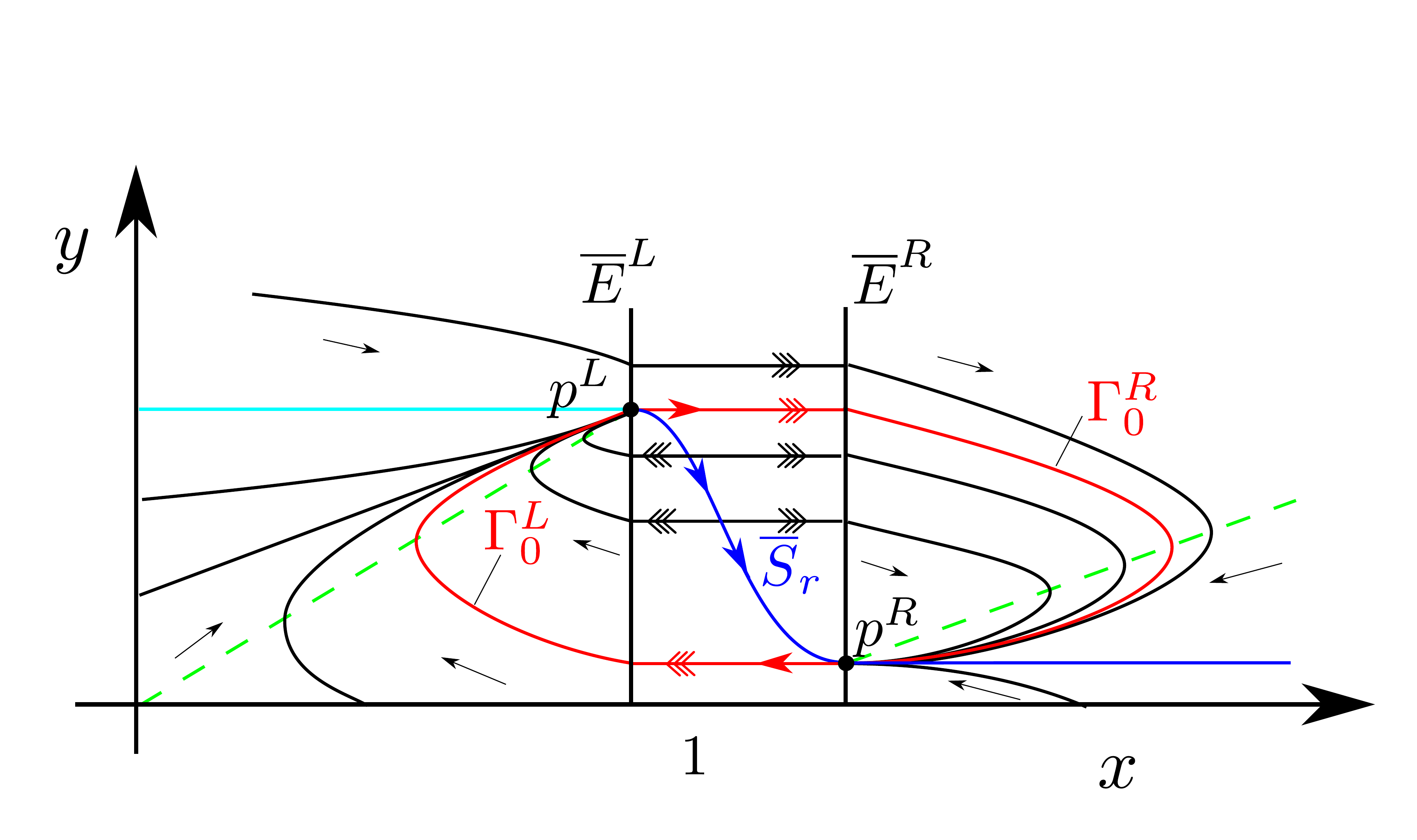}}
\subfigure[$\eta^R<\eta<\eta^L$, $\mu>0$]{\includegraphics[width=.49\textwidth]{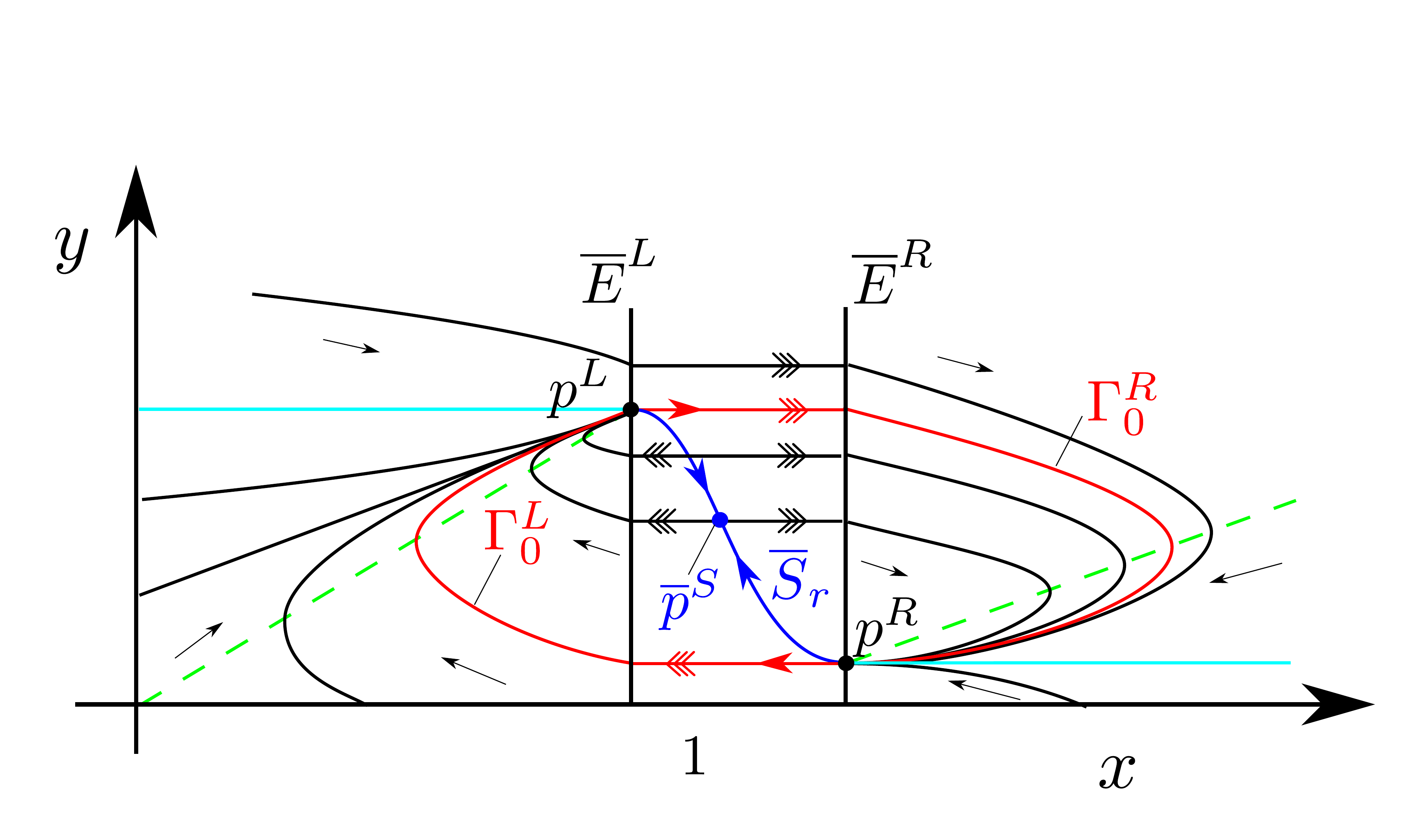}}
\subfigure[$\eta^R<\eta^L<\eta$, $\mu>0$]{\includegraphics[width=.49\textwidth]{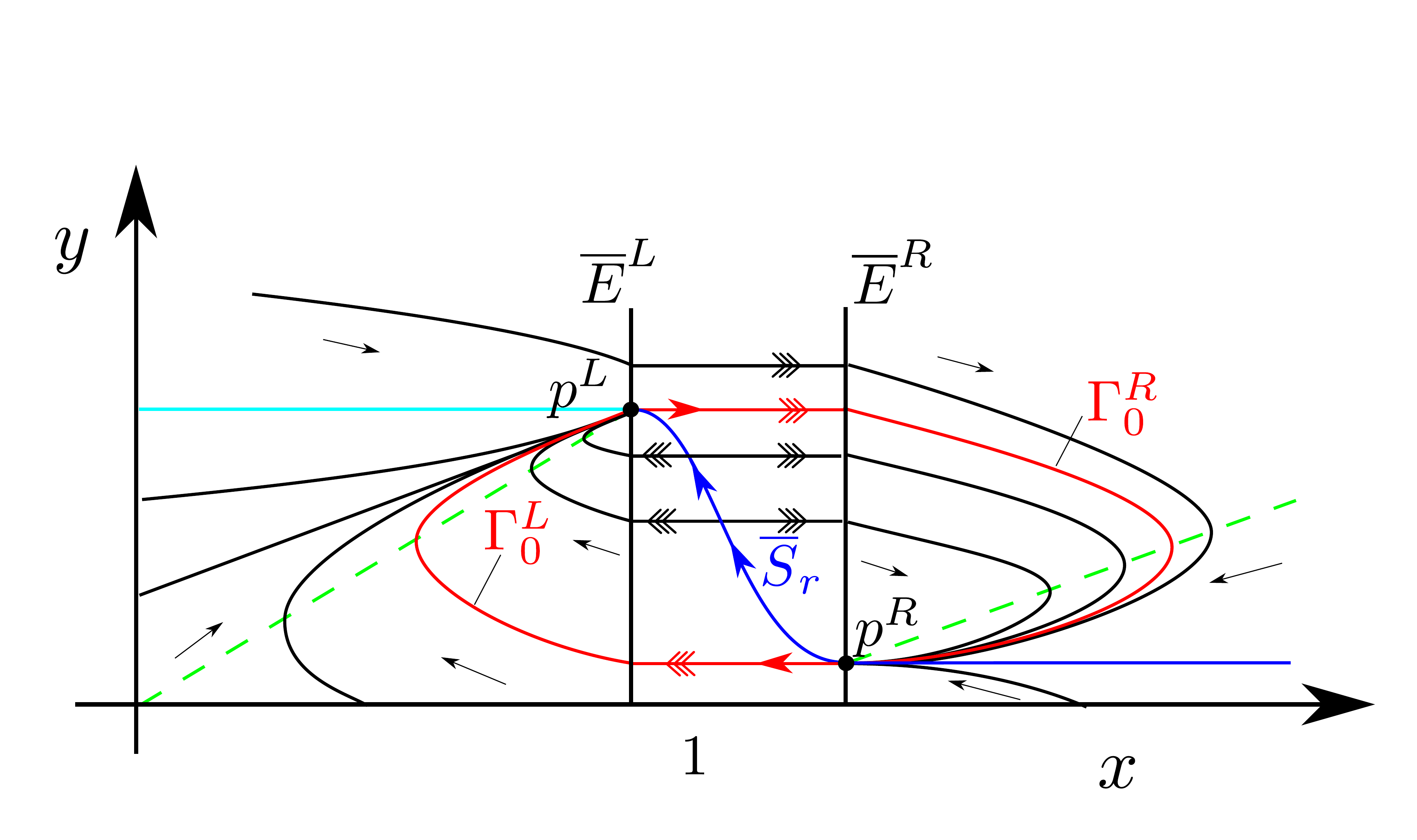}}
\subfigure[$\eta=\eta^R=\eta^L=0$, $\mu=0$]{\includegraphics[width=.49\textwidth]{./pwsblowup4New.pdf}}
\end{center}
 \caption{Blowup dynamics for different values of $\eta$ and $\mu\ge 0$. Here $(\eta,\mu)$ is $(\eta_1,\mu_1)$ in \eqref{blowupParameters1} with subscripts dropped. (a): For $\eta<\eta^R$ and $\mu>0$, $y$ decreases uniformly for the slow flow on $\overline C_{r}$ defined by \eqref{slowFlow2}. (b): For $\eta\in (\eta^R,\eta^L)$ and $\mu>0$, there exists a saddle $\overline p^S$ on $C_{r}$, $\overline p^S$ being attracting for the slow flow on the repelling critical manifold $\overline C_{r}$. (c): For $\eta>\eta^L$ and $\mu>0$, $y$ increases uniformly for the slow flow on $\overline C_r$. (d): For $\eta=0$ and $\mu=0$, the slow flow is the constant flow and hence $\overline C_r$ is a set of equilibria for \eqref{slowFlow2}. The case where $\eta=\eta^{L/R}$ (not shown) is similar to (a) and (c) but now the slow flow is forward complete, with any point being forward asymptotic to $p^{L/R}$, respectively. We see a singular cycle (in red) for all parameter values, but we cannot perturb away from it due to the degeneracies at $p^L$ and $p^R$.}
\figlab{pwsblowupNew}
\end{figure}
\subsection{Blowup of $p^{L/R}=(1,y^{L/R},0)$}\seclab{blowupPLR}
Let $X_{1,3}$ denote the desingularized vector-field $\overline X$ in the charts $\bar x=-1,\,\bar x=1$
as given in \eqsref{X1}{X3}, respectively. We then proceed to blow-up the nonhyperbolic points $p^{L/R}$: $x=1$, $y=y^{L/R}$, $\sigma=0$ in the directional charts $\bar x=\mp 1$, \eqsref{barXN1New}{barX1New}, respectively, by setting
\begin{align}
r_1 = \rho_1^k \bar r_1,\quad y_1=y^L+\rho_1^{k(k+1)} \bar y_1,\quad \delta_1 = \rho_1 \bar \delta_1,\quad \rho_1\ge 0,\,(\bar r_1,\bar \delta_1,\bar y_1)\in S^2,\eqlab{Psi1}
\end{align}
and
\begin{align}
r_3 = \rho_3^k \bar r_3,\quad y_3=y^R+\rho_3^{k(k+1)} \bar y_3,\quad \delta_3 = \rho_3 \bar \delta_3,\quad \rho_3\ge 0,\,(\bar r_3,\bar \delta_3,\bar y)\in S^2,\eqlab{Psi3}
\end{align}
In this way, $p^{L/R}=(1,y^{L/R},0)$ is blown up to two quarter-spheres 
\begin{align*}
\overline{S}^L &= \{(\bar r_1,\bar \delta_1,\bar y_1)\in S^2\vert \bar r_1\ge 0,\,\bar \delta_1\ge 0 \},\\
\overline{S}^R &= \{(\bar r_3,\bar \delta_3,\bar y_3)\in S^2\vert \bar r_3\ge 0,\,\bar \delta_3\ge 0 \}.
\end{align*}
See \figref{x1Eq1BlowupSphere}. 
\begin{figure}[h!] 
\begin{center}
{\includegraphics[width=.7\textwidth]{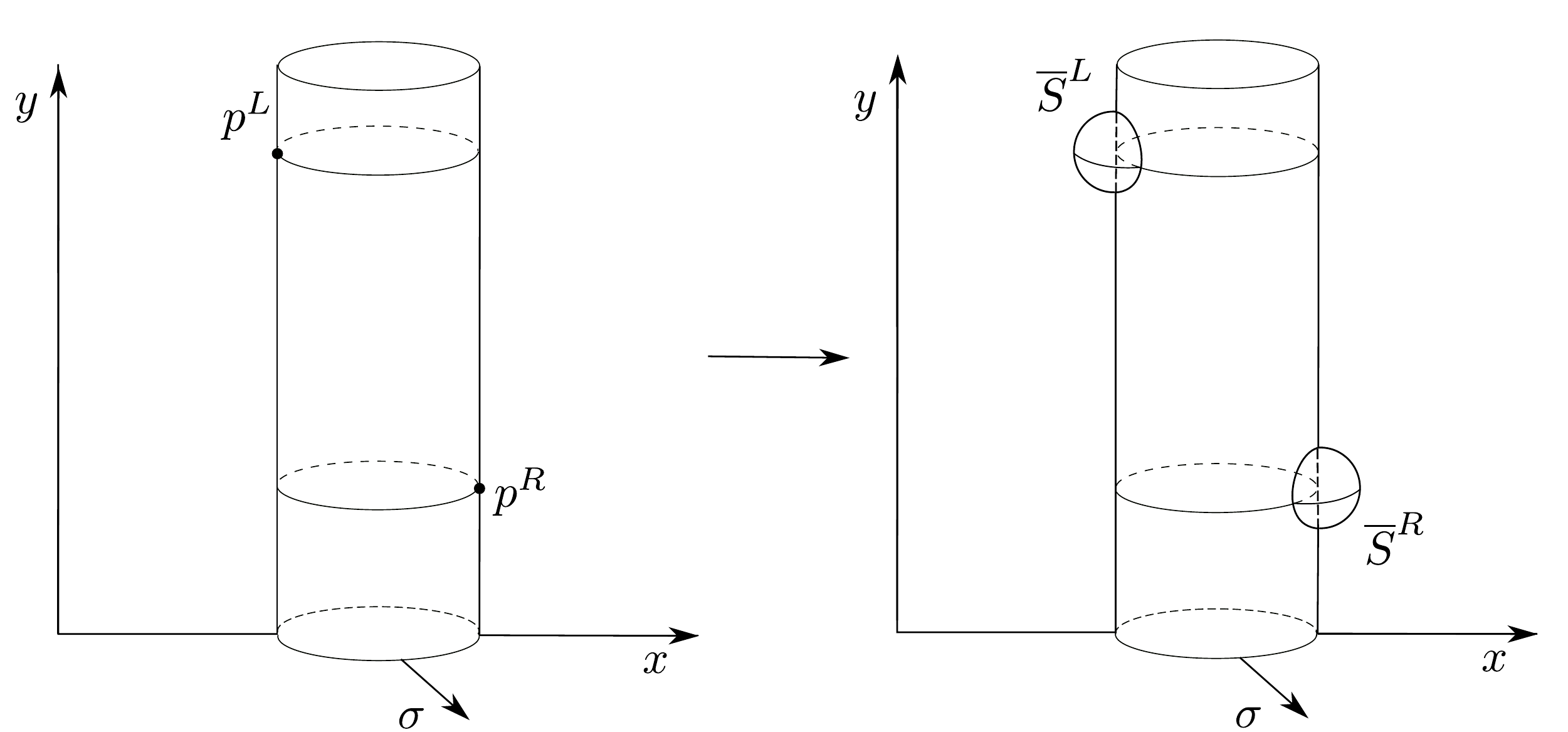}}
\end{center}
\caption{Blowup of $p^L$ and $p^R$ to two quarter-spheres $\overline S^L$ and $\overline S^R$, respectively.}\figlab{x1Eq1BlowupSphere}
\end{figure}

The transformations $\Psi_{i}:\,(\rho_i,(\bar r_i,\bar \delta_i,\bar y_i))\mapsto (r_i,y_i,\delta_i)$, $i=1,3$, defined by \eqsref{Psi1}{Psi3}, give rise to vector-fields $\overline X_{i} = \Psi_i^*(X_i)$ on $\{\rho_1\ge 0\}\times \overline S^L$ and $\{\rho_3\ge 0\}\times\overline S^R$, respectively. Here $\overline X_i\vert_{\rho_i=0}=0$ but the weights on $\rho_i$ have been chosen such that 
\begin{align}
 \widehat X_i = \rho_i^{-k(k+1)} \overline X_i,\eqlab{widehatXi}
\end{align}
is well-defined and non-trivial. In particular, the desingularized vector-field $\widehat X_i$ has improved hyperbolicity properties which will be important for our perturbation technique. It is therefore $\widehat X_i$ that we shall study in the following. 


%


\section{Dynamics on the blow-up spheres}\seclab{sphere}
We now describe the dynamics of $\widehat X_i$ on the two spheres $\overline S^L$ and $\overline S^R$, respectively, which is needed for the proof of \thmref{mainThm} in \secref{proof}.
\subsection{Dynamics on $\overline S^L$}
Henceforth we drop the subscript $1$ for simplicity. Recall also that
\begin{align*}
 \eta^L(\mu) = \mu y^L = \frac{\mu}{\alpha}.
\end{align*}
The results of the following proposition are summarized in \figref{pLblowup}, representing $\overline S^L$ as a half-disk (looking down along the $\bar \delta$-axis).
\begin{proposition}\proplab{prop1}

On $\overline S^L:\,(\bar r_1,\bar \delta_1,\bar y_1)\in S^2,\,\bar r_1\ge 0,\,\bar \delta_1\ge 0$ there exists $6$ or $7$ equilibria of $\widehat X$ \eqref{widehatXi} (the precise number depending on the value of $\eta$, see item \ref{zL} below), including: 
\begin{enumerate}[label=(\alph*)]
 \item $\overline q_{w}^L:\,\bar r^{-(k+1)} \bar y = -\alpha/(1-\alpha),\,\bar \delta=0$ is a hyperbolic saddle with a stable manifold $W^s(\overline q_{w}^L)$ along the invariant half-circle $\bar \delta=0$ and an unstable manifold $$\overline U^L=W^{u}(\overline q_{w}^L),$$ entering $\bar \delta>0$. 
 \item $\overline q_{s}^L:\,(\bar r,\bar y,\bar \delta) = (1,0,0)$ is a hyperbolic unstable node.
 \item $\overline q_{f}^L:\,(\bar r,\bar y,\bar \delta) = (0,0,1)$ is a hyperbolic stable node.
 \item $\overline q_{r}^L:\,\bar r = 0,\bar \delta^{-k(k+1)}\bar y = -\alpha^{-2}\beta \phi^L(0)$ is a nonhyperbolic saddle with a strong unstable manifold $W^u(\overline q_{r}^R)$ along the invariant half-circle $\bar r=0$ and a local center manifold $$\overline C_{loc}^L\equiv W^{c}_{loc}(\overline q_{r}^L),$$ entering $\bar r>0$. For $\eta<\eta^L(\mu)$ the local center manifold is unique as a local (nonhyperbolic) stable manifold of $\overline q_r^L$. 
 \item $\overline a^L:(\bar r,\bar y,\bar \delta) = (0,1,0)$ and $\overline b^L:(\bar r,\bar y,\bar \delta) = (0,-1,0)$ are both hyperbolic saddles.
\item \label{zL} And for any $\eta<\eta^L(\mu)$: 
 \begin{align*}
 \overline z^L:\,\bar \delta^{-k(k+1)}\bar y &= -\alpha^{-2}\beta \phi^L(0)-\alpha^{-1} (\eta^L-\eta)^{k+1},\\
 \bar \delta^{-k}\bar r &=\eta^L-\eta.
 \end{align*}
\end{enumerate}
Furthermore:
\begin{enumerate}[resume*]
 \item There exists a unique number $\eta_H^L(\mu)$, given by
 \begin{align}\eqlab{etaHL}
 \eta_H^L(\mu) = \eta^L(\mu)-\left(\frac{\beta k\phi^L(0)}{\alpha(\alpha+1)}\right)^{1/(k+1)},
 \end{align}
 such that for $\eta=\eta_H^L(\mu)$ and any $\mu\ge 0$, the equilibrium $\overline z^L$ undergoes a sub-critical Hopf bifurcation. Therefore:
 \begin{enumerate}[label=(\alph{enumi}.\arabic*)]
 \item \label{limitcycleHopf} There exists a $c>0$ sufficiently small such that for $\eta \in [\eta_H^L(\mu)-c,\eta_H^L(\mu))$ there exists a family of locally unique hyperbolic and repelling limit cycles. 
 \item $\overline z^L$ is hyperbolic and unstable (stable) for $\eta\in (\eta_H^L(\mu),\eta^L(\mu))$ ($\eta<\eta_H^L(\mu)$, respectively).
 \end{enumerate}
 \item \label{etaHet0LNeg} There exists a unique number $\eta_{Het,0}^L<0$ (independent of $\mu$) such that if $\eta_{Het}^L(\mu) = \eta^L(\mu)+\eta_{Het,0}^L$ then the following holds for any $\mu\ge 0$: For $\eta=\eta_{Het}^L(\mu)$ the system undergoes a heteroclinic bifurcation where the unique center/stable manifold $\overline C^L$ of $\overline q_{r}^L$ coincides with the unstable manifold $\overline U^L$ of $\overline q_{w}^L$. The intersection is transverse in the $((\bar r,\bar y,\bar \delta),\eta)$-space and:
  \begin{enumerate}[label=(\alph{enumi}.\arabic*)]
 \item  \label{limitcycleHom} For $\eta>\eta_{Het}^L(\mu)$ the unstable manifold $\overline U^L$ is forward asymptotic to $\overline q^L_f$. Furthermore, there exists a $c>0$ sufficiently small such that for $\eta \in (\eta_{Het}^L(\mu),\eta_{Het}^L(\mu)+c]$ there exists a family of locally unique hyperbolic and repelling limit cycles.
 \item For $\eta<\eta_{Het}^L(\mu)$ the center/stable manifold $\overline C^L$ is backward asymptotic to $\overline q_{s}^L$. 
\end{enumerate}
\end{enumerate}
\end{proposition}

We prove \propref{prop1} in \appref{proofprop1}. The main difficulty lies in \ref{etaHet0LNeg} and the existence of a unique $\eta_{Het,0}^L$. For existence, we study two ``extreme'' cases with $\eta\ll \eta^L(\mu)$ and $\eta$ close to $\eta^L(\mu)$, where we can determine the limit sets of the manifolds $\overline C^L$ and $\overline U^L$. Specifically, for $\eta\ll \eta^L(\mu)$, after having transformed the system into a perturbation of a Lienard system, we apply Cherkas' theorem, see e.g. \cite[Theorem 3 p. 265]{perko2001a}, to exclude existence of limit cycles. The existence of $\eta_{Het,0}^L$ is then based on a continuity argument that shows that at least one heteroclinic intersection exists in between these extreme cases. For uniqueness we use a mononicity argument based on a Melnikov computation. This is all laid out in the proof of \lemmaref{essentiallem}, see further details in \appref{existence} and \appref{uniqueness}. 

In \figref{pLblowup}, we suppose that $\eta^L_{Het}(\mu)<\eta_H^L(\mu)$. We do not have a proof of this global property but our computations seem to suggest that this is the case. We emphasize that this missing detail is not important for the proof of our main theorem. Our computations also suggest that the limit cycles in \ref{limitcycleHopf} and \ref{limitcycleHom} belong to the same family of repelling periodic orbits. (It is tempting to prove this using \cite[Theorem 1, p. 386]{perko2001a} but $\widehat X$ is not a rotated vector-field, see \cite[Definition 1, p. 384]{perko2001a}.)  

\begin{figure}[h!] 
\begin{center}
\subfigure[$\eta<\eta_{Het}^L<\eta_H^L<\eta^L$]{\includegraphics[width=.4\textwidth]{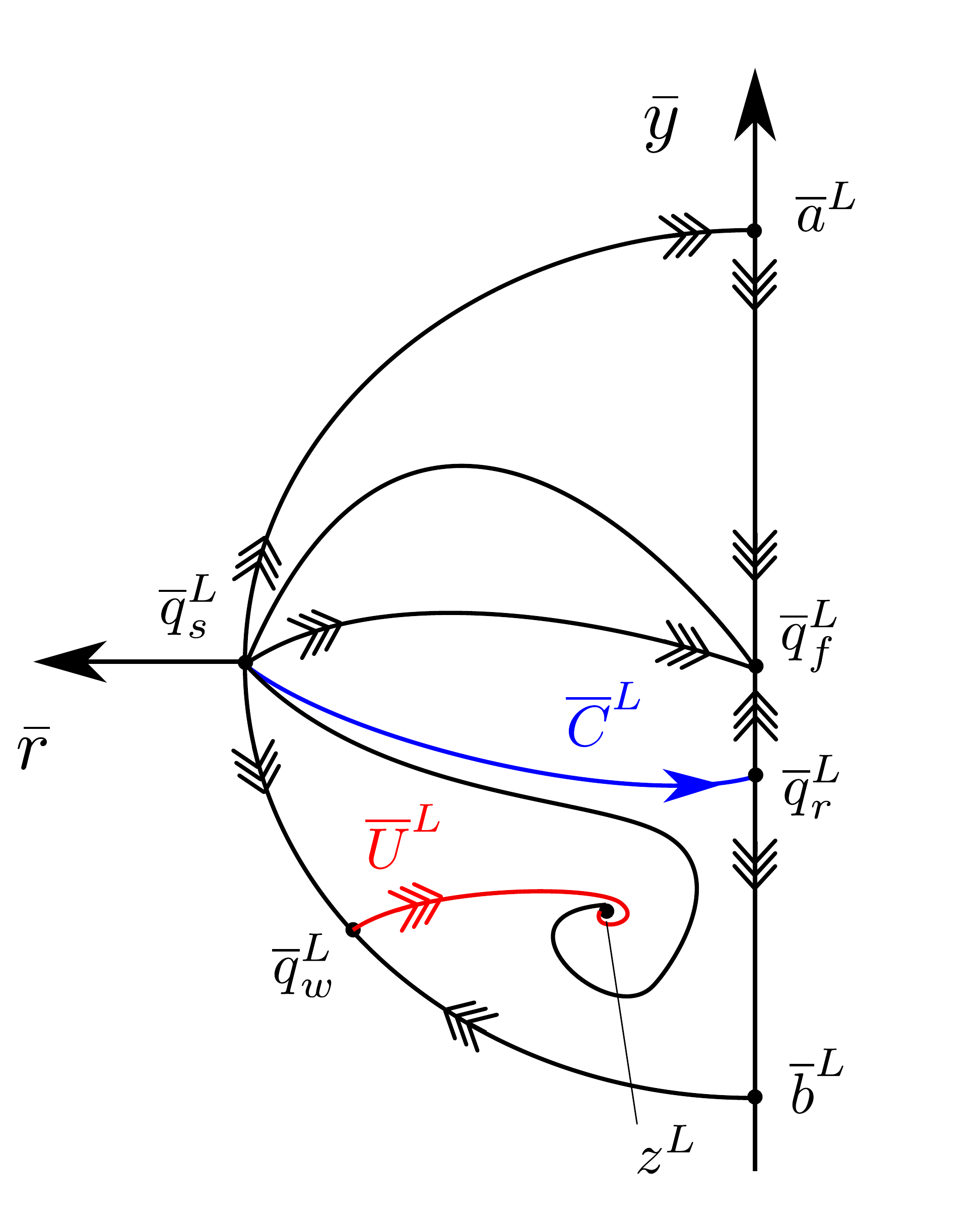}}
\subfigure[$\eta=\eta_{Het}^L<\eta_H^L<\eta^L$]{\includegraphics[width=.4\textwidth]{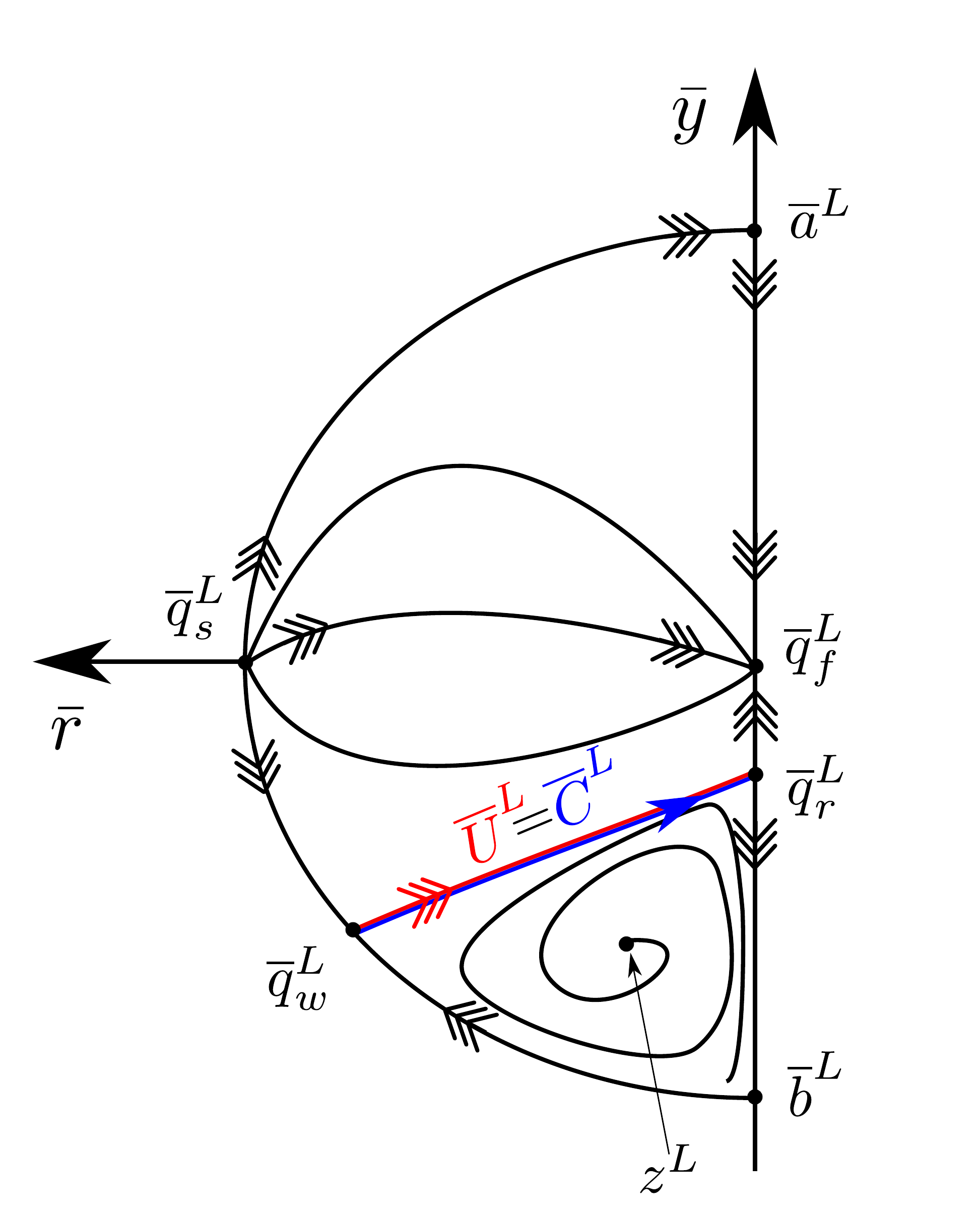}}
\subfigure[$\eta_{Het}^L<\eta<\eta_H^L<\eta^L$]{\includegraphics[width=.4\textwidth]{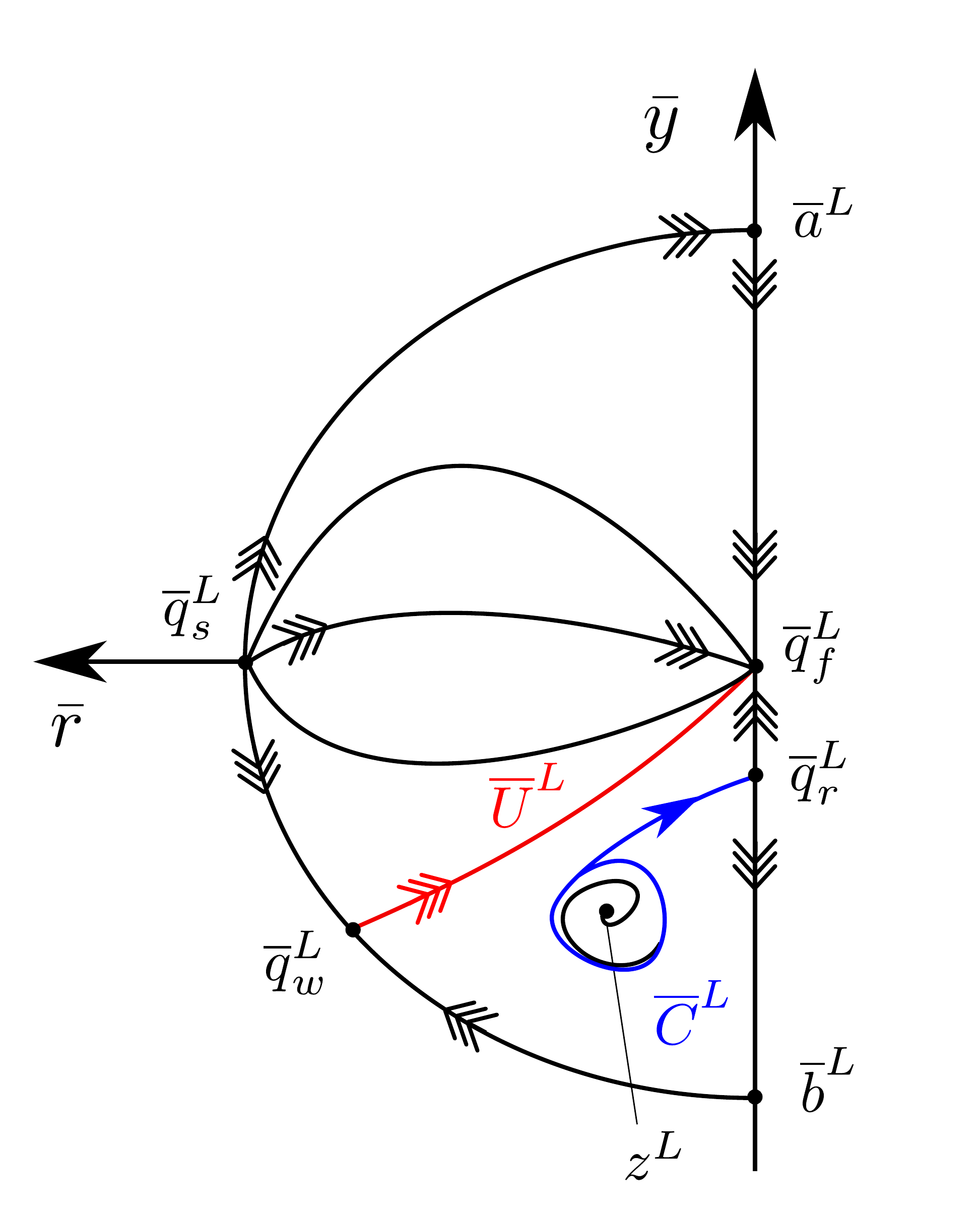}}
\subfigure[$\eta_{Het}^L<\eta_H^L<\eta<\eta^L$]{\includegraphics[width=.4\textwidth]{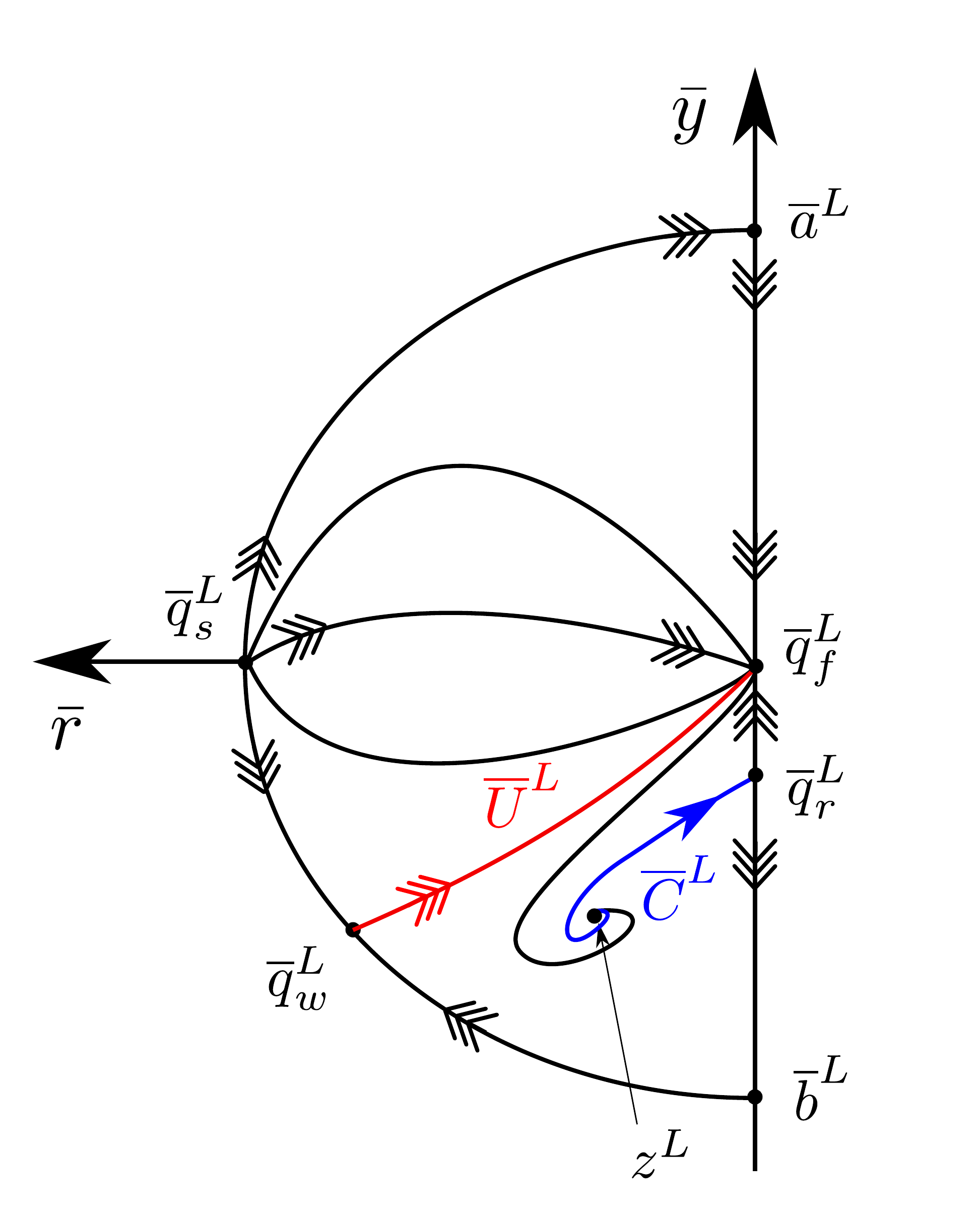}}
\end{center}
 \caption{Dynamics on $\overline S^L$ for different values of $\eta$ as described by \propref{prop1}.}
\figlab{pLblowup}
\end{figure}

\subsection{Dynamics on $\overline S^R$}
We now drop the subscript $3$ for simplicity and recall that
\begin{align*}
 \eta^R(\mu) = \mu y^R = \frac{\mu}{\alpha+\beta}.
\end{align*}
The results of the following proposition are summarized in \figref{pRblowup}.
\begin{proposition}\proplab{prop2}
On $\overline S^R:\,(\bar r,\bar y,\bar \delta)\in S^2,\,\bar r\ge 0,\,\bar \delta\ge 0$ there exists $6$ or $7$ equilibria of $\widehat X$ \eqref{widehatXi} (the precise number depending on the value of $\eta$, see item \ref{zR} below), including: 
\begin{enumerate}[label=(\alph*)]
 \item $\overline q_{w}^R:\,(\bar r,\bar y,\bar \delta) = (1,0,0)$ is a hyperbolic saddle with a stable manifold $W^s(\overline q_{w}^R)$ along the invariant half-circle $\bar \delta=0$ and an unstable manifold $\overline U^R\equiv W^{u}(\overline q_{w}^R)$ entering $\bar \delta>0$. 
 \item $\overline q_{s}^R:\,\bar r^{-(k+1)} \bar y = -(\alpha+\beta-1)/(\alpha+\beta),\,\bar \delta=0$ is a hyperbolic unstable node.
 \item $\overline q_{f}^R:\,(\bar r,\bar y,\bar \delta) = (0,0,1)$ is a hyperbolic stable node.
 \item $\overline q_{r}^R:\,\bar r = 0,\bar \delta^{-k(k+1)}\bar y = (\alpha+\beta)^{-2}\beta \phi^R(0)$ is a nonhyperbolic saddle with a strong unstable manifold $W^u(\overline q_{r}^R)$ along the invariant half-circle $\bar r=0$ and a local center manifold $\overline C_{loc}^R\equiv W^{c}_{loc}(\overline q_{r}^R)$ entering $\bar r>0$. For $\eta>\eta^R(\mu)$ the local center manifold is unique as a (nonhyperbolic) stable manifold of $\overline q_r^R$. 
 \item $\overline a^R:(\bar r,\bar y,\bar \delta) = (0,1,0)$ and $\overline b^R:(\bar r,\bar y,\bar \delta) = (0,-1,0)$ are both hyperbolic saddles.
  \item \label{zR} And for any $\eta>\eta^R(\mu)$: 
 \begin{align*}
 \overline z^R:\,\bar \delta^{-k(k+1)}\bar y &= (\alpha+\beta)^{-2}\beta \phi^R(0)+(\alpha+\beta)^{-1} (\eta-\eta^R)^{k+1},\\
 \bar \delta^{-k}\bar r &=\eta-\eta^R.
 \end{align*}
\end{enumerate}
Furthermore:
\begin{enumerate}[resume*]
 \item There exists a unique number $\eta_H^R(\mu)$, given by 
 \begin{align}\eqlab{etaHR}
 \eta_H^R(\mu) = \eta^R(\mu)+\left(\frac{\beta k\phi^R(0)}{(\alpha+\beta)(\alpha+\beta+1)}\right)^{1/(k+1)}.
 \end{align}
 such that for $\eta=\eta_H^R(\mu)$ and any $\mu\ge 0$ the equilibrium $\overline z^R$ undergoes a sub-critical Hopf bifurcation. Therefore:
  \begin{enumerate}[label=(\alph{enumi}.\arabic*)]
 \item  \label{limitcycleHopfR} There exists a $c>0$ sufficiently small such that for $\eta \in (\eta_H^R(\mu),\eta_H^R(\mu)+c]$ there exists a family of locally unique hyperbolic and repelling limit cycles. 
 \item $\overline z^R$ is hyperbolic and attracting (repelling) for $\eta>\eta^R_H(\mu)$ ($\eta\in (\eta^R(\mu),\eta_H^R(\mu))$, respectively). 
 \end{enumerate}
 \item  \label{etaHet0RPos} There exists a unique number $\eta_{Het,0}^R>0$ (independent of $\mu$)  such that if $\eta_{\text{Het}}^R(\mu) = \eta^R(\mu)+\eta_{\text{Het},0}^L$ then the following holds: For $\eta=\eta_{Het}^R(\mu)$ the system undergoes a heteroclinic bifurcation where the unique center/stable manifold $\overline C^R$ of $\overline q_{r}^R$ coincides with the unstable manifold $\overline U^R$ of $\overline q_{w}^R$. The intersection is transverse in the $((\bar r,\bar y,\bar \delta),\eta)$-space and:
  \begin{enumerate}[label=(\alph{enumi}.\arabic*)]
 \item For $\eta<\eta_{Het}^R(\mu)$ the unstable manifold $\overline U^R$ is forward asymptotic to $\overline q^R_f$. Furthermore, there exists a $c>0$ sufficiently small such that for $\eta \in [\eta_{Het}^R(\mu)-c,\eta_{Het}^R(\mu))$ there exists a family of locally unique hyperbolic and repelling limit cycles.
 \item For $\eta>\eta_{Het}^R(\mu)$ the center/stable manifold $\overline C^R$ is backward asymptotic to $\overline q_{s}^R$. 
\end{enumerate}
%
\end{enumerate}
\end{proposition}
In \figref{pRblowup} we suppose that $\eta_{H}^R(\mu)<\eta_{Het}^R(\mu)$. We also do not have a proof of this global property but computations again suggest that this is the case. 
The proof of \propref{prop2} follows from \propref{prop1}, see \appref{transformation}.


\begin{figure}[h!] 
\begin{center}
\subfigure[$\eta^R<\eta_{H}^R<\eta_{Het}^R<\eta$]{\includegraphics[width=.4\textwidth]{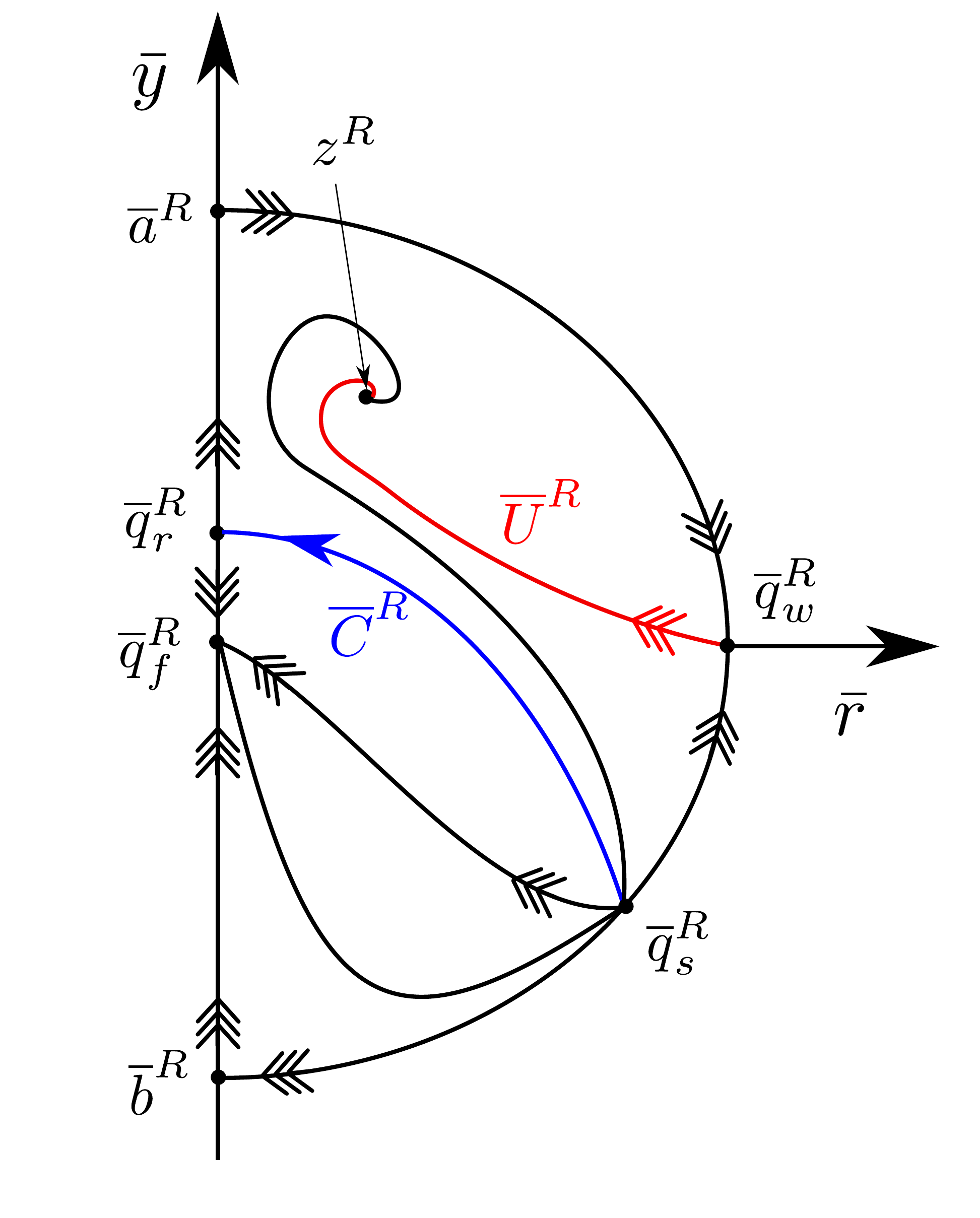}}
\subfigure[$\eta^R<\eta_{H}^R<\eta=\eta_{Het}^R$]{\includegraphics[width=.4\textwidth]{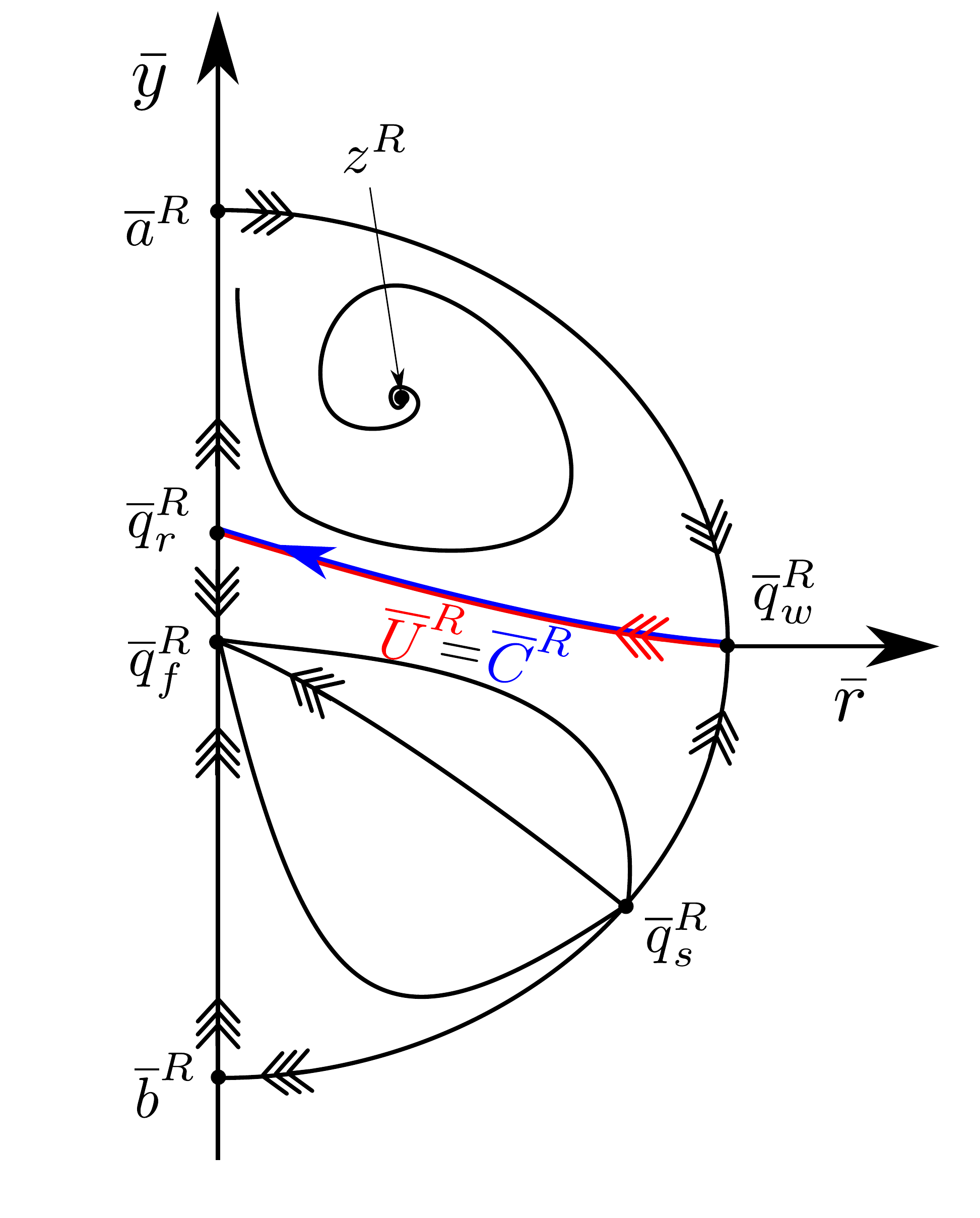}}
\subfigure[$\eta^R<\eta_{H}^R<\eta<\eta_{Het}^R$]{\includegraphics[width=.4\textwidth]{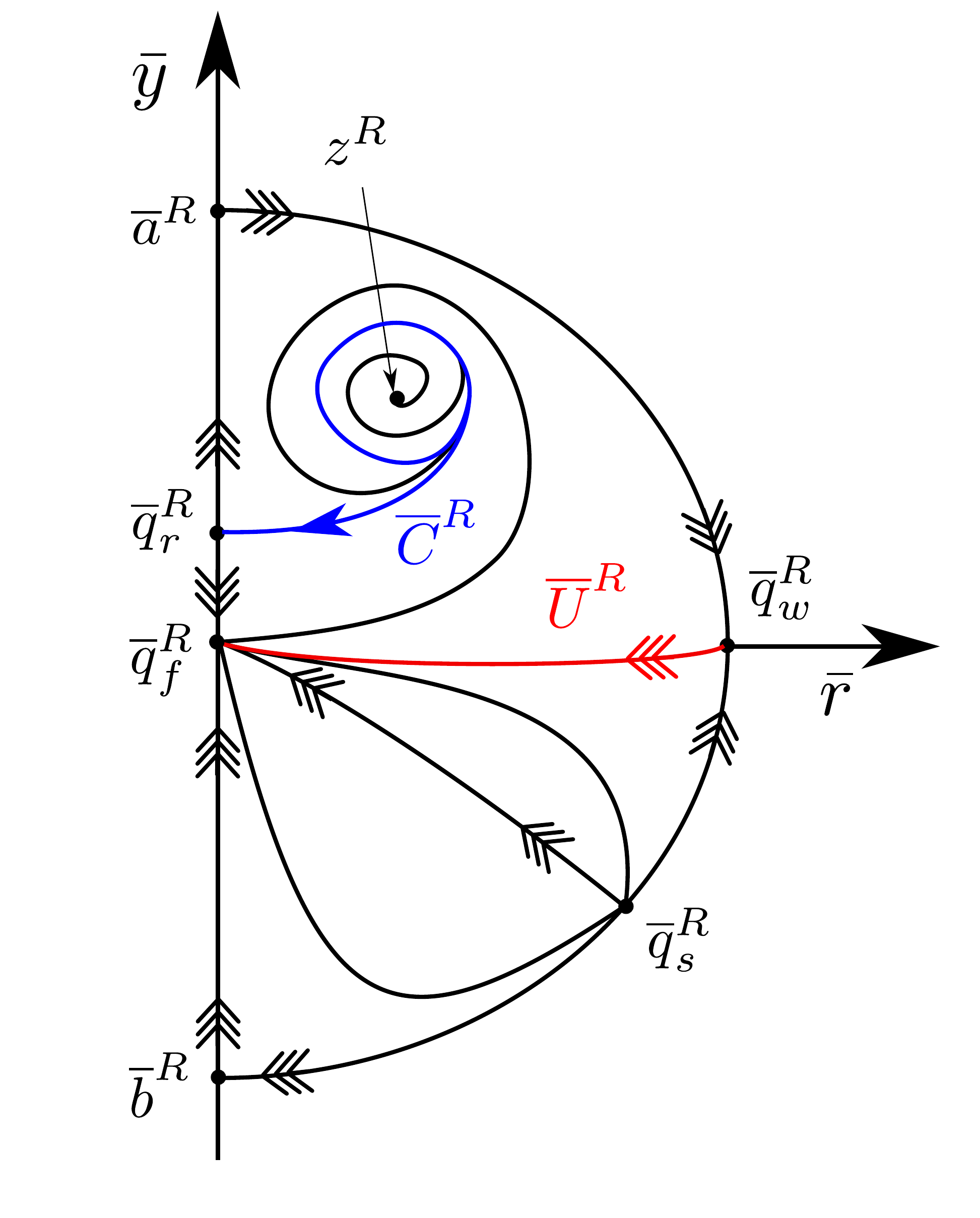}}
\subfigure[$\eta^R<\eta<\eta_{H}^R<\eta_{Het}^R$]{\includegraphics[width=.4\textwidth]{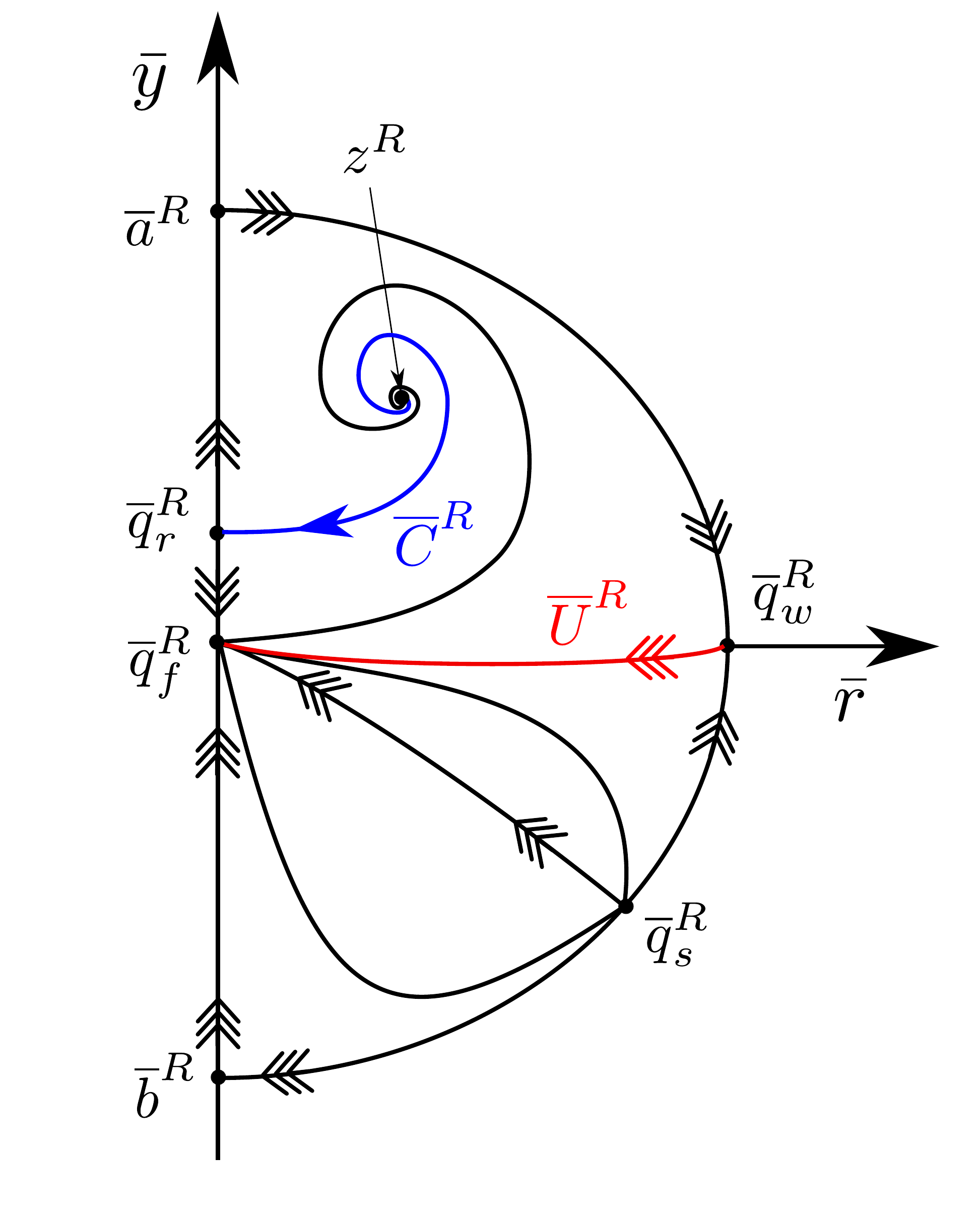}}
\end{center}
 \caption{Dynamics on $\overline S^R$ for different values of $\eta$ as described by \propref{prop2}.}
\figlab{pRblowup}
\end{figure}

\begin{remark}
In \figref{pLR}(a) and (b), we illustrate the dynamics on $\overline S^L$ and $\overline S^R$ in the cases where $\eta>\eta^L(\mu)$ and $\eta<\eta^R(\mu)$. In these cases, $\overline z^L$ and $\overline z^R$ have disappeared following the bifurcation that occurs at $\eta=\eta^L(\mu)$ and $\eta = \eta^R(\mu)$, respectively, where $\overline z^{L/R}$ coincide with $\overline q_{r}^{L/R}$, respectively, for any $\mu\ge 0$. As a consequence, the direction of the flow on the center manifolds change at $\eta=\eta^{L/R}(\mu)$. Recall also that the slow flow on the critical manifold changes, see \eqref{slowFlow2}, at the same value of $\eta=\eta^{L/R}(\mu)$. Therefore the center manifolds $\overline C^{L/R}$ are nonunique for $\eta \gtrless \eta^{L/R}(\mu)$, respectively.  In either of these cases, the unstable manifolds $\overline U^{L/R}$, as trajectories, are always forward asymptotic to $\overline q_f^{L/R}\in \overline p^{L/R}$, respectively.
\end{remark}
\begin{figure}[h!] 
\begin{center}
\subfigure[$\eta^L<\eta$]{\includegraphics[width=.4\textwidth]{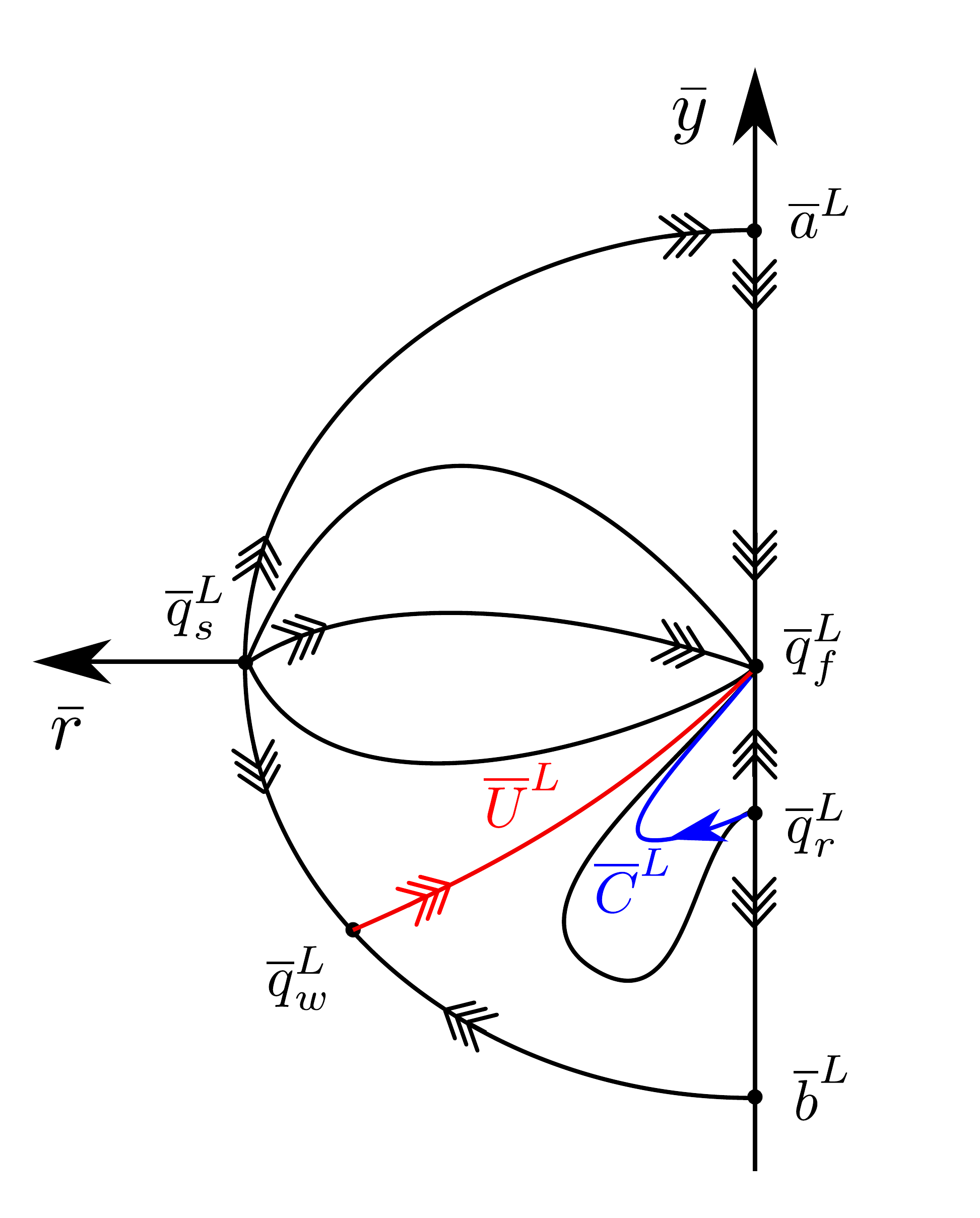}}
\subfigure[$\eta<\eta^R$]{\includegraphics[width=.4\textwidth]{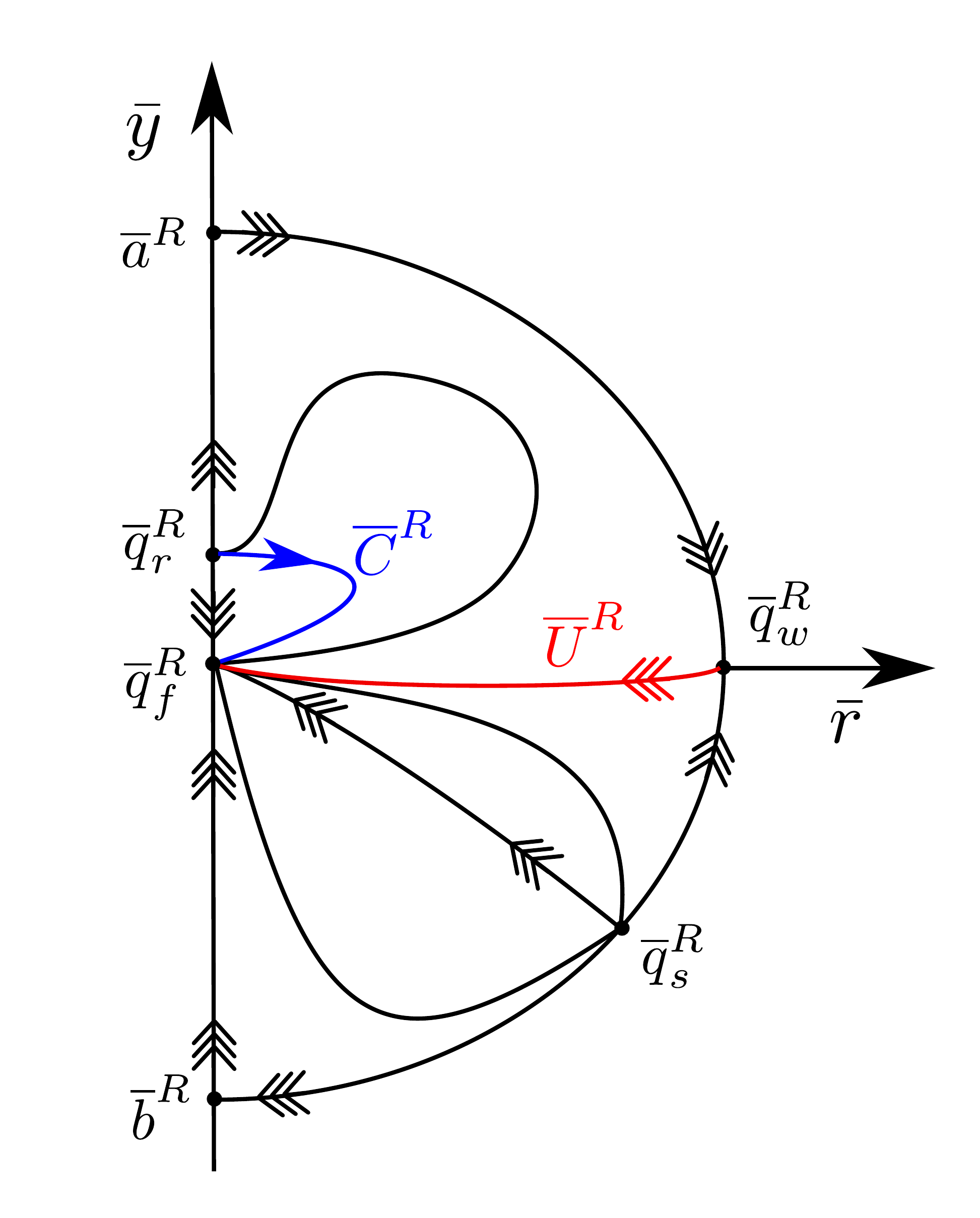}}
\end{center}
 \caption{Dynamics on $\overline S^{L}$ (a) and $\overline S^R$ (b) in the cases $\eta>\eta^L$ and $\eta<\eta^R$, respectively. Here $\overline C^{L/R}$ are both nonunique. Also $\overline U^{L/R}$, $\overline C^{L/R}$ all converge as trajectories to $\overline q_f^{L/R}$, respectively. }
\figlab{pLR}
\end{figure}
%
%
%
%
%

%


\section{Proof of \thmref{mainThm}}\seclab{proof}
To prove \thmref{mainThm} we now combine the results in \secref{blowup} and \secref{sphere} which are illustrated in \figref{relax}. 
First, we notice that \eqref{numbers} follows from \propref{prop1} \ref{etaHet0LNeg} and \propref{prop2} \ref{etaHet0RPos}. The existence of a unique $\mu_*$ : $\eta^L_{Het}(\mu_*)=\eta^R_{Het}(\mu_*)$ in \thmref{mainThm} \ref{item3} is then a simple geometric consequence of the slope $(\alpha+\beta)^{-1}$ of $\eta^R_{\text{Het}}(\mu)$ being smaller than the positive slope $\alpha^{-1}$ of $\eta^L_{Het}(\mu)$, see \eqref{eta1Leta1R} and \figref{etaLRHet}.


%
We will now prove item \ref{item4} and therefore consider the case when $\eta_1\in (\eta^L_{\text{Het}}(\mu_1),\eta^R_{\text{Het}}(\mu_1))$. In the following, we again suppress the subscripts on $\eta_1$ and $\mu_1$. 
By our blow-up approach we identify a closed cycle $\overline \Gamma_0$ with improved hyperbolicity properties. See \figref{relax}. The points $\overline q_w^{L/R}$ and $q_{f}^{L/R}$ are each hyperbolic, recall \propref{prop1} and \propref{prop2}. We therefore let $\Sigma$ be a small section transverse to $\Gamma_0^L$ within $x<1$ and then describe the return mapping $P$ to $\Sigma$ using local hyperbolic methods near $\overline q_w^{L/R}$, $\overline q_{f}^{L/R}$ and $\overline E^{L/R}$. Before providing more details, we note that from the diagram in \figref{relax}, it is almost obvious that $P$ is well-defined for $0<\varepsilon\ll 1$. In particular, for $0<\varepsilon\ll 1$ the forward flow of $\Sigma$ is initially close to the stable manifold of $\overline q_w^L$. Therefore by the contraction within this $2D$ stable manifold, along with similar contractions near $\overline q_f^L$, $\overline q_w^R$, and $\overline q_f^R$, it is clear that for any $c>0$: (i) the image of $P(\Sigma)$ is closer to $\Gamma_0\cap \Sigma$ than $c$ and (ii) $P$ is a Lipschitz map with Lipschitz constant smaller than $c$, for all $\varepsilon>0$ sufficiently small. The attracting limit cycle (i.e. the relaxation oscillation) is then obtained as the unique fix-point of $P$ on $\Sigma$, by the contraction mapping theorem.  


We now provide further details. For this we first use regular perturbation theory up close to $\overline q_w^L$. We do this in charts $\bar x=-1$ and $\bar r=1$, see \eqref{barXN1New} and \eqref{barR1pL} in \appref{proofprop1}, using the local coordinates $(\rho_1,y_1,\delta_1)$ to describe the blow-up \eqref{Psi1}:
\begin{align*}
 x &= 1-\rho_1^{k(k+1)},\\
 y &= y^L+\rho_1^{k(k+1)} y_1,\\
\sigma  &= \rho_1^{k+1} \delta_1.
\end{align*}
In these coordinates, $\overline q_w^L$ is given by $(\rho_1,y_1,\delta_1)=(0,-\alpha^{-1}(1-\alpha),0)$ and denoted $q_{w,1}^L$. 
Here we then apply the following local result.
\begin{figure}[h!] 
\begin{center}
{\includegraphics[width=.9\textwidth]{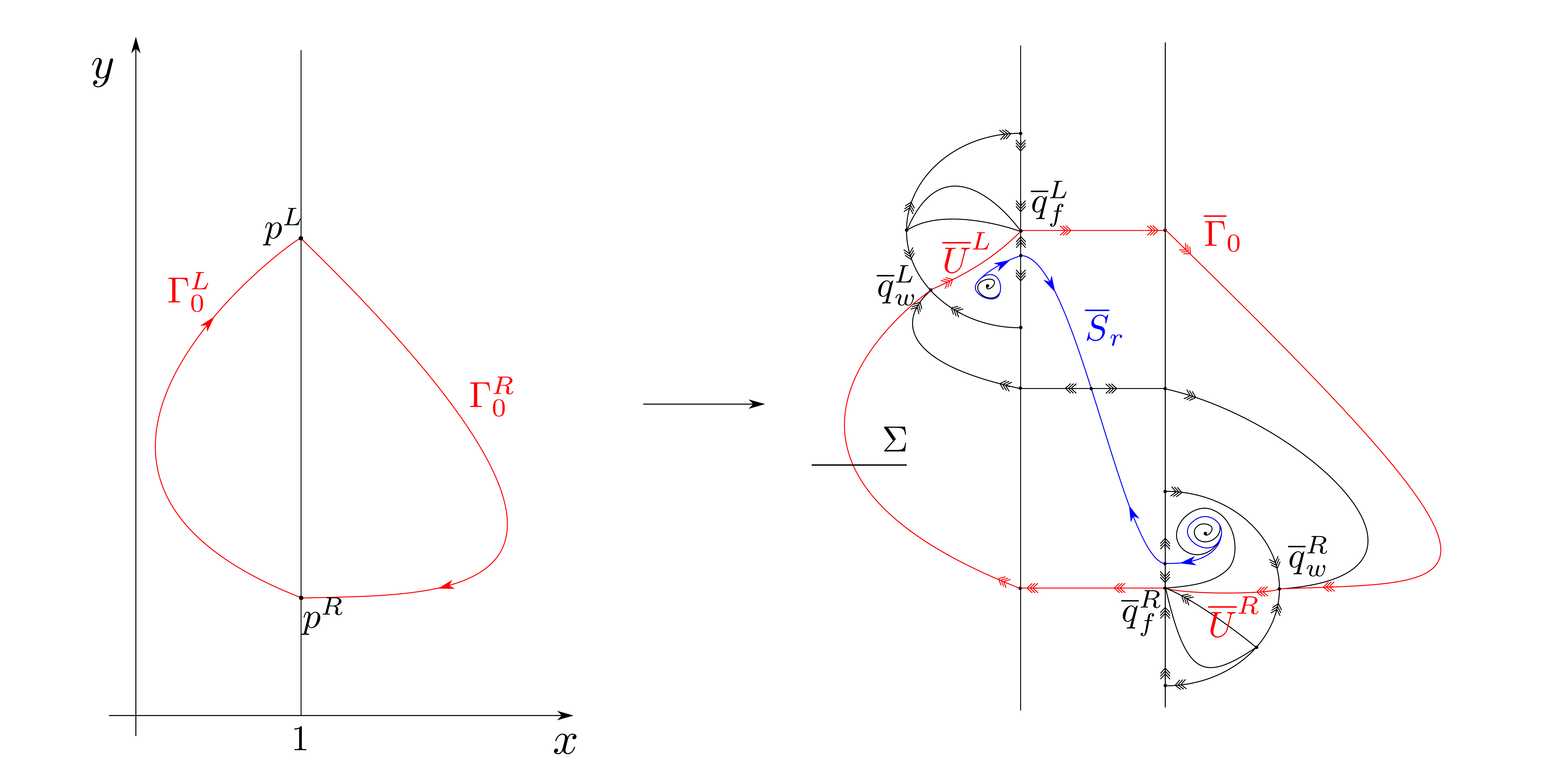}}
\end{center}
 \caption{Blowup of the singular cycle for $\eta\in (\eta_{Het}^L(\mu),\eta_{Het}^R(\mu))$. Here $\eta$ and $\mu$ are actually $\eta_1$ and $\mu_1$ in \eqref{eta1mu1}.}
\figlab{relax}
\end{figure}
\begin{lemma}\lemmalab{Plocal1}
 Let $B=-\alpha^{-1}(1-\alpha)$ be the $y_1$-value of $q_{w,1}^L$,  and set 
 \begin{align*}
 \Sigma_{w,1}^{L,\text{in}} &= \{(\rho_1,y_1,\delta_1) \vert \rho_1 =\xi,\,y_1 \in [B-\upsilon,B+\upsilon], \delta_1 \in [0,\nu]\},\\
 \Sigma_{w,1}^{L,\text{out}}&= \{(\rho_1,y_1,\delta_1) \vert \rho_1 \in [0,\xi],\,y_1 \in [B-\upsilon,B+\upsilon], \delta_1 =\nu\},
  \end{align*}
  for appropriately small $\xi>0,\upsilon>0$ and $\nu>0$ and consider the associated mapping $$P_1:\Sigma_{w,1}^{L,\text{in}}\rightarrow \Sigma_{w,1}^{L,\text{out}},\,(\rho_1,y_1,\delta_1)\mapsto (\rho_{1+},y_{1+},\delta_{1+})$$ obtained by the first intersection of the forward flow. Then there exists two locally defined $C^1$-functions $H(y_1,\delta_1^k,\sigma)$ and $\tilde H(y_1,\delta_1^k,\sigma)$ such that 
  \begin{enumerate}[label=(\alph*)]
   \item $y_1=H(\tilde y_1,\delta_1^k,\sigma) \Leftrightarrow \tilde y_1 = \tilde H(y_1,\delta_1^k,\sigma)$ locally.
     \item $H(0,0,0)=B$ and the smooth graph $y_1 = H(0,\delta_1^k,0)$, $\rho_1=0$, $\delta_1\in [0,\nu]$ is $U^L_{loc,1}$.
  \item  
  \begin{align*}
   P_1(\rho_1,y_1,\delta_1) = \begin{pmatrix}
                               \left(\delta_1\nu^{-1}\right)^{1/(k+1)}\xi\\
                               H\left(\left(\delta_1\nu^{-1}\right)^{k\vert B\vert}\tilde H(y_1,\delta_1^k,\rho_1^{k+1}\delta_1),\nu^k,\rho_1^{k+1}\delta_1\right)\\
                               \nu
                              \end{pmatrix}.
                                \end{align*}
                   \item             In particular,
                                \begin{align*}
                                 P_1(\rho_1,y_1,0) = U_{loc,1}^L\cap \Sigma_{w,1}^{L,\text{out}}
                                \end{align*}
and the restricted mapping: $y_1\mapsto P(\rho_1,y_1,\delta_1)$, for each fixed $\delta_1$, has a Lipschitz constant $L_1(\delta_1)$ satisfying
\begin{align*}
 L_1(\delta_1)\le c \delta_1^{k\vert B\vert},
\end{align*}
with $c>0$ sufficiently large, for all $\delta_1 \in (0,\nu]$.  
\end{enumerate}
\end{lemma}
\begin{proof}
By \propref{prop1} the mapping is described by a passage near a hyperbolic saddle. The result follows from partial linearization using \cite{sternberg1958a}, see further details in \appref{lemma5}.
\end{proof}
Subsequently, we use regular perturbation to track orbits up close to $\overline q_f^L$ by following $\overline U^L$. Near $\overline q_f^L$ we use charts $\bar x=-1$,  $\bar \delta=1$, see \eqref{barXN1New} and \eqref{bard1pL}, and the local coordinates $(\rho_2,r_2,y_2)$ to describe the blow-up \eqref{Psi1}:
\begin{align*}
 x &= 1+\rho_2^{k(k+1)}r_2,\\
 y &= y^L+\rho_2^{k(k+1)} y_2,\\
\sigma  &= \rho_2^{k+1}.
\end{align*}
In these coordinates, $\overline q_f^L$ is given by $(\rho_2,r_2,y_2)=(0,0,0)$  and denoted $q_{f,2}^L$. 
Here we then have the following local result.
\begin{lemma}\lemmalab{Plocal2}
 Let 
 \begin{align*}
 \Sigma_{f,2}^{L,\text{in}} &= \{(\rho_2,r_2,y_2) \vert \rho_2 \in [0,\xi],\,y_2 \in [-\upsilon,\upsilon], r_2 =\nu\},\\
 \Sigma_{f,2}^{L,\text{out}}&= \{(\rho_2,r_2,y_2) \vert \rho_2 =\xi,\,y_2 \in [-\upsilon,\upsilon], r_2 \in [0,\nu]\},
  \end{align*}
  for appropriately small $\xi>0,\upsilon>0$ and $\nu>0$ and consider the associated mapping $$P_2:\Sigma_{f,2}^{L,\text{in}}\rightarrow \Sigma_{f,2}^{L,\text{out}},\,(\rho_2,r_2,y_2)\mapsto (\rho_{2+},r_{2+},y_{2+})$$ obtained by the first intersection of the forward flow. Then there exists two (new) locally defined smooth functions $H(y_2,r_2)$ and $\tilde H(y_2,r_2)$ such that 
  \begin{enumerate}[label=(\alph*)]
%
   \item $y_2=H(\tilde y_2,r_2) \Leftrightarrow \tilde y_2 = \tilde H(y_2,r_2)$ locally.
  \item $H(\tilde y_2,r_2) = \tilde y_2+\mathcal O(r_2^{k+1})$ and $\tilde H(y_2,r_2) = y_2+\mathcal O(r_2^{k+1})$
  \item  \begin{align*}
   P_2(\rho_2,r_2,y_2) = \begin{pmatrix}
                           \xi\\
                              \left(\rho_2\xi^{-1}\right)^{k+1}r_2\\
                               H\left(\left(\rho_2\xi^{-1}\right)^{k(k+1)}\tilde H(y_2,r_2)+\mathcal O (\rho_2^{k+1} \log \rho_2^{-1}) ,\left(\rho_2\xi^{-1}\right)^{k+1}r_2\right)
                              \end{pmatrix},                              
  \end{align*}
  with the $\mathcal O$-term being Lipschitz with respect to $y_2$ with a Lipschitz constant $\mathcal O (\rho_2^{k+1} \log \rho_2^{-1})$.
  \item In particular, 
  \begin{align*}
   P_2(0,r_2,y_2) = \begin{pmatrix}
                     \xi\\
                     0\\
                     0
                    \end{pmatrix},
  \end{align*}
and the restricted mapping $y_2\mapsto P_2(\rho_2,r_2,y_2)$, for each fixed $\rho_2$, has a Lipschitz constant $L_2(\rho_2)$ satisfying
\begin{align*}
L_2(\rho_2) \le c \rho_2^{k+1} \log \rho_2^{-1},
\end{align*}
with $c>0$ sufficiently large, for all $\rho_2 \in (0,\xi]$. 
\end{enumerate}

\end{lemma}
\begin{proof}
By \propref{prop1} the mapping is described by a passage near a hyperbolic saddle. The result follows from partial linearization using \cite{sternberg1958a}, see further details in \appref{lemma6}.
\end{proof}
At $\Sigma_{f,2}^{L,\text{out}}$ we can blow down to the coordinates $(r_1,y_1,\delta_1)$ in chart $\bar x=-1$, see \eqref{barXN1New}.
Then transforming the result into $\bar \delta=1$, see \eqref{baru1}, and applying regular perturbation theory we end up in chart $\bar x=1$ with coordinates $(r_3,y_3,\delta_3)$, see \eqref{barX1New} and the equations \eqref{X3}, near the line of equilibria $E^R_3 =\{r_3=\delta_3=0,\,y_3\ne y^R\}$. This line is normally hyperbolic and has stable and unstable manifolds $W^s(E_3^R)$ and $W^u(E_3^R)$ and we can therefore describe the passage near this manifold as follows.
\begin{lemma}\lemmalab{Plocal3}
 Let 
 \begin{align*}
 \Sigma_{e,2}^{R,\text{in}} &= \{(r_3,y_3,\delta_3) \vert r_3 \in [0,\xi],\,y_3 \in [y^L-\upsilon,y^L+\upsilon], \delta_3 =\nu\},\\
 \Sigma_{e,2}^{R,\text{out}}&= \{(r_3,y_3,\delta_3) \vert r_3 =\xi,\,y_3 \in [y^L-\upsilon,y^L+\upsilon], \delta_3 \in [0,\nu]\},
  \end{align*}
  for appropriately small $\xi>0,\upsilon>0$ and $\nu>0$ and consider the associated mapping $$P_3:\Sigma_{e,2}^{R,\text{in}}\rightarrow \Sigma_{e,2}^{R,\text{out}},\,(r_3,y_3,\delta_3)\mapsto (r_{3+},y_{3+},\delta_{3+})$$ obtained by the forward flow.  Then there exists (new) locally defined smooth functions $H(y_3,r_3)$ and $\tilde H(y_3,r_3)$ such that
  \begin{enumerate}[label=(\alph*)]
   \item $y_3=H(\tilde y_3,r_3)\Leftrightarrow \tilde y_3 = \tilde H(y_3,r_3)$ locally. 
   \item Fix any $\tilde y_3\in [y^L-\upsilon,y^L+\upsilon]$. Then the graph $y_3=H(\tilde y_3,r_3)$, $\delta_3=0$, $r_3\in [0,\xi]$ is the local unstable manifold $W_{loc}^u(0,\tilde y_3,0)$of $(r_3,y_3,\delta_3)=(0,\tilde y_3,0)$ within $W^u(E_3^R)$. 
   \item $H(y_3,r_3) = y_3+\mathcal O(r_3^{k+1})$ and $\tilde H(y_3,r_3)=y_3+\mathcal O(r_3^{k+1})$. 
   \item \begin{align*}
          P_3(r_3,y_3,\delta_3) = \begin{pmatrix}
          \xi\\
                                   H\left(\tilde H(y_3,r_3)+\mathcal O(r_3^{k+1} \log r_3^{-1} ),\xi\right)\\
                                   (r_3 \xi^{-1}) \delta_3
                                  \end{pmatrix}.
         \end{align*}
           with the $\mathcal O$-term being Lipschitz with respect to $y_3$ with a Lipschitz constant $\mathcal O (r_3^{k+1} \log r_3^{-1})$.
         \item In particular,
         \begin{align*}
          P_3(0,y^L,\delta_3) = W^u_{loc}(0,y^L,0) \cap \Sigma_{e,3}^{R,\text{out}} 
         \end{align*}
and the restricted mapping $y_3\mapsto P_3(r_3,y_3,\delta_3)$, for each fixed $r_3$, has a Lipschitz constant $L_3(r_3)$ satisfying
\begin{align*}
L_3(r_3) \le c,
\end{align*}
with $c>0$ sufficiently large, for all $r_3\in (0,\xi]$. 
  \end{enumerate}

\end{lemma}
\begin{proof}
Standard, see e.g. \cite[Theorem 4.2]{Gucwa2009783} and further details in \appref{lemma7}.
\end{proof}

At $\Sigma_{e,2}^{L,\text{out}}$, we then blow down to the $(x,y)$-variables and track $\Gamma_0^R$ up close to $p^R$ using regular perturbation theory. Here we perform local analysis using our blow-up approach, dividing the description of the dynamics into three essential pieces: near $\overline q^R_w$, near $\overline q^R_f$ and near $\bar E^R\cap \{y=y^R\}$. The analysis and results are very similar to \lemmaref{Plocal1}, \lemmaref{Plocal2} and \lemmaref{Plocal3} and we therefore leave out the details. But in conclusion we have that the first return mapping $\Sigma\rightarrow \Sigma$ obtained by the forward flow is a contraction for $\sigma=\varepsilon^{k/(k+1)}$ sufficiently small. The contraction mapping theorem then gives a unique fix-point and this fix point correspond to the attracting (relaxation oscillation) limit cycle. 

Next, suppose that $\eta_1\notin [\eta^L_{\text{Het}}(\mu_1),\eta^R_{\text{Het}}(\mu_1)]$. Then it follows from the analysis in \secref{sphere} that the omega limit set of any point within a sufficiently small neighborhood $U$ of $\Gamma_0$ will be contained with a small neighborhood of either $z^L$ or $z^R$. 
But then no limit cycles can exists near $\Gamma_0$ which completes the proof of \ref{item4} in \thmref{mainThm}. The proof of \ref{item5} is similar, see also \figref{norelax} for illustration of the generic cases where $\eta\ne \eta_{Het}^{L/R}$. This completes the proof of \thmref{mainThm}.

\begin{figure}[h!] 
\begin{center}
\subfigure[$\eta<\eta_{Het}^R<\eta_{Het}^L$]{\includegraphics[width=.485\textwidth]{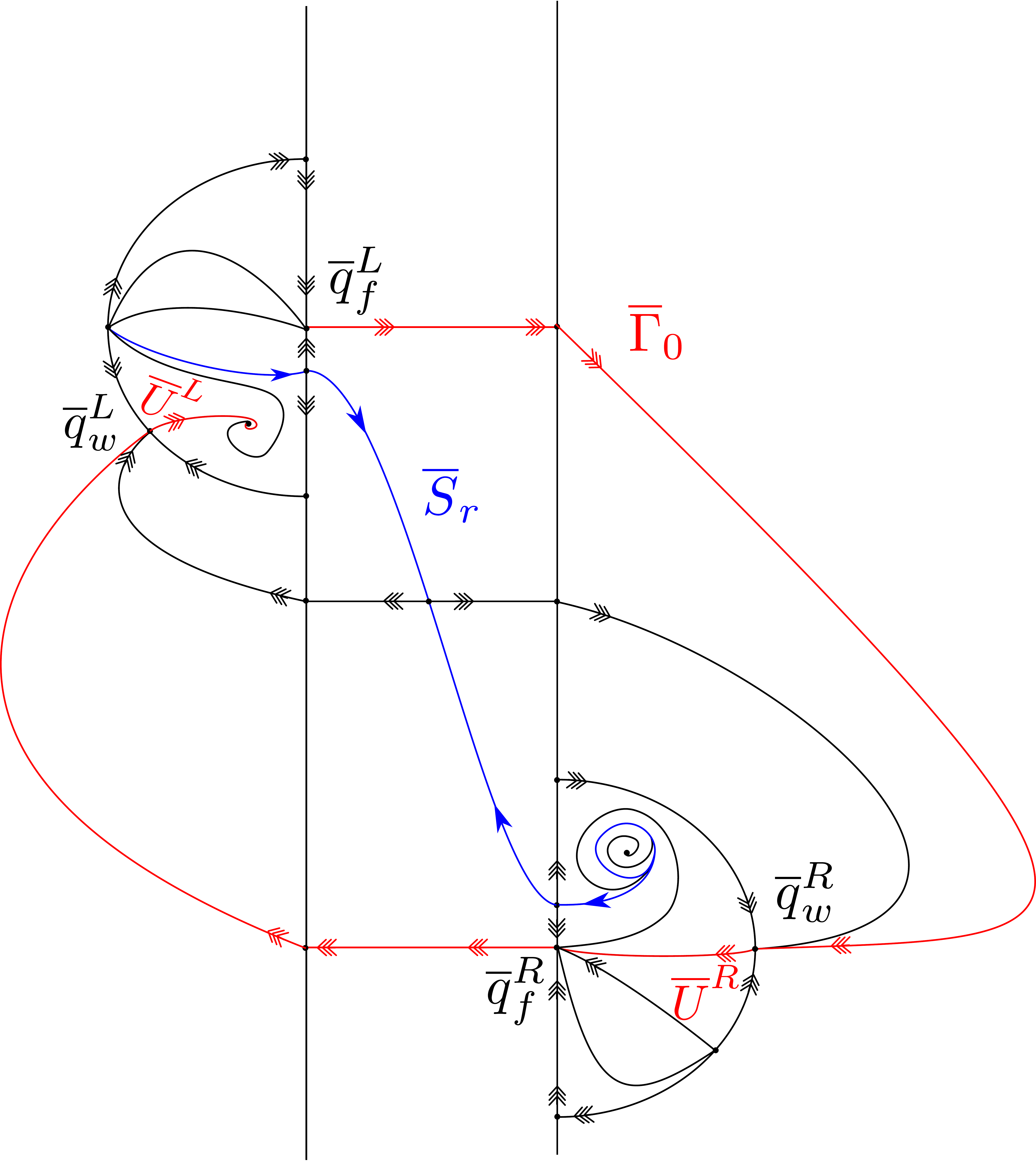}}
\subfigure[$\eta_{Het}^R<\eta<\eta_{Het}^L$]{\includegraphics[width=.485\textwidth]{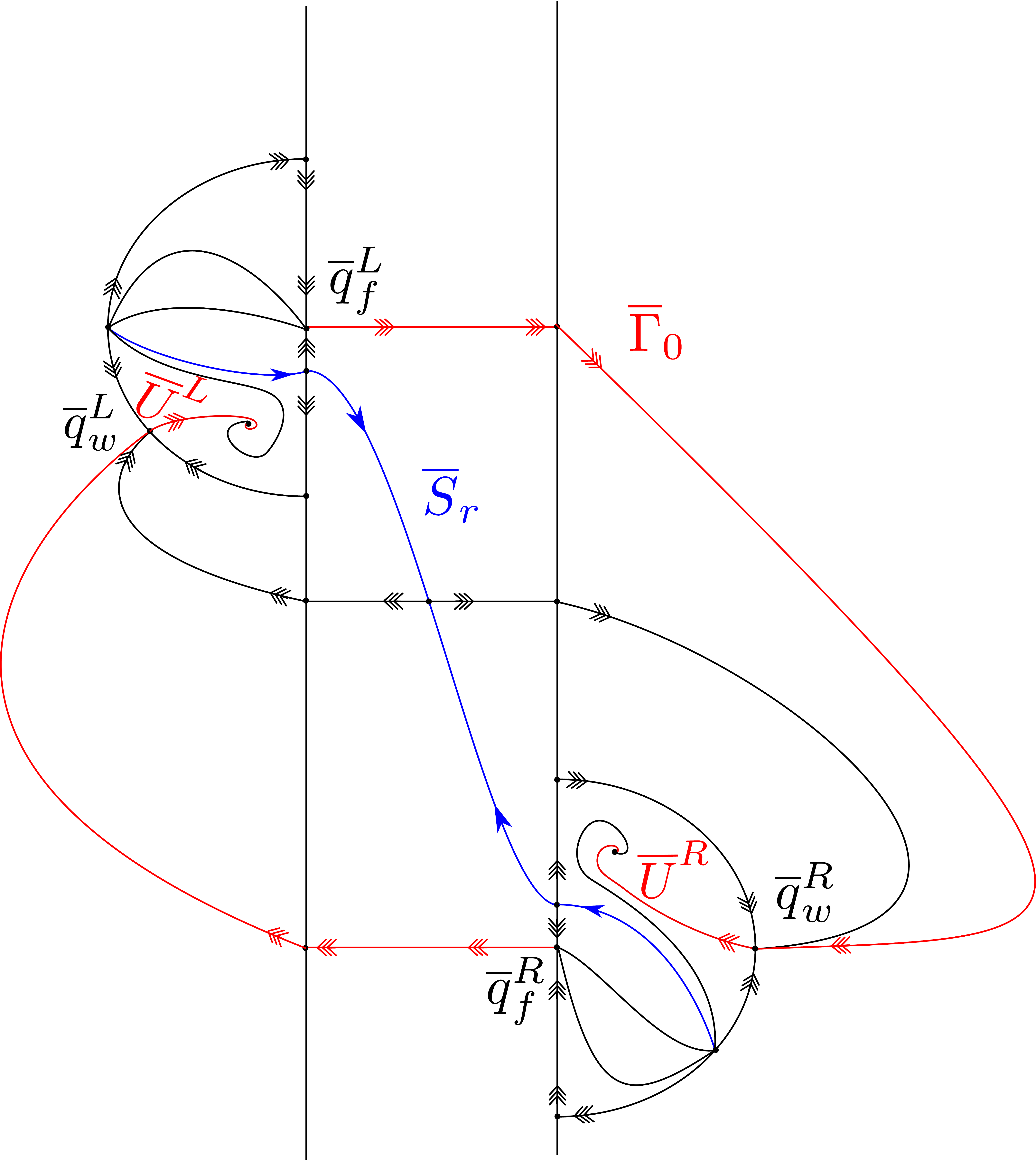}}
\subfigure[$\eta_{Het}^R<\eta_{Het}^L<\eta$]{\includegraphics[width=.485\textwidth]{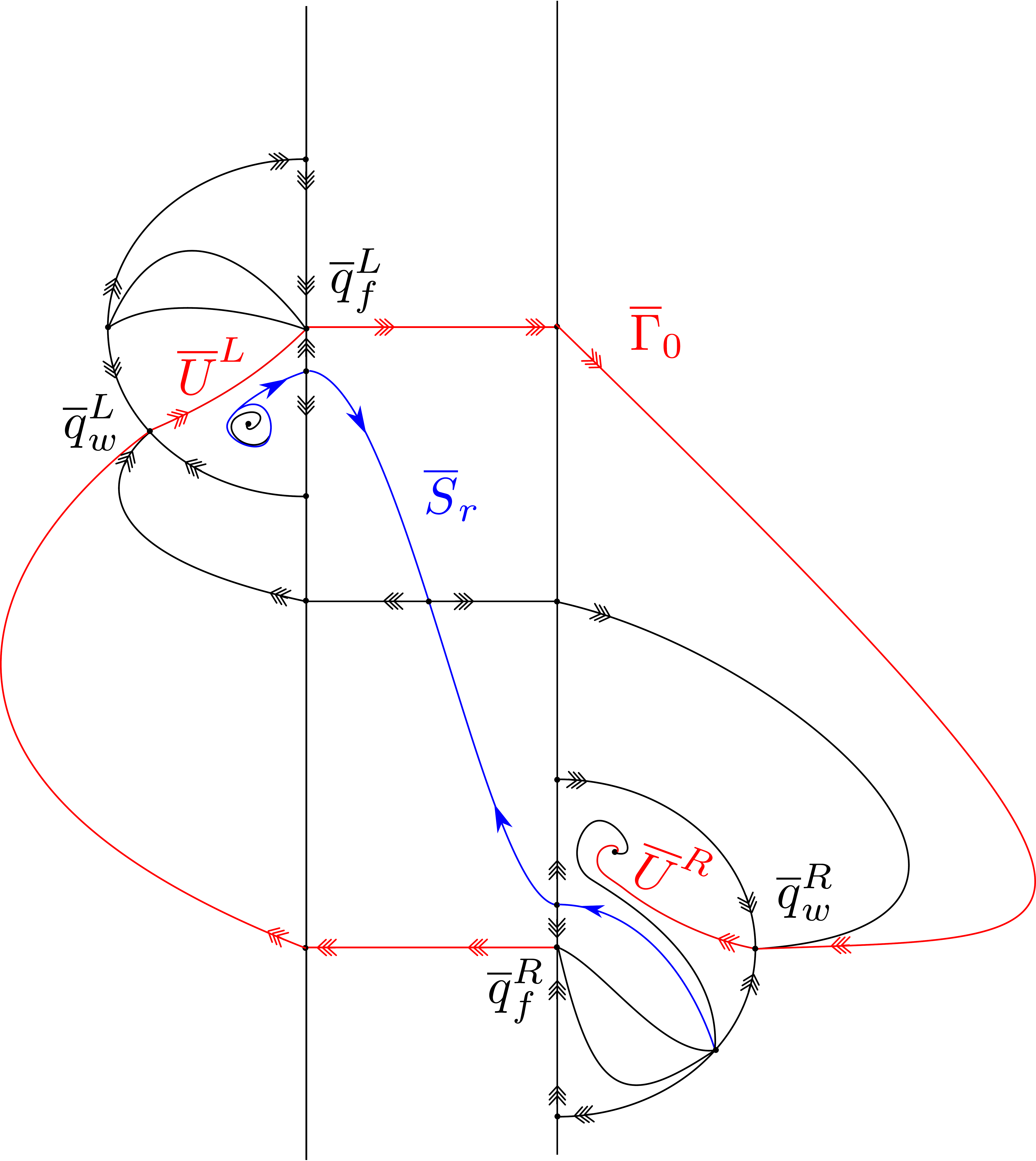}}
\end{center}
 \caption{Singular dynamics when $\eta_{Het}^R<\eta_{Het}^L$ for different values of $\eta$. There exists no $\eta$-value for which the system has a limit cycles near $\Gamma_0$.}
\figlab{norelax}
\end{figure}

\section{A generalization of \thmref{mainThm}}\seclab{discussion}

The main result in the paper can easily be generalised. In particular, the linearity of the piecewise smooth vector-fields within $x<1$ and $x>1$ is not exploited in any way. Interestingly, similar relaxation oscillations can also be found in a more generic setting consisting of smooth systems that have a pair of piecewise smooth vector-fields as pointwise limits as $\varepsilon\rightarrow 0$. In fact, for definiteness suppose that 
\begin{align}
 X_\varepsilon(x,y) = X^L(x,y)(1-\phi(\varepsilon^{-1} (x-1))) + X^R(x,y)\phi(\varepsilon^{-1} (x-1))),\eqlab{XepsilonHere}
\end{align}
such that $x=1$ is the switching manifold and $\phi$ satisfies the assumptions of (A).  
Then 
\begin{align*}
 X_\varepsilon (x,y) \rightarrow \left\{\begin{matrix}
                                              X^L(x,y),\quad x<1,\\
                                              X^R(x,y),\quad x>1,
                                             \end{matrix}\right.
\end{align*}
pointwise as $\varepsilon\rightarrow 0$ and we suppose that $X^L$ and $X^R$ are smooth everywhere. We also suppose that $X^{L/R}$ depend smoothly on a single parameter $\eta$ and satisfy the following three local and two global conditions. The following conditions describe the situation where a single stable node intersects the discontinuity set $x=1$ in a point $p^L$. 
\begin{itemize}
\item[(i)] A proper stable node $z^L$ of $X^L$ intersects the switching manifold at $p^L=(1,y^L)$ transversally under variation of the parameter $\eta$ at $\eta=\eta^L$. We suppose that the transverse intersection is such that there exists no equilibria of $X^L\vert_{x<1}$ near $p^L$ for $\eta$-values close to but larger than $\eta^L$. (We could obviously also assume that the node is an equilibrium of $X^R$). 
\item[(ii)] For $\eta=\eta^L$ the weak eigenspace of $z^L$ is not parallel with $x=1$. 
\item[(iii)] In the piecewise smooth literature there are $4$ generic, piecewise smooth unfoldings of the local situation in (i) and (ii). This was proven in \cite[Theorem 2]{hogan2016a}. We consider the case \cite[Fig 2(a), $\alpha=0$]{hogan2016a} which corresponds to the scenario that occurs in the substrate-depletion model, see also \figref{pwsDiscussion}. In particular, we assume that the slow flow on the repelling manifold, appearing upon blow-up on one side of $z^L$ (here for smaller values of $y$), is directed away (triple-headed arrows in \figref{pwsDiscussion}) from the node. 
\end{itemize}
along with two global conditions:
\begin{itemize}
 \item[(iv)] For $\eta=\eta^L$, the forward orbit $\Gamma_0^R$ of $p^L$ under the flow of $X^R\vert_{x\ge 1}$ has a first return to $x=1$ (in finite time) at a regular ``crossing'' point $q^R$ for the piecewise smooth system given by $X^R\vert_{x<1}$ and $X^L\vert_{x>1}$.
 \item[(v)] For $\eta=\eta^L$, the forward orbit $\Gamma_0^L$ of $q^R$ under the flow of $X^L\vert_{x\le 1}$ approaches $z^L$ as $t\rightarrow \infty$ along the weak eigendirection of $z^L$.
\end{itemize}
Let $\Gamma_0=\Gamma_0^L\cup \Gamma_0^R$ denote the closed curve, see \figref{pwsDiscussion}. Without loss of generality we take $\eta^L=0$. For $\eta>0$, $X^L$ and $X^R$ are just ``crossing'' near $p^L$ and $q^R$, respectively. The situation is therefore considerable easier than the one described in \thmref{mainThm}.
 We have the following
\begin{theorem}\thmlab{mainThm2}
 Suppose (i)-(v). For any $c>0$ there exists a closed interval $I\subset (0,\eta_0]$ with $\eta_0>0$ sufficiently small such that the following holds: There exists an $\varepsilon_0>0$ such that $X_\varepsilon$, for all $\varepsilon \in (0,\varepsilon_0]$ and any $\eta\in I$, has an attracting limit cycle $\Gamma_\varepsilon$ which is closer to $\Gamma_0$ than $c$ in Hausdorff distance.
\end{theorem}
\begin{proof}
Consider $\eta>0$ and let $\Sigma$ be a small section transverse to $\Gamma_0^L$ and let $\Sigma_1$ be a small section at $x=1-\xi$, $\xi>0$ but small, near $p^L$. Then due to the stable node $z^L$ of $X^L$, which is ``virtual'' for $X_0$ when $\eta>0$, we have the following: 
\begin{lemma}\lemmalab{final}
For any $L>0$ there exist a $\xi>0$, an $\varepsilon_0>0$ and an $\eta_0>0$ such that the mapping from $\Sigma$ to $\Sigma_1$, obtained by the first intersection of the forward flow, is a Lipschitz map with a Lipschitz constant less than $L$ for any $\varepsilon\in (0,\varepsilon_0]$, $\eta\in (0,\eta_0]$. 
\end{lemma}
\begin{proof}
Consider $\varepsilon=0$, then the contraction of the mapping is due to the contraction of $X^L$ towards the stable node. Taking $0<\varepsilon\ll 1$ the result then follows by regular perturbation theory.
\end{proof}
Now, fix $I\in (0,\eta_0]$ and consider any $\eta\in I$. We then blow-up $x=1$, $\varepsilon=0$ using the cylindrical blow-up \eqref{initialBlowup}. Within this blow-up, the forward flow of $\Sigma_1$ for $\varepsilon=0$ is asymptotic to a partially hyperbolic points on $\overline E^L$, see e.g. \figref{pwsblowup}(c) near $\overline p^L$. The description of the flow near this part of $\overline E^L$ is described by a local mapping with properties similar to the ones in \lemmaref{Plocal3}. In particular, by decreasing $\xi$, $\eta_0$ and $\varepsilon>0$ we can bring the restricted mapping $y\mapsto y_+$ as close to the identity mapping as desired. We can also describe the flow near $\overline E^R$ and $\overline q^R$ (since this point is now regular; in contrast to the $p^R$ for the substrate-depletion oscillator) in a similar fashion, eventually bringing us back to $\Sigma$. In total, fixing $L$ in \lemmaref{final} sufficiently small, the first return map is a contraction for all $0<\varepsilon\ll 1$ and therefore there exists an attracting limit cycle close to $\Gamma_0$. 

\end{proof}
An example of a system satisfying the assumptions (i)-(v) is the substrate-depletion model in the (unphysical) regime $\mu<0$. In this case, both $z^L$ and $z^R$ are both virtual for $\eta\in (\eta^L,\eta^R)$. Notice that $\eta^L$ and $\eta^R$ shift role in \figref{bif}(a) when $\mu<0$. The forward flow of $z^L=p^L$ for $\eta=\eta^L$ by $X^R\vert_{x>1}$ therefore also intersects $x=1$ in a regular ``crossing'' point as desired.

\begin{figure}[h!] 
\begin{center}
{\includegraphics[width=.65\textwidth]{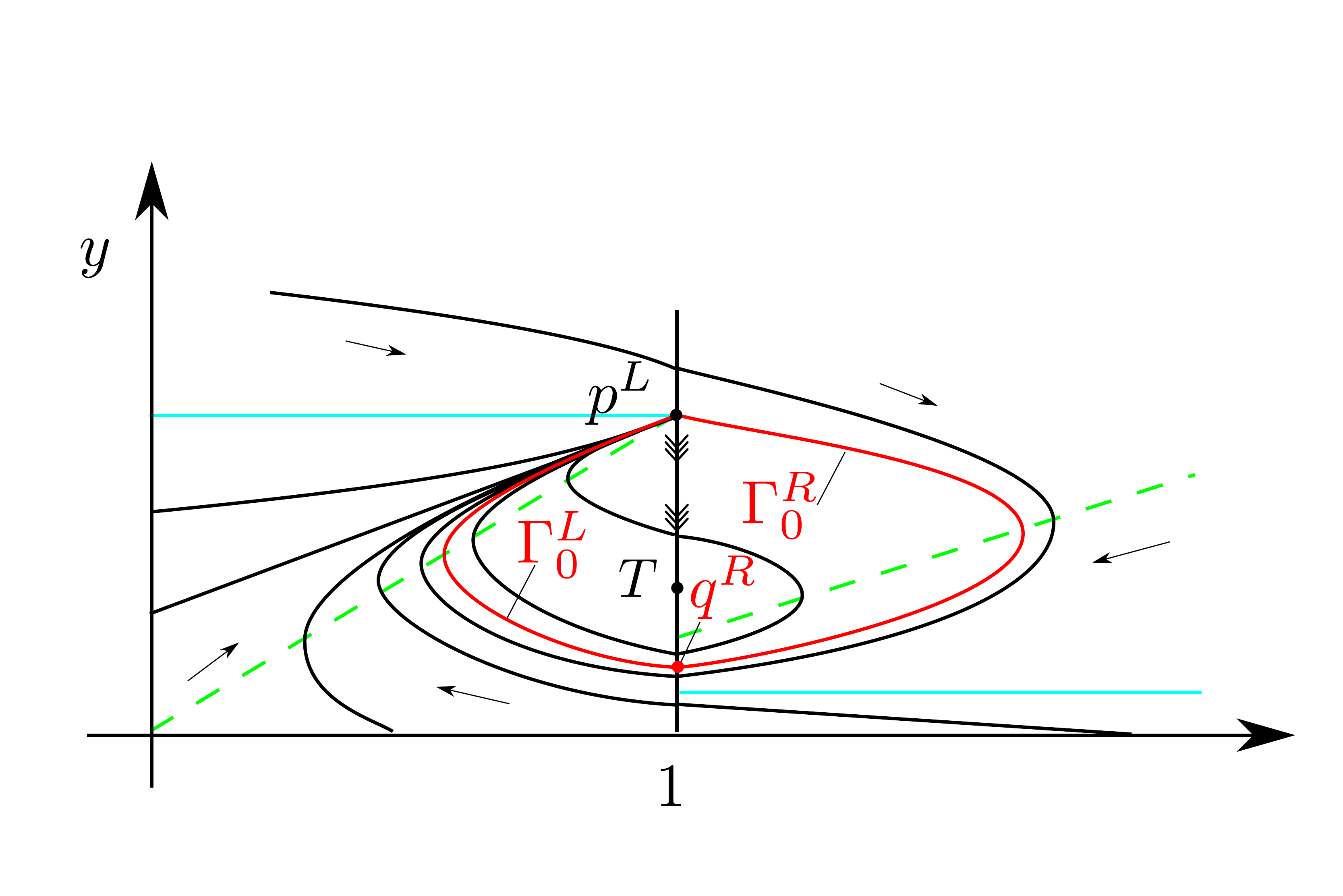}}
\end{center}
 \caption{A PWS smooth system satisfying the assumptions (i)-(v). An example is the substrate-depletion model in the (unphysical) regime $\mu<0$. The tripple-headed arrow (following the convention for Filippov systems) indicates that the Filippov sliding flow on the switching manifold is directed away from the node. This particularly relates to assumption (iii). This Filippov sliding flow on the switching manifold agrees with the slow flow on a repelling manifold obtained upon a cylindrical blow-up of $x=1,\,\varepsilon=0$, see e.g. \cite{hogan2016a}. The point $T$ is an invisible tangency point for $X^R$.  }
\figlab{pwsDiscussion}
\end{figure}

\section{Outlook}
In this paper, we have described relaxation oscillations in substrate-depletion oscillators close to the nonsmooth limit, see \thmref{mainThm}. We have also presented a simpler more general version of the existence of these oscillations in \thmref{mainThm2}. These results describe a new mechanism for global oscillations in smooth systems close to nonsmooth ones. The canard-like phenomena that occurs in these problems can be described in full details by our approach. However, we leave this as part of future work. 

From the piecewise smooth point of view, the underlying mechanism for the oscillations is a ``boundary node bifurcation'', \cite{Kuznetsov2003,hogan2016a}. Other types of ``boundary bifurcations'' can also produce oscillations. For example, a ``boundary focus bifurcation'' occurs in the classical friction oscillator problem for zero belt speed, see \cite{2019arXiv190806781U,jelbart}. The mathematical description of these bifurcations as nonsmooth perturbation problems can be treated by the blow-up approach in the present paper.
\appendix
%

\section{Proof of \propref{prop1}}\applab{proofprop1}
To prove \propref{prop1}, we first consider $X_1$ \eqref{X1} and drop the subscript $1$:
\begin{align}
 \dot r &=-\frac{1}{k+1}r\left(r^{k+1}-F(\delta^{k+1},y)\right),\eqlab{X1New}\\
 \dot y&=r^{k+1} \left(F(\delta^{k+1},y)+r^k\delta^k \left(\eta-\mu y\right)\right),\nonumber\\
 \dot \delta &=\frac{1}{k+1}r \delta^{k+1} \left(r^{k+1}-F(\delta^{k+1},y)\right).\nonumber
\end{align}
 with
 \begin{align*}
  F(\delta^{k+1},y) = 1-(\alpha+\beta \delta^{k(k+1)}\phi^L(\delta^{k+1}))y.
 \end{align*}
To study $\widehat X$ on the sphere $\overline S^L$, recall \eqref{widehatXi}, we then work in the separate charts:
\begin{align}
\bar r = 1:\quad & r = \rho_1^k, & y &= y^L+\rho_1^{k(k+1)} y_1,&\delta & =\rho_1 \delta_1,\eqlab{barR1pL}\\
\bar \delta = 1:\quad & r = \rho_2^kr_2, & y &= y^L+\rho_2^{k(k+1)}y_2,&\delta & =\rho_2,\eqlab{bard1pL}\\
\bar y = 1:\quad & r = \rho_3^kr_3, & y &= y^L+\rho_3^{k(k+1)},&\delta & =\rho_3 \delta_3,\eqlab{bary1pL}\\
\bar y = -1:\quad & r = \rho_4^kr_4, & y &= y^L-\rho_4^{k(k+1)},&\delta & =\rho_4 \delta_4,\eqlab{baryN1pL}
\end{align}
We then describe $\widehat X_{\rho=0}$ by setting $\rho_i=0$ in each of these corresponding charts. The charts cover the quarter-sphere $\overline S^L$ completely. When the charts overlap we can apply coordinate changes. We will use the following coordinate changes in the sequel.
\begin{align}
 K_{41}:\, (\rho_1,y_1,\delta_1)\mapsto \rho_4 &= (-y_1)^{1/(k(k+1))}\rho_1,\eqlab{K41}\\
 r_4&=(-y_1)^{-1/(k+1)},\nonumber\\
 \delta_4 &= (-y_1)^{-1/(k(k+1))} \delta_1,\nonumber\\
  K_{21}:\, (\rho_1,y_1,\delta_1)\mapsto \rho_2 &= \delta_1\rho_1,\eqlab{K21}\\
 r_2&=\delta_1^{-k},\nonumber\\
 y_2 &= \delta_1^{-k(k+1)} y_1,\nonumber\\
  K_{42}:\,(\rho_2,r_2,y_2)\mapsto \rho_4 &= (-y_2)^{1/(k(k+1))}\rho_2,\eqlab{K42}\\
 r_4&=(-y_2)^{-1/(k+1)}r_2,\nonumber\\
 \delta_4&=(-y_2)^{-1/(k(k+1))},\nonumber
 \end{align}
for $y_1<0$, $\delta_1>0$ and $y_2<0$, respectively. We set $K_{ij}=K_{ji}^{-1}$.

\subsubsection*{Chart $\bar r=1$}
Inserting \eqref{barR1pL} into \eqref{X1New} gives
\begin{align}
 \dot \rho_1 &= -\frac{1}{k(k+1)}\rho_1 \left(1-F_1((\rho_1\delta_1)^{k+1},\delta_1^{k(k+1)},y_1)\right),\eqlab{XbarR1pL}\\
 \dot y_1 &=F_1((\rho_1\delta_1)^{k+1},\delta_1^{k(k+1)},y_1)+\delta_1^k\left(\eta-\mu (y^L+\rho_1^{k(k+1)} y_1)\right)+\left(1-F_1((\rho_1\delta_1)^{k+1},\delta_1^{k(k+1)},y_1)\right)y_1,\nonumber\\
 \dot \delta_1 &=\frac{1}{k}\delta_1 \left(1-F_1((\rho_1\delta_1)^{k+1},\delta_1^{k(k+1)},y_1)\right),\nonumber
\end{align}
after division of the right hand side by the common factor $\rho_1^{k(k+1)}$. Here
\begin{align*}
 F_1((\rho_1\delta_1)^{k+1},\delta_1^{k(k+1)},y_1)=-\frac{\beta}{\alpha}\delta_1^{k(k+1)} \phi^L(\rho_1^{k+1}\delta_1^{k+1}) - \left(\alpha+\beta \rho_1^{k(k+1)} \delta_1^{k(k+1)} \phi^L(\rho_1^{k+1}\delta_1^{k+1})\right)y_1.
\end{align*}
Setting $\rho_1=0$ then gives 
\begin{align}
 \dot y_1 &=F_1(0,\delta_1^{k(k+1)},y_1)+\delta_1^k\left(\eta-\eta^L(\mu))\right)+\left(1-F_1(0,\delta_1^{k(k+1)},y_1)\right)y_1,\eqlab{XbarR1pLrho0}\\
 \dot \delta_1 &=\frac{1}{k}\delta_1 \left(1-F_1(0,\delta_1^{k(k+1)},y_1)\right),\nonumber
\end{align}
with 
\begin{align*}
 F_1(0,\delta_1^{k(k+1)},y_1) = -\alpha y_1-\frac{\beta}{\alpha}\delta_1^{k(k+1)}\phi^L(0). 
\end{align*}
Notice that the equations within $\rho_1=0$ only depend upon $\eta$ and $\mu$ through $\eta-\eta^L(\mu)$. This is (obviously) true in all charts. 
\begin{lemma}\lemmalab{pLChartr1}
We have
\begin{enumerate}[label=(\alph*)]
%
 \item \label{item1PlChartr1} $q_{s,1}^L:\,(y_1,\delta_1) = (0,0)$ is a hyperbolic unstable node of \eqref{XbarR1pLrho0}.
 \item $q_{w,1}^L:\,(y_1,\delta_1) = (-\alpha^{-1}(1-\alpha),0)$ is a hyperbolic saddle with stable manifold along the invariant $\delta_1$-axis and a local unstable manifold $U_{loc,1}^L\equiv W^u_{loc}(q_{w,1}^L)$ of the following form:
 \begin{align*}
  U_{loc,1}^L:\, y_1 &= -\alpha^{-1}(1-\alpha) +\delta_1^k\left(\eta-\eta^L(\mu)+\delta_1^{k^2} \left(-\frac{\beta\phi^L(0)}{\alpha^2}+\delta_1^k m_1(\delta_1,\eta-\eta^L(\mu))\right)\right),\\
  \delta_1 &\in [0,c],
 \end{align*}
with $c>0$ sufficiently small and $m_1$ smooth. 
\item  \label{item3PlChartr1} For $\eta<\eta^L(\mu)$, $\delta_1$ increases along $U^L_{loc,1}$ while $y_1$ decreases. For $\eta \ge \eta^L(\mu)$, $y_1$ and $\delta_1$ both increase along $U^L_{loc,1}$. 
\item \label{item4PlChartr1} For $\eta<\eta^L(\mu)$ there exists a separate equilibrium $z_1^L$ of \eqref{XbarR1pLrho0} with coordinates
\begin{align*}
 (y_1,\delta_1) = \left( -\alpha^{-1} -\alpha^{-2} \beta \phi^L(0) (\eta^L(\mu)-\eta)^{-(k+1)},(\eta^L-\eta)^{-1/k}\right).
\end{align*}
$z_1^L$ undergoes a sub-critical Hopf bifurcation for $\eta=\eta_H^L(\mu)$ where
\begin{align*}
 \eta_H^L(\mu) = \eta^L(\mu)-\left(\frac{\beta k\phi^L(0)}{\alpha(\alpha+1)}\right)^{1/(k+1)}.
\end{align*}
$z_1^L$ is hyperbolic and attracting (repelling) for $\eta<\eta_H^L(\mu)$ ($\eta_H^L(\mu)<\eta<\eta^L(\mu)$), respectively. Therefore, the limit cycles born in the Hopf bifurcation appear locally for $\eta<\eta_H^L(\mu)$.
\item \label{item5PlChartr1} There exists an $\eta_0(\mu)<\eta^L$ such that $z_1^L$ is a stable node that attracts $U^L_1$ for all $\eta\le \eta_0(\mu)$.
\end{enumerate}

\end{lemma}
\begin{proof}
 Items \ref{item1PlChartr1}-\ref{item4PlChartr1} are straightforward calculations. In particular, for \ref{item4PlChartr1} we note that the Jacobian at $z_1^L$ has determinant $\alpha>0$ and trace
 \begin{align*}
  \text{tr}  = \alpha^{-1} \left(k\beta \phi^L(0) (\eta^L-\eta)^{-(k+1)} -\alpha(1+\alpha)\right).
 \end{align*}
Setting $\text{tr}  =0$ therefore gives a Hopf bifurcation at $\eta=\eta_H^L$. Lengthly calculations (which turned out to be easier to do in chart $\bar \delta=1$) leads to a complicated expression for the Lyapunov coefficient \cite{perko2001a}:
\begin{align*}
 L = \frac{1}{16}\left(\eta^L-\eta_H^L\right)^{k+1} \left( \left( \alpha+\beta \right) ^{3}{k}^{2}+\left( 3\,\alpha+3\,\beta+2 \right)  \left( \alpha+\beta
 \right) ^{2}k+2\left( \alpha+\beta+1 \right) 
 \left( \alpha+\beta \right) ^{2}\right).
\end{align*}
However, all terms are positive and the Hopf bifurcation is therefore sub-critical. 

For item \ref{item5PlChartr1} we set $\eta=\eta^L-\nu^{-k}$, $\delta_1 = \nu \tilde \delta_1$ and consider $\nu\rightarrow 0$. Inserting this into \eqref{XbarR1pLrho0} and setting $\nu=0$ gives the following system:
\begin{align}
\dot y_1 &=-\alpha y_1-\tilde \delta_1^k+ \left(1+\alpha y_1\right) y_1,\eqlab{scaledsystem}\\
\frac{d}{dt} \left(\tilde{\delta}_1^k\right) &= \tilde \delta_1^k \left(1+\alpha y_1\right).\nonumber
\end{align}
Notice the system is linear in terms of $\tilde \delta_1^k$. Also, in this scaling $z_1^L$ becomes $(y_1,\tilde \delta_1) = (-\alpha^{-1},1)$. The eigenvalues of the linearization is $-1$ and $-\alpha$ and the equilibrium is therefore a stable node. On the other hand, the equilibrium $\tilde q_{w,1}^L$ given by $(y_1,\tilde \delta_1) = (-\alpha^{-1}(1-\alpha),0)$ is a saddle. Let $B=-\alpha^{-1} (1-\alpha)$ be its $y_1$-value. Then the unstable manifold $U^L_1$ is bounded and contained within $y_1\le B$ since $\dot y_1<0$ along $y_1=B$, $\tilde \delta_1>0$.

Now, by eliminating $\tilde \delta^k$ we can write the scaled system \eqref{scaledsystem} as a second order Lienard system
\begin{align*}
 \ddot y_1 +f(y_1) \dot y_1 +g(y_1) = 0, 
\end{align*}
where
\begin{align*}
 f(y_1) &=-3(1+\alpha y_1)+1+\alpha,\\
 g(y_1)&=y_1(1+\alpha y_1)(\alpha y_1+1-\alpha).
\end{align*}
Let
\begin{align*}
 F(y_1) = \int_{-\alpha^{-1}}^{y_1} f(s) ds = -\frac12 \alpha(\alpha y_1+1)(\alpha (3y_1-2)+1).
\end{align*}
Notice that $-\alpha^{-1}<B<0$ by assumption. 
Therefore $g(y_1) (y_1+\alpha^{-1})>0$ for $y_1\in (-\infty,-\alpha^{-1})\cup (-\alpha^{-1},B)$ and $F(y_1)<0$ for $y_1<-\alpha^{-1}$ and $F(y_1)>0$ for $y_1\in (-\alpha^{-1},B)$. Therefore by Cherkas' theorem, see e.g. \cite[Theorem 3 p. 265]{perko2001a}, it follows that the system
 \eqref{scaledsystem} does not contain limit cycles within the strip $y_1\in (-\infty,B]$. Since $U_1^L$ is bounded it must therefore limit to the stable node. We perturb this into $\nu \ll 1$ to obtain the desired result. 
\end{proof}

\subsubsection*{Chart $\bar \delta=1$}
Inserting \eqref{bard1pL} into \eqref{X1New} gives
\begin{align}
 \dot \rho_2 &= \frac{1}{k+1}\rho_2 \left(r_2^{k+1}-F_2(\rho_2^{k+1},y_2)\right),\eqlab{pLdelta1}\\
 \dot r_2 &=-r_2 \left(r_2^{k+1}-F_2(\rho_2^{k+1},y_2)\right),\nonumber\\
 \dot y_2 &=r_2^{k+1} \left(F_2(\rho_2^{k+1},y_2)+r_2^k \left(\eta-\mu \left(y^L+\rho_2^{k(k+1)} y_2\right)\right)\right)-k y_2 \left(r_2^{k+1}-F_2(\rho_2^{k+1},y_2)\right),\nonumber
\end{align}
after division of the right hand side by the common factor $\rho_2^{k(k+1)}$. Here
\begin{align}
 F_2(\rho_2^{k+1},y_2) = -\frac{\beta}{\alpha}\phi^L(\rho_2^{k+1})-\left(\alpha+\beta \phi^L(\rho_2^{k+1})\rho_2^{k(k+1)}\right)y_2.\eqlab{pLF2}
\end{align}
Setting $\rho_2=0$ then gives 
\begin{align}
 \dot r_2 &=-r_2 \left(r_2^{k+1}-F_2(0,y_2)\right),\eqlab{Xbard1pLrho0}\\
  \dot y_2 &=r_2^{k+1} \left(F_2(0,y_2)+r_2^k \left(\eta-\eta^L(\mu)\right)\right)-k y_2 \left(r_2^{k+1}-F_2(0,y_2)\right),\nonumber
\end{align}
with 
\begin{align*}
 F_2(0,y_2) = -\alpha y_2-\frac{\beta}{\alpha}\phi^L(0). 
\end{align*}
\begin{lemma}\lemmalab{pLChartd1}
We have
\begin{enumerate}[label=(\alph*)]
 \item $q_{f,2}^L:\,(r_2,y_2) = (0,0)$ is a hyperbolic stable node of \eqref{Xbard1pLrho0}.
 \item $q_{r,2}^L:\,(r_2,y_2) = (0,-\alpha^{-2}\beta \phi^L(0))$ is a partially hyperbolic equilibrium with a strong unstable manifold along the invariant $r_2$-axis and local center manifold $C^L_{loc,2}\equiv W^c_{loc}(q_{r,2}^L)$ of the following form:
 \begin{align*}
  C^L_{loc,2}:\, y_2 &= -\alpha^{-2}\beta \phi^L(0)-\alpha^{-1} r_2^{k+1}\bigg(1\\
  &+r_2^k \frac{\alpha^2}{\beta k\phi^L(0)}\left(\eta-\eta^L(\mu)+r_2\right)(1+r_2^2m_2(r_2,\eta-\eta^L(\mu)))\bigg),\quad r_2 \in [0,c],
 \end{align*}
with $c>0$ sufficiently small and $m_2$ smooth. 
\item Let $\eta<\eta^L(\mu)$. Then $q_{r,2}^L$ is a nonhyperbolic saddle and $C^L_{loc,2}\cap \{r_2\in [0,\nu]\}$  is a unique local stable manifold for $\nu$ ($\le c$) sufficiently small. 
\item Let $\eta>\eta^L(\mu)$. Then $q_{r,2}^L$ is a nonhyperbolic unstable node and $C^L_{loc,2}$ is non-unique along which $r_2$ increases while $y_2$ decreases. 
\item Consider any $\eta\in  [\eta^L(\mu)- \nu_0(\mu),\eta^L(\mu))$ with $\nu_0(\mu)>0$ sufficiently small. Then the equilibrium $z_2^L$ with coordinates 
\begin{align*}
 (r_2,y_2)=(\eta^L(\mu)-\eta,-\alpha^{-2}\beta \phi^L(0)-\alpha^{-1}(\eta^L(\mu)-\eta)^{k+1}),
\end{align*}
which is just the image of $z_1^L$ under the coordinate change $K_{21}$ \eqref{K21}, is an unstable node connected with the local center manifold $C^L_{loc,2}$. 
\end{enumerate}
\end{lemma}
\begin{proof}
 Straightforward calculations. 
\end{proof}
\subsubsection*{Chart $\bar y=1$}
Inserting \eqref{bary1pL} into \eqref{X1New} gives
\begin{align*}
 \dot \rho_3 &=\frac{1}{k(k+1)}\rho_3 r_3^{k+1}\left(F_3((\rho_3\delta_3)^{k+1},\delta_3^{k(k+1)})+r_3^k\delta_3^k (\eta-\mu ( y^L +\rho_3^{k(k+1)}))\right),\\
 \dot r_3 &=-\frac{1}{k+1}r_3 \left(r_3^{k+1}-F_3((\rho_3\delta_3)^{k+1},\delta_3^{k(k+1)})\right)\\
 &-\frac{1}{k+1}r_3^{k+2} \left(F_3((\rho_3\delta_3)^{k+1},\delta_3^{k(k+1)})+r_3^k\delta_3^k (\eta-\mu ( y^L +\rho_3^{k(k+1)}))\right),\\
 \dot \delta_3 &= -\frac{1}{k+1}\delta_3 \left(r_3^{k+1}-F_3((\rho_3\delta_3)^{k+1},\delta_3^{k(k+1)})\right)\\
 &-\frac{1}{k(k+1)}\delta_3 r_3^{k+1} \left(F_3((\rho_3\delta_3)^{k+1},\delta_3^{k(k+1)})+r_3^k\delta_3^k (\eta-\mu ( y^L +\rho_3^{k(k+1)}))\right),
\end{align*}
after division of the right hand side by the common factor $\rho_3^{k(k+1)}$. Here
\begin{align*}
 F_3((\rho_3\delta_3)^{k+1},\delta_3^{k(k+1)}) = -\alpha -\beta \rho_3^{k(k+1)}\delta_3^{k(k+1}) \phi^L((\rho_3 \delta_3)^{k+1})-\frac{\beta}{\alpha}\delta_3^{k(k+1)} \phi^L((\rho_3 \delta_3)^{k+1})
\end{align*}
Setting $\rho_3=0$ then gives 
\begin{align}
\dot r_3 &=-\frac{1}{k+1}r_3 \left(r_3^{k+1}-F_3(0,\delta_3^{k(k+1)})\right)-\frac{1}{k+1}r_3^{k+2} \left(F_3(0,\delta_3^{k(k+1)})+r_3^k\delta_3^k (\eta-\eta^L(\mu))\right),\eqlab{Xbary1pLrho0}\\
 \dot \delta_3 &= \frac{1}{k+1}\delta_3 \left(r_3^{k+1}-F_3(0,\delta_3^{k(k+1)})\right)-\frac{1}{k(k+1)}\delta_3 r_3^{k+1} \left(F_3(0,\delta_3^{k(k+1)})+r_3^k\delta_3^k (\eta-\eta^L(\mu))\right),\nonumber
\end{align}
with
\begin{align*}
 F_3(0,\delta_3^{k(k+1)}) = -\alpha -\frac{\beta}{\alpha}\delta_3^{k(k+1)} \phi^L(0).
\end{align*}
\begin{lemma}
 $a_3^L:\,(r_3,\delta_3) = (0,0)$ is a unique equilibrium of \eqref{Xbary1pLrho0}. It is a hyperbolic saddle. The stable manifold is along the invariant $r_4$-axis while the unstable manifold is along the invariant $\delta_4$-axis.  
\end{lemma}
\begin{proof}
 Straightforward calculations.
\end{proof}

\subsubsection*{Chart $\bar y=-1$}
Inserting \eqref{baryN1pL} into \eqref{X1New} gives
\begin{align*}
 \dot \rho_4 &=-\frac{1}{k(k+1)}\rho_4 r_4^{k+1}\left(F_4((\rho_4\delta_4)^{k+1},\delta_4^{k(k+1)})+r_4^k\delta_4^k (\eta-\mu ( y^L -\rho_4^{k(k+1)}))\right),\\
 \dot r_4 &=-\frac{1}{k+1}r_4 \left(r_4^{k+1}-F_4((\rho_4\delta_4)^{k+1},\delta_4^{k(k+1)})\right)\\
 &+\frac{1}{k+1}r_4^{k+2} \left(F_4((\rho_4\delta_4)^{k+1},\delta_4^{k(k+1)})+r_4^k\delta_4^k (\eta-\mu ( y^L -\rho_3^{k(k+1)}))\right),\\
 \dot \delta_4 &= -\frac{1}{k+1}\delta_4 \left(r_4^{k+1}-F_4((\rho_4\delta_4)^{k+1},\delta_4^{k(k+1)})\right)\\
 &+\frac{1}{k(k+1)}\delta_4 r_4^{k+1} \left(F_4((\rho_4\delta_4)^{k+1},\delta_4^{k(k+1)})+r_4^k\delta_4^k (\eta-\mu ( y^L -\rho_3^{k(k+1)}))\right),
\end{align*}
after division of the right hand side by the common factor $\rho_4^{k(k+1)}$. Here
\begin{align*}
 F_4((\rho_4\delta_4)^{k+1},\delta_4^{k(k+1)}) = \alpha +\beta \rho_4^{k(k+1)}\delta_4^{k(k+1)} \phi^L((\rho_4 \delta_4)^{k+1})-\frac{\beta}{\alpha}\delta_4^{k(k+1)} \phi^L((\rho_4 \delta_4)^{k+1})
\end{align*}
Setting $\rho_4=0$ then gives 
\begin{align}
\dot r_4 &=-\frac{1}{k+1}r_4 \left(r_4^{k+1}-F_4(0,\delta_4^{k(k+1)})\right)+\frac{1}{k+1}r_4^{k+2} \left(F_4(0,\delta_4^{k(k+1)})+r_4^k\delta_4^k (\eta-\eta^L(\mu))\right),\eqlab{Xbary1pLrho0New}\\
 \dot \delta_4 &= \frac{1}{k+1}\delta_4 \left(r_4^{k+1}-F_4(0,\delta_4^{k(k+1)})\right)+\frac{1}{k(k+1)}\delta_4 r_4^{k+1} \left(F_4(0,\delta_4^{k(k+1)})+r_4^k\delta_4^k (\eta-\eta^L(\mu))\right),\nonumber
\end{align}
with
\begin{align*}
 F_4(0,\delta_4^{k(k+1)}) = \alpha -\frac{\beta}{\alpha}\delta_4^{k(k+1)} \phi^L(0).
\end{align*}
In the following, let $\widehat X_4$ denote the right hand side of \eqref{Xbary1pLrho0New}. 
In this chart, we rediscover 
\begin{align}
 q_{w,4}^L &= K_{41}(q_{w,1}^L):\,(r_4,\delta_4) = \left(\left(\frac{\alpha}{1-\alpha}\right)^{1/(k+1)},0\right),\eqlab{qw4L}\\
 q_{r,4}^L &= K_{42}(q_{r,2}^L):\,(r_4,\delta_4) = \left(0,\left(\frac{\alpha^2}{\beta\phi^L(0)}\right)^{1/(k(k+1))}\right),\nonumber
\end{align}
as a hyperbolic saddle and a nonhyperbolic saddle. Let $U^L_4 = K_{41}(U^L_1\cap \{y_1<0\})$ and $C^L_4=K_{42}(C_2^L\cap \{y_2<0\})$ be the corresponding unstable manifold and center manifold, respectively, in the present chart. Recall that the local center manifold is unique for $\eta<\eta^L(\mu)$ as a stable manifold of $q_{r,4}^L$. 
\begin{lemma}\lemmalab{essentiallem}
We have
\begin{enumerate}[label=(\alph*)]
 \item \label{item1Here} $b_4^L:\,(r_4,\delta_4) = (0,0)$ is a hyperbolic saddle with stable manifold along the invariant $\delta_4$-axis and unstable manifold along the invariant $r_4$-axis.
 \item \label{item2Here} There exists a unique number $\eta_{Het,0}^L$ (independent of $\mu$)  such that if $\eta_{\text{Het}}^L(\mu) = \eta^L(\mu)+\eta_{\text{Het},0}^L$ then the following holds: $\eta_{\text{Het}}^L(\mu)\in (\eta_0(\mu),\eta^L(\mu)-\nu_0(\mu))$, with $\eta_0$ and $\nu_0$ from \lemmaref{pLChartr1} and \lemmaref{pLChartd1}, respectively, and for $\eta=\eta_{Het}^L(\mu)$ the hyperbolic saddle $q_{w,4}^L$ is heteroclinic with the nonhyperbolic saddle $q_{r,4}^L$: $U^L_4= C^L_4$.
 \item \label{item3Here} For $\eta\in (\eta_{Het}^L(\mu),\eta_{Het}^L(\mu)+\nu_1(\mu)]$ with $\nu_1(\mu)$ sufficiently small, there exists a family of hyperbolic and repelling limit cycles. These limit cycles approach the heteroclinic cycle: $q_{w,4}^L$ connecting $q_{r,4}^L$, $q_{r,4}^L$ connecting $b_4^L$ along the invariant $\delta_4$-axis, $b_4^L$ connecting $q_{w,4}^L$ along the invariant $r_4$-axis, as $\eta\rightarrow \eta_{Het}^L(\mu)^+$.  
 \end{enumerate}
\end{lemma}
\begin{proof}
Item \ref{item1Here} is straightforward. For \ref{item2Here} we first prove existence of a heteroclinic connection. For this consider $\eta=\eta_0(\mu)$. Then by \lemmaref{pLChartr1} item \ref{item5PlChartr1} , we have that $U_4^L$ is bounded in the present chart and limits to the stable node $z_4^L=K_{41}(z_1^L)$. Following the analysis in the separate charts, it then also follows that $C^L_4$ limits to the unstable node $q_{s,1}$ in chart $\bar r=1$ in backwards time. All other unstable equilibria are saddles with unstable manifolds along edges of the quarter sphere. See \figref{hetero}(a). Next consider $\eta=\eta^L-\nu_0$. Then by \lemmaref{pLChartd1} the unique center/stable manifold of $q_{r,4}^L$ is backwards asymptotic to $z_4^L$. Following the analysis in the separate charts, it then also follows that the unstable manifold $U^L_4$ limits to the stable node $q_{f,2}$ in chart $\bar \delta=1$ in forward time. See \figref{hetero}(c). By continuity there must exist a $\eta_{Het}^L$ such that for $\eta=\eta_{Het}^L$ we have $U^L_4=W_c^4$, see \figref{hetero}(b) and \appref{existence} for further details. Notice again that the equations only depend upon $\eta$ and $\mu$ as $\eta-\eta^L(\mu)$ and therefore $\eta_{Het,0}^L=\eta_{Het}^L-\eta^L$ is independent of $\mu$ as claimed. 
 
 Now for uniqueness we perform a Melnikov computation. Let 
 \begin{align}
 t\mapsto (\tilde r_4(t),\tilde \delta_4(t)),\eqlab{heteroclinicPara}
 \end{align}be a parametrization by time $t\in \mathbb R$ of a heteroclinic connection for $\eta=\eta_{Het}^L$. We suppose without loss of generality that $(\tilde r_4(0),\tilde \delta_4(0))\in \Sigma$, the section defined in \appref{existence}, see also  \figref{hetero}(b). Then using \lemmaref{pLChartr1} and \lemmaref{pLChartd1} and information about the nullclines, we can show that the functions $\tilde r_4(t)$ and $\tilde \delta_4(t)$ are strictly monotone and satisfy 
 \begin{align}
 \tilde \delta_4'(t)>0, \,\tilde r_4'(t)<0\quad \mbox{for all $t\in \mathbb R$}.\eqlab{hetmono}
 \end{align}
 See \figref{hetero}(b) and \appref{uniqueness} for further details. Therefore 
 \begin{align*}
   \widehat X_4(\tilde r_4(t),\tilde \delta_4(t)) \wedge \partial_\eta \widehat X_4(\tilde r_4(t),\tilde \delta_4(t)) &= \begin{pmatrix} 
                                                         -\tilde \delta_4'(t)\\
                                                         \tilde r_4'(t)
                                                          \end{pmatrix}\cdot \begin{pmatrix} 
  \frac{1}{k+1}\tilde r_4^{2k+2}\tilde \delta_4^k\\
  \frac{1}{k+1}\tilde r_4^{2k+1}\tilde \delta_4^{k+1}
  \end{pmatrix} \\
  &= \frac{1}{k+1}\tilde r_4^{2k+1} \tilde \delta_4^k \left(-\tilde \delta_4'(t)\tilde r_4+\tilde r_4'(t)\tilde \delta_4\right)<0,
\end{align*}
for all $t\in \mathbb R$
using \eqref{hetmono}. As a consequence, the Melnikov integral
 \begin{align}
  M& = \int_{-\infty}^\infty \exp\left({-\int_{0}^t \text{div}\,\widehat X_4(\tilde r_4(s),\tilde \delta_4(s))ds}\right) \widehat X_4(\tilde r_4(t),\tilde \delta_4(t)) \wedge \partial_\eta \widehat X_4(\tilde r_4(t),\tilde \delta_4(t)) dt,\nonumber
 \end{align}
is also negative:
\begin{align}
 M<0.\eqlab{Melnikov}
\end{align}
This means that the heteroclinic connection is transverse with respect to $\eta$ at $\eta=\eta_{Het}^L$. But \eqref{Melnikov} also implies that the heteroclinic connection is unique. To show this, we first recall the following about the Melnikov function. Let $\tilde u(\eta)$ and $\tilde c(\eta)$ be the first intersection points of $U^L_4$ and $C^L_4$, respectively, with the section $\widetilde \Sigma$ which goes through $(\tilde r_4(0),\tilde \delta_4(0))$ and is spanned by the vector $(-\tilde \delta_4'(0),\tilde r_4'(0))$ for $\eta$ sufficiently close to $\eta_{Het}^L$. Let $d(\eta)=(-\tilde \delta_4'(0),\tilde r_4'(0))\cdot (\tilde u(\eta)-\tilde c(\eta))$. Then 
$d'(\eta_{Het}^L) = M$,
see e.g. \cite{perko2001a}. In \appref{existence}, we denote by $u(\eta)$ and $c(\eta)$ the intersection of $U^L_4$ and $C^L_4$, respectively, with the section $\Sigma$. Since $\Sigma$ is vertical, being parallel to the $\delta_4$-axis,  see \figref{hetero}(b), we use for simplicity the same symbols $u(\eta)$ and $c(\eta)$ for the $\delta_4$-coordinates of $u(\eta)$ and $c(\eta)$. Then by the orientation of the section $\widetilde \Sigma$ described in \eqref{hetmono}, it follows that the sign of $d(\eta)$ coincides with the sign of $u(\eta)-c(\eta)$ for values of $\eta$ where the former is defined  ($\eta\in [\eta_0,\eta_f)$, see \appref{existence}). Also $d'(\eta_{Het}^L)=k(u'(\eta_{Het}^L)-c'(\eta_{Het}^L))$ for some $k>0$. Therefore $u'(\eta_{Het}^L)-c'(\eta_{Het}^L)<0$ by \eqref{Melnikov} for any heteroclinic connection. But then by continuity of $u(\eta)-c(\eta)$, it follows that $u(\eta)-c(\eta)$ can only change sign once and hence the heteroclinic connection is unique.

 Now, finally we prove item \ref{item3Here}. Fix $\xi>0$ sufficiently small and consider $\eta$ sufficiently close to $\eta_{Het}^L$. Let $u_1(\eta)$ and $c_1(\eta)$ denote the $r_4$-value of the first intersections of $U_4^L$ and $C_L^4$ with the section
 \begin{align*}
  \Sigma_1 &= \{(r_4,\delta_4)\in [0,\infty)^2\vert \delta_4 = \xi,\,r_4\in I_1\}.
  \end{align*}
  Here $I_1$ is an appropriate closed interval such that $U_{4,loc}^L$ intersects $\Sigma_1$ . Furthermore, set
 \begin{align*}
  \Sigma_0 &= \{(r_4,\delta_4)\in [0,\infty)^2\vert \delta_4 = \xi,\,r_4\in I_0\},
 \end{align*}
 with $I_0= (0,\nu]$ for $\nu$ sufficiently small.  See \figref{hetero}(b). Then we construct limit cycles by flowing points on $\Sigma_0$ forward and backward in time  and measure their intersection with $\Sigma_1$. In particular, let $P(r_4,\eta)$ denote the $r_4$-value of the first intersection with $\Sigma_0$ of the forward orbit from $(r_4,\nu) \in \Sigma_0$. Similarly, we let $\widehat P(r_4,\eta)$ denote the $r_4$-value of the first intersection of the backward orbit from $(r_4,\nu)\in \Sigma_0$. Periodic solutions are therefore solutions of $P(r_4,\eta)=\widehat P(r_4,\eta)$ with $r_4\in I_0$.  
 Now, $b_4$ and $q_{w,4}^L$ are both saddles while $q_{r,4}^L$ is a nonhyperbolic saddle for $\eta$ near $\eta_{Het}^L<\eta^L$. Therefore, if we fix $\theta \in (0,1)$, then for $I_0$ sufficiently small, it is easy to show, using local analysis near the saddles $q_{w,4}$ (hyperbolic) and $q_{r,4}$ (nonhyperbolic), the following estimates
 \begin{align}
 P(r_4,\eta) - u_1(\eta) &\in \left( c r_4^{\alpha^{-1}(1-\alpha)k \theta},c^{-1} r_4^{\alpha^{-1}(1-\alpha)k \theta^{-1}}\right)
                                         \eqlab{Pest}\\
                                         \widehat P(r_4,\eta) - c_1(\eta) &\in \left(c e^{ -r_4^{2k+1} \alpha^{-1} k\beta\phi^L(0) \theta}, c^{-1} e^{ -r_4^{2k+1} \alpha^{-1} k\beta\phi^L(0) \theta^{-1}}\right),\eqlab{Phatest}
 \end{align}
for $c>0$ sufficiently small. Therefore $P$ and $\widehat P$ are both continuous at $r_4=0$ with values $u_1(\eta)$ and $c_1(\eta)$, respectively, and therefore satisfy $P(0,\eta_{Het}^L)=\widehat P(0,\eta_{Het}^L)$. Therefore by the transverse intersection of $U^L_4$ with $C^L_4$ along $\eta=\eta_{Het}^L$ we have that there exists a continuous function $\eta_{Per}^L(r_4)$, $r_4\in I_0$, such that 
\begin{align*}
 P_4(r_{4},\eta_{Per}^L(r_4)) = \widehat P_4(r_{4},\eta_{Per}^L(r_4)).
\end{align*}
In particular, using that $c_1'(\eta_{Het}^L)-u_1'(\eta_{Het}^L)>0$ cf. \eqref{Melnikov} and the estimates \eqsref{Pest}{Phatest} we have that $\eta_{Per}^L(r_4)>\eta_{Het}^L$ for $r_4\in I_0$. The periodic orbits are hyperbolic and unstable because if we consider the Poincare map obtained by the forward flow from $\Sigma_0$ to itself, then its derivative will be large due to \eqref{Phatest} and the existence of the repelling center manifold. 

\end{proof}
\begin{figure}[h!] 
\begin{center}
\subfigure[$\eta=\eta^L-\nu_0$]{\includegraphics[width=.4\textwidth]{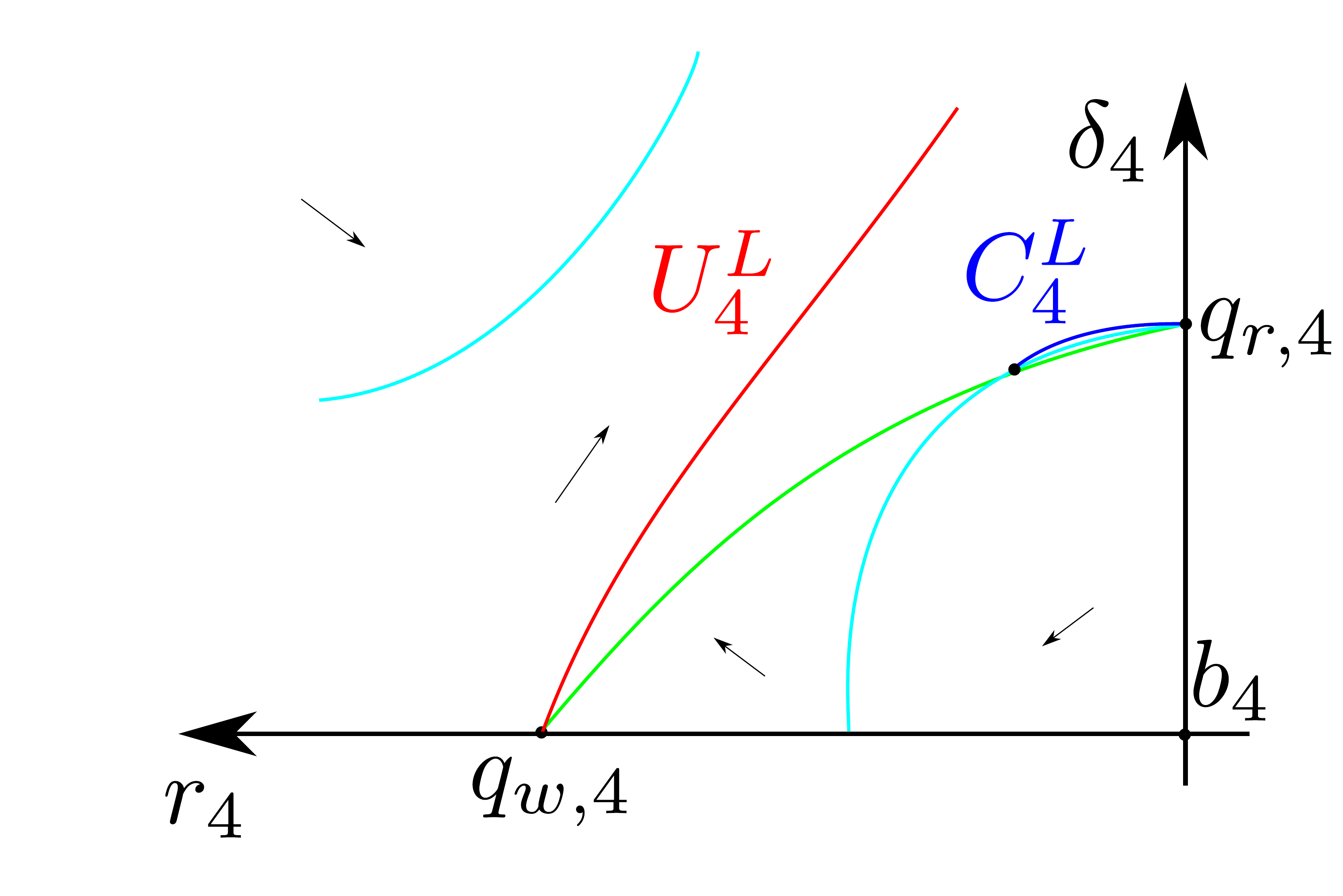}}
\subfigure[$\eta = \eta_{Het}^L$]{\includegraphics[width=.4\textwidth]{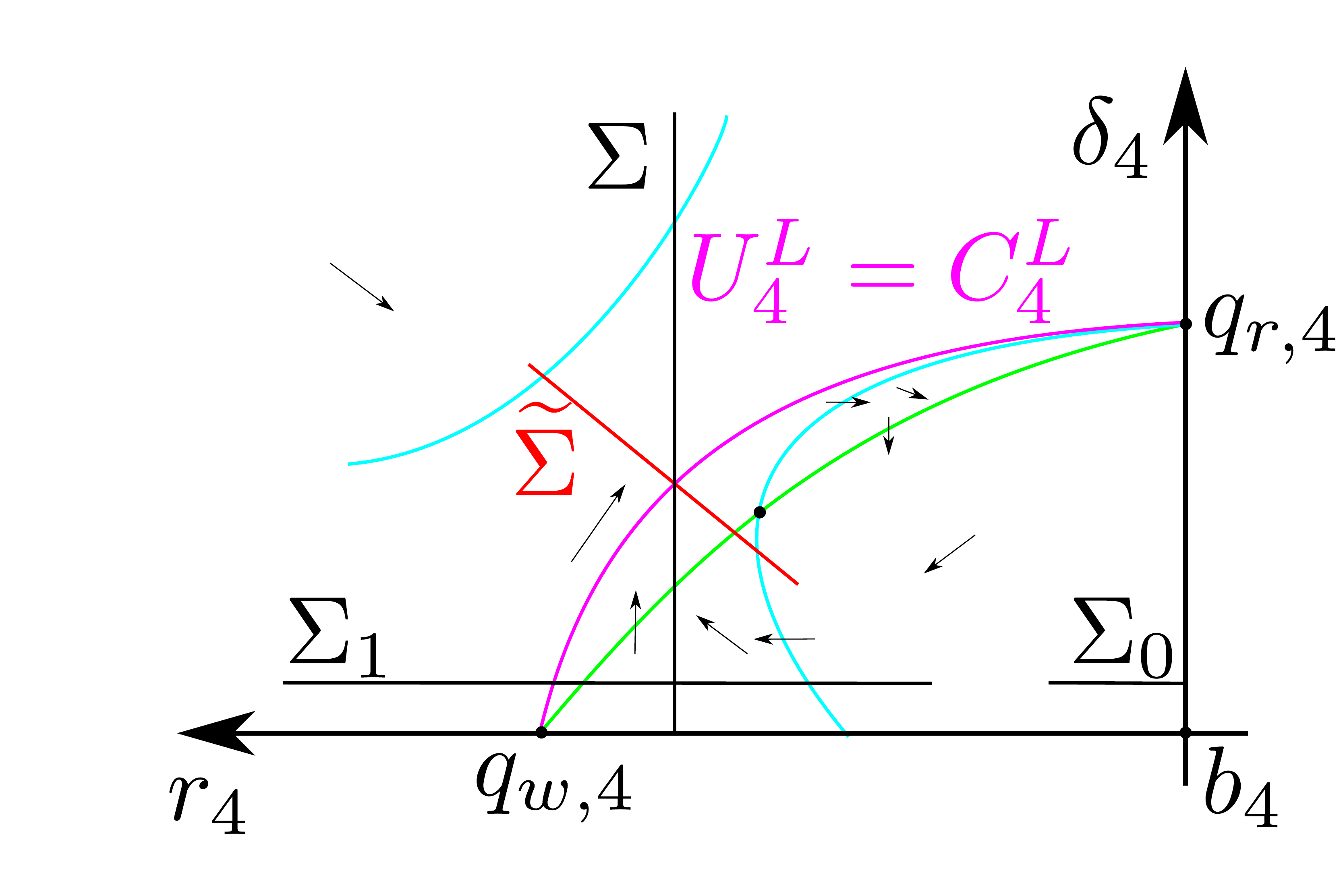}}
\subfigure[$\eta=\eta_0$]{\includegraphics[width=.4\textwidth]{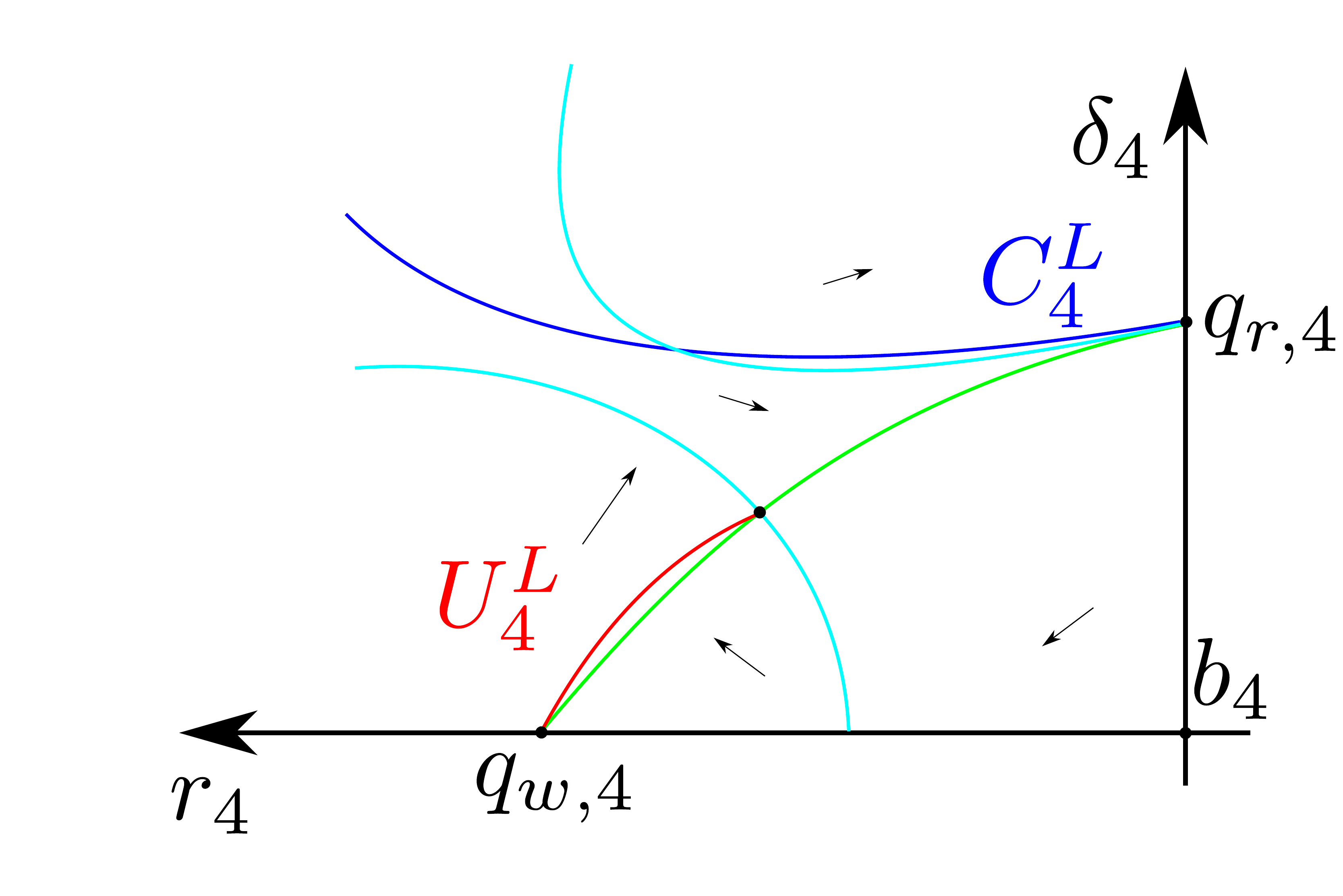}}
\end{center}
 \caption{Three different scenarios in the chart $\bar y=-1$, following \lemmaref{pLChartr1} and \lemmaref{pLChartd1} (items $5$). (a) For $\eta=\eta^L-\nu_0$, $C_4^L$ is backwards asymptotic, as a trajectory, to the unstable node $z_4^L$ whereas $U_4^L$ is unbounded. (b) For $\eta=\eta^L_{Het}$ a heteroclinic exists. (c) The unstable manifold $U_4^L$ is asymptotic, as a trajectory, to the stable node $z_4^L$ whereas $C_4^L$ is unbounded. The nullclines of $\delta_4$ and $r_4$ are shown in cyan and green, respectively.}
\figlab{hetero}
\end{figure}

\subsection{Existence of heteroclinic}\applab{existence}
We fix the interval $I=(\eta_0,\eta^L-\nu_0)$ with $\eta_0$ and $\nu_0$ having the properties described in \lemmaref{pLChartr1} and  \lemmaref{pLChartd1} (items $5$), respectively. Heteroclinics can only occur for $\eta$-values in the interval $I$. Consider \figref{hetero} and let $\Sigma$ be a vertical section, i.e. parallel to the $\delta_4$-axis, at an $r_4$-value slightly below $r_4 =\left(\frac{\alpha}{1-\alpha}\right)^{1/(k+1)}$, the $r_4$-value of $q_{w,4}^L$, see \eqref{qw4L} and \figref{hetero}(b). Let $u(\eta)$ denote the first intersection of $U_{4}^L$ with $\Sigma$. Since $\Sigma$ is parallel to the $\delta_4$-axis we will for simplicity use the same symbol $u(\eta)$ for the $\delta_4$-coordinate of $u(\eta)$ (the $r_4$-coordinate being a constant). We apply a similar identification of other points on $\Sigma$ when required in the following. $u(\eta)$ is smooth in $\eta\in I$. Also by  \lemmaref{pLChartr1}, item $2$ it is always above $f(\eta)$ which we define as the unique intersection of the $r_4$-nullcline with $\Sigma$ (see description of the nullclines in \appref{uniqueness}). Now for $\eta=\eta_0$, $C_4^L$ intersects $\Sigma$ once transversally at a point $c_0$. By the implicit function theorem we therefore obtain a smooth curve $c(\eta)$, $\eta\in [\eta_0,\eta_0+\xi]$ with $c(\eta_0)=c_0$ for $\xi>0$ sufficiently small, of first intersection points of $C_4^L$ with $\Sigma$. We can continue to extend $c(\eta)$ as a smooth function for larger $\eta$-values, again invoking the implicit function theorem successively. This process cannot continue for all $\eta\in I$ by \lemmaref{pLChartd1}, item $5$ and the fact that $C_4^L$ does not intersect $\Sigma$ at $\eta=\eta^L-\nu_0$. Therefore there exists an $\eta_f$ such that the continuation of $c(\eta)$ cannot go beyond $\eta_f\in I$. Clearly, either (i) $c(\eta)$ goes unbounded or (ii) $C_4^L$ becomes tangent to $\Sigma_4$ as $\eta\rightarrow \eta_f^-$. A simple phase plane analysis, using the description of nullclines in \appref{uniqueness}, shows that $C_4^L$ is bounded within the vertical strip obtained by $\Sigma$ and the $\delta_4$-axis. Therefore (ii) must hold. But then $c(\eta)\rightarrow f(\eta_f)$ as $\eta \rightarrow \eta_f^-$ since this is the only tangency with $\Sigma_4$. But then by continuity, and the fact that $f(\eta_f)<u(\eta_f)$, there exists at least one $\eta$-value $\eta^L_{\text{Het}}\in (\eta_0,\eta_f)$ such that $u(\eta^L_{\text{Het}})=c(\eta^L_{\text{Het}})$. This proves existence of a heteroclinic.
\subsection{Monotonicity of heteroclinic}\applab{uniqueness}
We consider the parametrization of the heteroclinic in \eqref{heteroclinicPara} but for simplicity we drop the tilde in the following.
Suppose that either $r_4(t)$ or $\delta_4(t)$ is not a monotone function. Let $t_1$ be the least positive time where either $r_4'(t_1)=0$ or $\delta_4'(t_1)=0$. Then $r_4'(t)<0$, $\delta_4'(t)>0$ for all $t<t_1$ by the local analysis near $q_{w,4}^L$, see \lemmaref{pLChartr1} item \ref{item3PlChartr1}. We first show that $r_4'(t_1)=0$ is impossible by contradiction. 

For this, we first study the $r_4$-nullcline of \eqref{Xbary1pLrho0New}. Clearly, it is the union of the set $r_4=0$ and the set defined by equation
\begin{align}
 r_4^{k+1}\left(1-\alpha + \frac{\beta}{\alpha}\delta_4^{k(k+1)}\phi^L(0)-r_4^k\delta_4^k (\eta-\eta^L(\mu))\right) = \alpha - \frac{\beta}{\alpha}\delta_4^{k(k+1)}\phi^L(0).\eqlab{r4Nc}
\end{align}
The big bracket on the left hand side is always positive for $\eta<\eta^L(\mu)$ and increasing with respect to $\delta_4$. The right hand side is just $F_4(0,\delta_4^{k(k+1)})$ and it is decreasing with respect to $\delta_4$, vanishing only at $$\delta_{4,0}= \left(\frac{\alpha^2}{\beta\phi^L(0)}\right)^{1/(k(k+1))},$$ i.e. the $\delta_4$-value of $q_{r,4}^L$, see \eqref{qw4L}. Therefore the set defined by \eqref{r4Nc} within the positive quadrant of the $(r_4,\delta_4)$-plane is therefore a graph $r_4=N_{r_4}(\delta_4)$ over $\delta_4 \in \left[0,\delta_{4,0}\right]$ 
 having a negative slope: $N_{r_4}'(\delta_4)<0$. This curve separates a bounded, triangular region, where $r_4'>0$ from an unbounded region where $r_4'<0$ (see green curve in \figref{hetero}). Now, since $\delta_4'(t)>0$ for all $t\le t_1$, we therefore conclude that if $r_4'(t_1)=0$ then the orbit has to come from within the bounded region. But this contradicts the definition of $t_1$ and the fact that the heteroclinic is asymptotic to $q_{w,4}^L$ for $t\rightarrow -\infty$. The only possibility is therefore $\delta_4'(t_1)=0$. The remainder of this section is devoted to excluding this case.

We therefore now consider the $\delta_4$-nullcline of \eqref{Xbary1pLrho0New}. Clearly, it is the union of the set $\delta_4=0$ and the set defined by
\begin{align}
 r_4^{k+1}\left(1+\frac{1}{k}\left(\alpha-\frac{\beta}{\alpha}\delta_4^{k(k+1)}\phi^L(0)+ r_4^k\delta_4^k (\eta-\eta^L(\mu))\right)\right) = \alpha - \frac{\beta}{\alpha}\delta_4^{k(k+1)}\phi^L(0).\eqlab{d4Nc}
\end{align}
For any fixed $\delta_4\ge 0$, we let $H_{\delta_4}(r_4)$ denote the left hand side. Then 
\begin{lemma}\lemmalab{uniquenesslemma}
Consider $\eta<\eta^L(\mu)$. For any $\delta_4>0$ let 
\begin{align*}
 r_{4,0} &= \frac{1}{\delta_4 (\eta^L(\mu)-\eta)^{1/k}}\left(\frac{k(k+1)}{2k+1}\left(1+\frac{1}{k}\left(\alpha-\frac{\beta}{\alpha}\delta_4^{k(k+1)}\phi^L(0)\right)\right)\right)^{1/k},\\
 r_{4,1} &=\frac{1}{\delta_4 (\eta^L(\mu)-\eta)^{1/k}}\left(1+\frac{1}{k}\left(\alpha-\frac{\beta}{\alpha}\delta_4^{k(k+1)}\phi^L(0)\right)\right)^{1/k},
\end{align*}
Then the function $H_{\delta_4}:[0,\infty)\rightarrow \mathbb R$ satisfies the following:
 \begin{enumerate}[label=(\alph*)]
  \item For $\delta_4=0$, $H_0(0)=0$ and $H_0'(r_4)>0$ for any $r_4>0$. 
  \item For $\delta_4>0$, we have 
  \begin{enumerate}[label=(\alph{enumi}.\arabic*)]
   \item \label{2a} $H_{\delta_4}'(r_{4})>0$ for any $r_4\in (0,r_{4,0})$, $H_{\delta_4}'(r_{4,0})=0$ and $H_{\delta_4}'(r_{4})<0$ for any $r_4>r_{4,0}$. 
    \item \label{2b} Also $H_{\delta_4}(0)=0$, $H_{\delta_4}(r_4)>0$ for any $r_4\in  (0,r_{4,1})$, $H_{\delta_4}(r_{4,1})=0$, $H_{\delta_4}(r_4)<0$ for any $r_{4}\in (r_{4,1},\infty)$. 
  \end{enumerate}

   \end{enumerate}
\end{lemma}
\begin{proof}
 Simple calculations. 
\end{proof}


The right hand side of \eqref{d4Nc} is again just $F_4(0,\delta_4^{k(k+1)})$. An immediate consequence of this lemma, is therefore that for each $\delta_4\ge 0$, there can be $0$, $1$ or $2$ solutions $r_4$ satisfying \eqref{d4Nc}. The exact number depends on how the graph of $H_{\delta_4}$ intersects the constant $F_4(0,\delta_4^{k(k+1)})$ for $r_4\ge 0$. We illustrate this in \figref{HGraph} and consider all cases in the following. 

\begin{figure}[h!] 
\begin{center}
{\includegraphics[width=.55\textwidth]{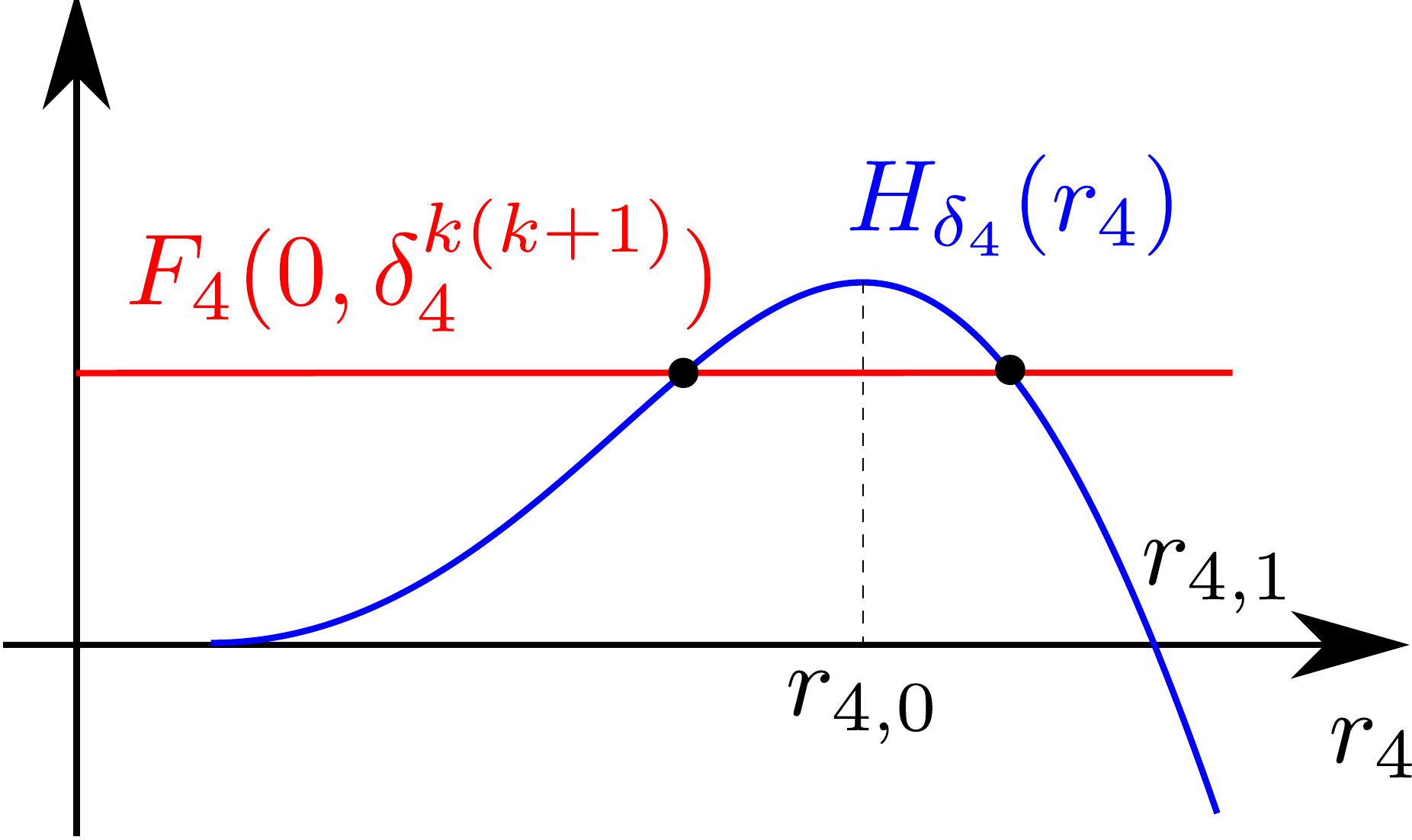}}
\end{center}
 \caption{We describe the $\delta_4$-nullcline as the intersection of the graph of $H_{\delta_4}(r_4)$ with $F_4(0,\delta_4^{k(k+1)})$. There can be at most two intersections for each $\delta_4\ge 0$, see \lemmaref{uniquenesslemma}. }
\figlab{HGraph}
\end{figure}

By \ref{2a} and \ref{2b} in \lemmaref{uniquenesslemma}, it follows that for each $\delta_4>0$ sufficiently small, the graph of $H_{\delta_4}(r_4)$ over $r_4$ intersects $F_4(0,\delta_4^{k(k+1)})$ in two points. This implies that the set of points satisfying \eqref{d4Nc} within $\delta_4\in (0,c)$, with $c>0$ sufficiently small, is the union of two smooth graphs $r_4=N_{\delta_4,1}(\delta_4)$, $r_4=N_{\delta_4,2}(\delta_4)$ over $\delta_4\in (0,c)$. These graphs do not intersect for $c>0$ sufficiently small and by \ref{2a} the graph corresponding to the largest values of $r_4$, say $r_4=N_{\delta_4,1}(\delta_4)$ is unbounded: $N_{\delta_4,1}(\delta_4)\rightarrow \infty$ as $\delta_4\rightarrow 0^+$. The other one, $r_4=N_{\delta_4,2}(\delta_4)$, is bounded and intersects the $r_4$-axis in $r_4=r_{4,2}$ where
\begin{align*}
 r_{4,2} = \left(\frac{k\alpha}{k+\alpha}\right)^{1/(k+1)},
\end{align*}
is the unique solution of \eqref{d4Nc} with $\delta_4=0$.

Similarly, for $\delta_4>\delta_{4,0}$, the right hand side of \eqref{d4Nc} is negative $F_4(0,\delta_4^{k(k+1)})<0$. Therefore by \ref{2b} the set of points $r_4> 0,\delta_4>\delta_{4,0}$ satisfying \eqref{d4Nc} is a graph $r_4=N_{\delta_4,3}(\delta_4)$ over $\delta_4> \delta_{4,0}$. This graph intersects the line $\delta_4=\delta_{4,0}$ in a point $(r_4,\delta_4)=(r_{4,3},\delta_{4,0})$ where 
%
%
\begin{align*}
 r_{4,3} =  \frac{1}{\delta_{4,0}}\left(\frac{k}{\eta^L(\mu)-\eta}\right)^{1/k}. 
\end{align*}
This value of $r_4$ is obtained by setting $\delta_4=\delta_{4,0}$ in \eqref{d4Nc} and canceling out the trivial solution $r_4=0$.
Clearly, by the implicit function theorem the graph $r_4=N_{\delta_4,3}(\delta_4)$ extends smoothly to $\delta_4 \ge \delta_{4,0}-c$, for $c>0$ sufficiently small, as a set of solutions of \eqref{d4Nc}. 
 Notice that the solution $(r_4,\delta_4)=(0,\delta_{4,0})$ of \eqref{d4Nc} corresponds to $q_{r,4}^L$. From here, cf. \ref{2b} and the implicit function theorem, a separate branch of solutions of \eqref{d4Nc} also emanate as a graph $r_4=N_{\delta_4,4}(\delta_4)$ over $\delta_4\in (\delta_{4,0}-c,\delta_{4,0}]$, for $c>0$ sufficiently small. Here $N_{\delta_4,4}(\delta_4)>0$ for $\delta_4\in (\delta_{4,0}-c,\delta_{4,0}]$.  This gives the picture in \figref{nullclinesA}. 
 
 \begin{figure}[h!] 
\begin{center}
{\includegraphics[width=.7\textwidth]{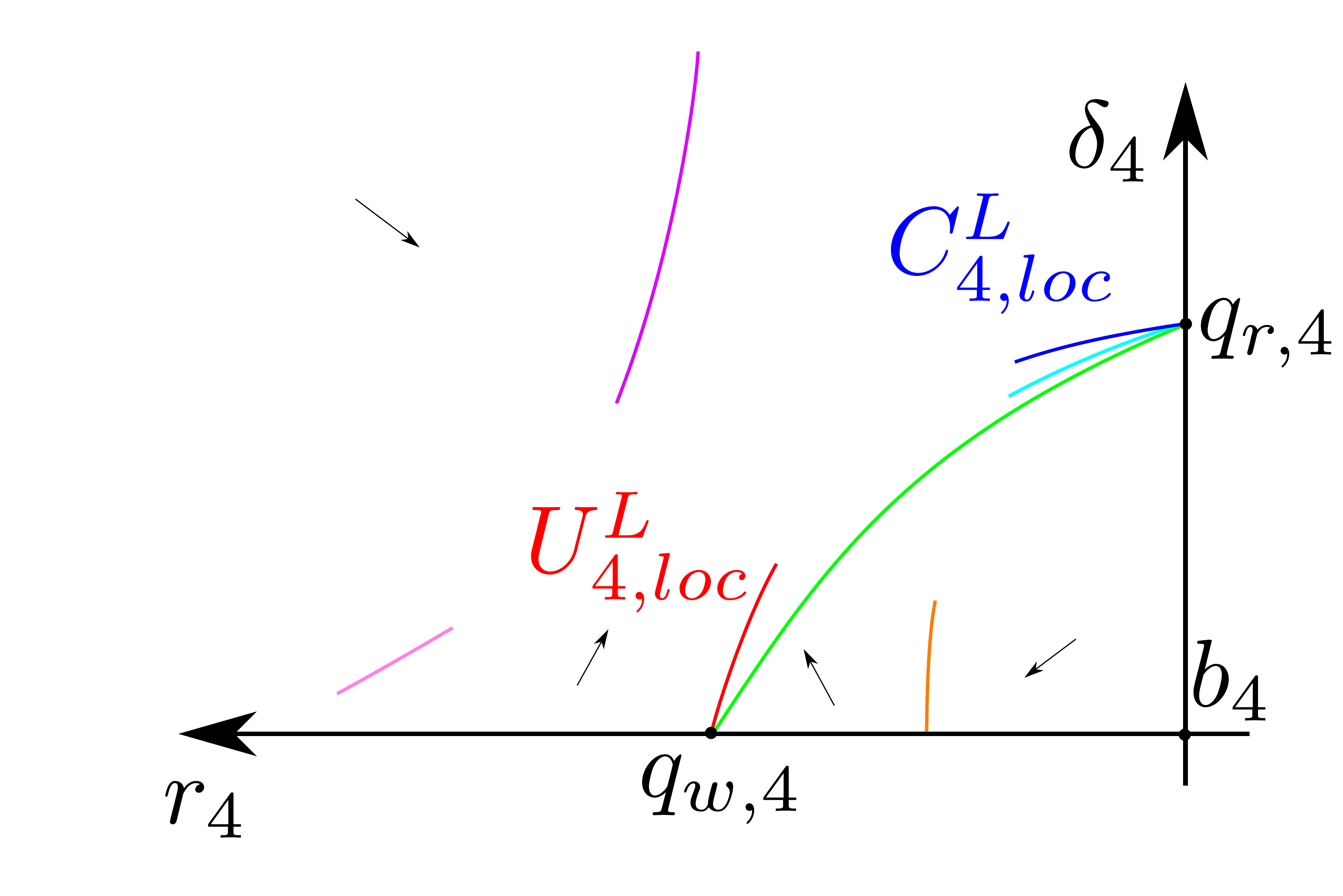}}
\end{center}
 \caption{Nullclines of $r_4$ (in green) and $\delta_4$. Using \lemmaref{uniquenesslemma} we obtain local information about the $\delta_4$-nullcline  (purple, pink, orange and cyan). In reference to the text, the purple is the graph of $N_{\delta_4,3}$, the cyan is the graph of $N_{\delta_4,4}$, while the pink and orange curves are graphs of $N_{\delta_4,1}$ and $N_{\delta_4,2}$, respectively. }
\figlab{nullclinesA}
\end{figure}
 
 From this geometric viewpoint it follows that two branches of the $\delta_4$-nullcline can have folds when the right hand side of \eqref{d4Nc}, $F_4(0,\delta_4^{k(k+1)})$, is equal to the value of $H_{\delta_4}(r_{4,1}(\delta_4))$ at the tangency of the graph of $H_{\delta_4}$. This condition gives the following equation
 \begin{align}
  c(k,\eta) = \delta_4^{k+1} \frac{F_4(0,\delta_4^{k(k+1)})}{\left(1+\frac{1}{k}F_4(0,\delta_4^{k(k+1)})\right)^{(2k+1)/k}},\eqlab{delta4Eqn}
 \end{align}
where
\begin{align*}
 c(k,\eta ) = \frac{k}{(\eta^L(\mu)-\eta)^{(k+1)/k} (2k+1)}\left(\frac{k(k+1)}{2k+1}\right)^{(k+1)/k},
\end{align*}
is a positive constant. It is then a straightforward computation to show that the graph of the right hand side of \eqref{delta4Eqn} over $\delta_4\in (0,\delta_{4,0}]$ has at most one tangency point, say $\delta_{4,1}\in (0,\delta_{4,0})$, being increasing for $\delta_4\in (0,\delta_{4,1})$ and decreasing for $\delta_4\in (\delta_{4,1}, \delta_{4,0})$. Therefore there can be $0,1$ or $2$ solutions of \eqref{delta4Eqn}. Putting all this information together, we only have to study three cases, shown in \figref{nullclinesBCD}. In particular, the analysis of \eqref{delta4Eqn} excludes the existence of isolas which would imply at least four solutions of \eqref{delta4Eqn}. In all three cases, a simple phase plane analysis, show that if $\delta_4'(t_1)=0$ then also $r_4'(t_2)=0$ for some $t_2>t_1$. But then, using the local information about $U_{4,loc}^L$ and $C_{4,loc}^L$, it follows $\delta_4'(t_1)=0$ cannot occur, otherwise the heteroclinic would have to self-intersect.

 \begin{figure}[h!] 
\begin{center}
{\includegraphics[width=.495\textwidth]{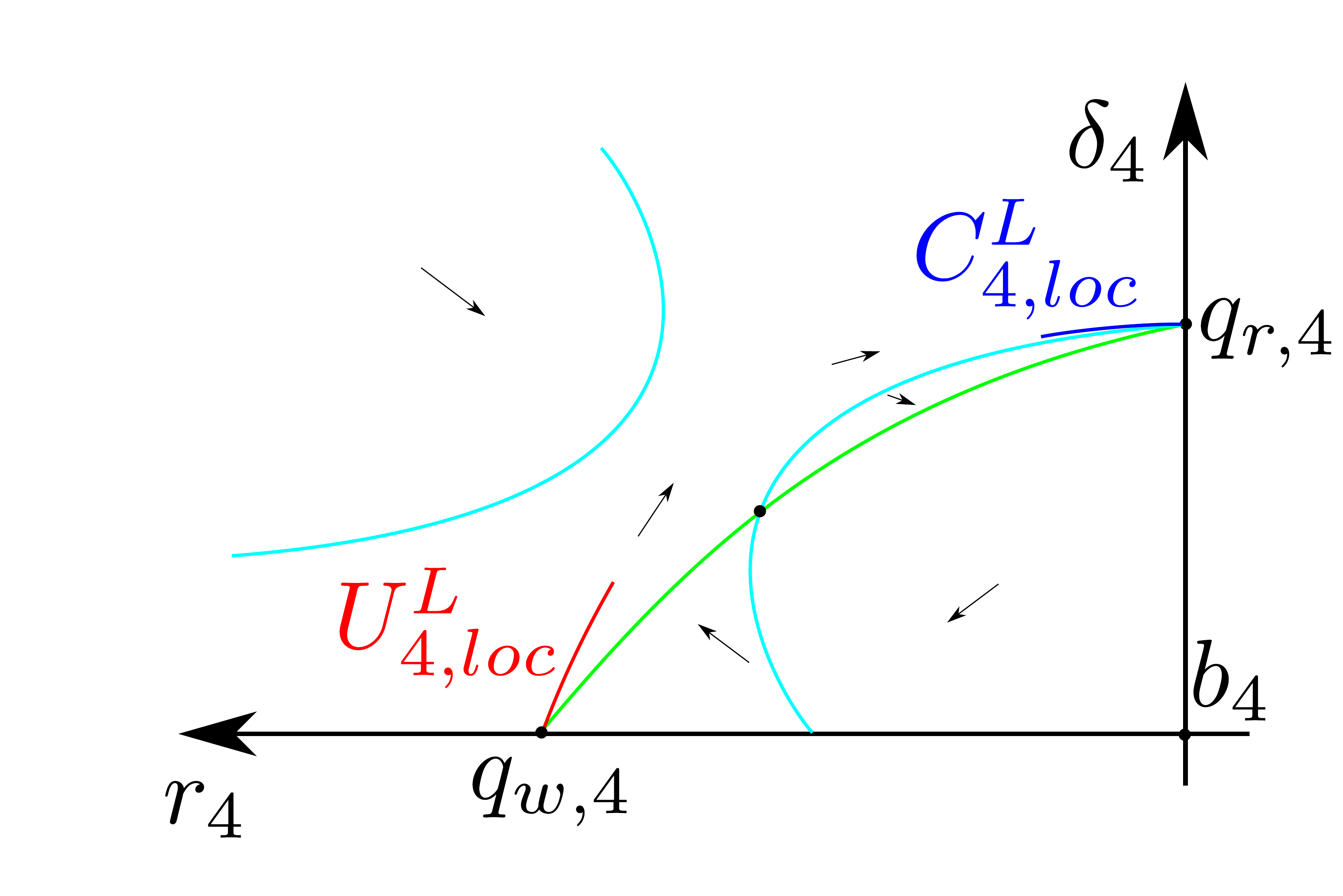}}
{\includegraphics[width=.495\textwidth]{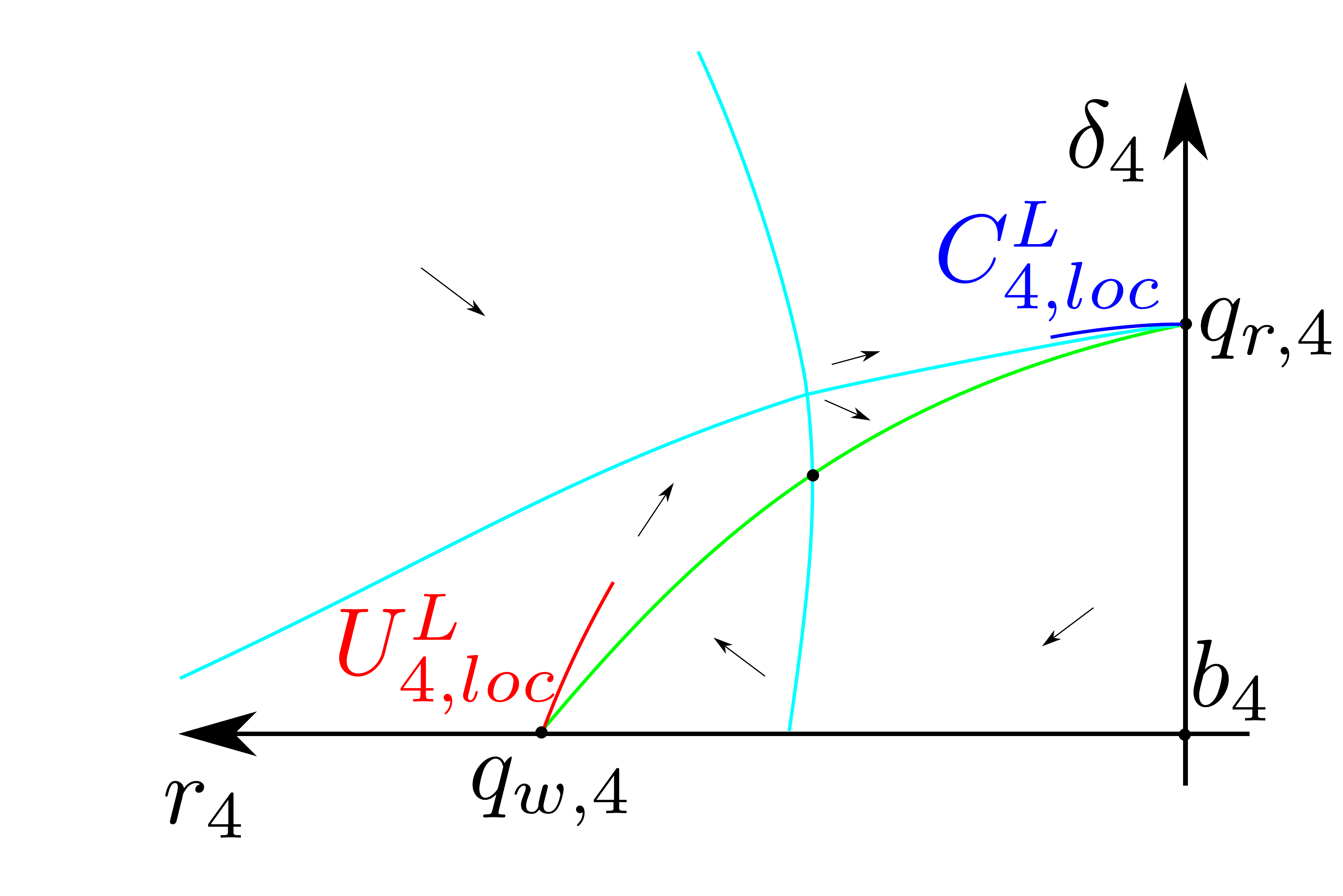}}
{\includegraphics[width=.495\textwidth]{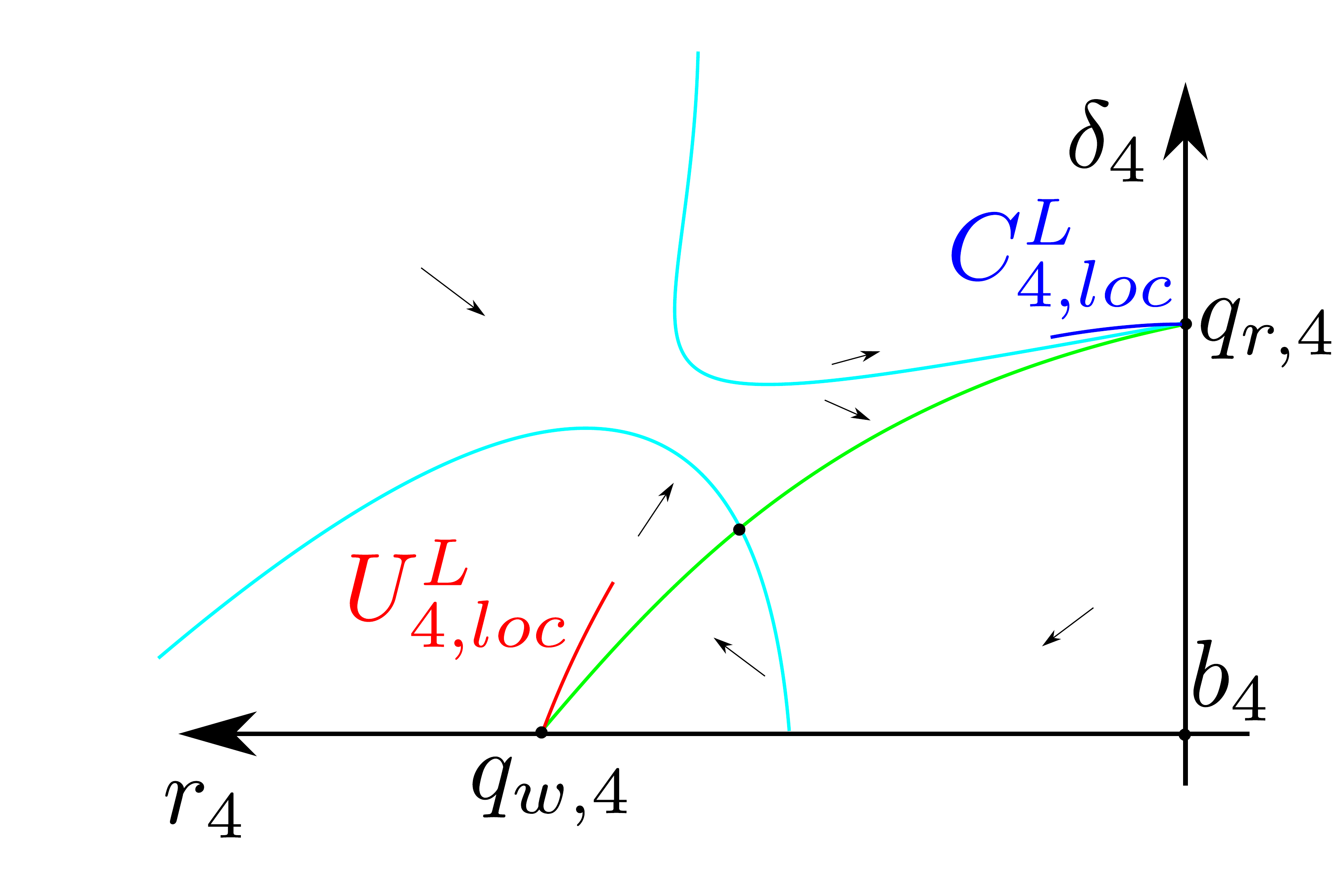}}
\end{center}
 \caption{Three possible topologies of the nullclines of $\delta_4$. In (a) there are no fold points (tangencies parallel to the $r_4$-axis, corresponding to solutions of \eqref{delta4Eqn}) of the $\delta_4$-nullcline, in (b) there is one (degenerate) fold. Finally, in (c) there are two folds. Heteroclinic connections can only occur in the case illustrated in (a) in which case the connection is monotone $r_4'(t)<0$, $\delta_4'(t)>0$.  }
\figlab{nullclinesBCD}
\end{figure}


%

\section{Proof of \propref{prop2}}\applab{transformation}
To realize that \propref{prop1} implies \propref{prop2}, let $\gamma^i=(\alpha,\beta,\eta,\mu,\phi^i(0))$, $i=L,R$ and write the vector-fields $\widehat X_i$, $i=1,3$ on $\overline S^{L/R}$ as $\widehat X_{1,\gamma^L}$ and $\widehat X_{3,\gamma^R}$, respectively, to highlight the dependency on the parameters. Then we have the following 
\begin{lemma}
Fix any $\gamma^L=(\alpha,\beta,\eta,\mu,\phi^L(0))$ and let $\tilde \gamma^R$ be defined by
\begin{align*}
 \tilde \gamma^R = (\alpha^{-1}-\beta,\beta,-\eta,-\mu\alpha^{-2},\phi^L(0)\alpha^{-3}).
\end{align*}
Then $\widehat X_{1,\gamma^L}$ and $\widehat X_{3,\tilde \gamma^R}$ are topologically equivalent. 
\end{lemma}
\begin{proof}
 We find it easiest to work with \eqref{xyFastExtDelta}. For this system, we apply a transformation of parameters defined by the following:
  \begin{align}\eqlab{parameters}
 \begin{cases}
  \frac{1}{\alpha+\beta}&\mapsto \alpha,\\
  \eta &\mapsto -\eta,\\
 \frac{\mu}{(\alpha+\beta)^2} &\mapsto -\mu,\\
  \phi^R(0) &\mapsto (\alpha+\beta)^2 \phi^L(0).
 \end{cases}
\end{align}
 It is then possible to transform the system near $z^R$ into the one at $z^L$ up to terms which vanish on the sphere $\overline S^L$ setting $\rho_1=0$, recall \eqref{Psi1}, by doing the following: (i) multiply the right hand side by $\alpha$ and (ii) apply a simple linear transformation of $x$ and $y$, serving to bring the eigendirections into the ones at $z^L$. With $\tilde \gamma^R$ obtained from \eqref{parameters}, this produces a smooth conjugacy between $\widehat X_{1,\gamma^L}$ and $\widehat X_{3,\tilde \gamma^R}$. 
\end{proof}
\begin{remark}
 It is easy to check that e.g. $\eta^L_H$ in \eqref{etaHL} can be obtained from the $\eta^R_H$ in \eqref{etaHR} upon applying \eqref{parameters}.
\end{remark}

\section{Proof of \lemmaref{Plocal1}}\applab{lemma5}
In the corresponding chart, we have
 \begin{align}
 \dot \rho_1 &= -\frac{1}{k(k+1)}\rho_1 \left(1-F_1((\rho_1\delta_1)^{k+1},\delta_1^{k(k+1)},y_1)\right),\eqlab{XbarR1pLNew}\\
 \dot y_1 &=F_1((\rho_1\delta_1)^{k+1},\delta_1^{k(k+1)},y_1)+\delta_1^k\left(\eta-\mu (y^L+\rho_1^{k(k+1)} y_1)\right)+\left(1-F_1((\rho_1\delta_1)^{k+1},\delta_1^{k(k+1)},y_1)\right)y_1,\nonumber\\
 \dot \delta_1 &=\frac{1}{k}\delta_1 \left(1-F_1((\rho_1\delta_1)^{k+1},\delta_1^{k(k+1)},y_1)\right),\nonumber
\end{align}
with
\begin{align}
 F_1((\rho_1\delta_1)^{k+1},\delta_1^{k(k+1)},y_1)=-\frac{\beta}{\alpha}\delta_1^{k(k+1)} \phi^L((\rho_1\delta_1)^{k+1}) - \left(\alpha+\beta (\rho_1\delta_1)^{k(k+1)} \phi^L((\rho_1\delta_1)^{k+1})\right)y_1.\eqlab{F1Eqn}
\end{align}
 see also \eqref{XbarR1pL}. By construction, $\rho_1^{k+1}\delta_1= \sigma = \text{const}.$, see \eqref{barXN1New} and \eqref{barR1pL}. Inserting this into \eqref{F1Eqn} we obtain 
 \begin{align*}
 F_1((\rho_1\delta_1)^{k+1},\delta_1^{k(k+1)},y_1)=F_1(\sigma \delta_1^{k},\delta_1^{k(k+1)},y_1)
\end{align*}
 and therefore the $\rho_1$-equation decouples and we are left with the following reduced system:
 \begin{align}
   \dot y_1 &=F_1(\sigma \delta_1^{k},\delta_1^{k(k+1)},y_1)+\delta_1^k \eta-\mu (\delta_1^k y^L +\sigma^k y_1)+\left(1-F_1(\sigma \delta_1^{k},\delta_1^{k(k+1)},y_1)\right)y_1,\eqlab{XbarR1pLNewReduced}\\
 \dot \delta_1 &=\frac{1}{k}\delta_1 \left(1-F_1(\sigma \delta_1^{k},\delta_1^{k(k+1)},y_1)\right),\nonumber
 \end{align}
Now, we consider this system as a planar system depending on $\sigma$ as a parameter. Then $(y_1,\delta_1) = (-\alpha^{-1}(1-\alpha),0)=(B,0)$ is clearly a saddle for any $\sigma>0$: The eigenvalues are $-(1-\alpha)<0$ and $\alpha/k>0$. Near $(y_1,\delta_1) = (B,0)$ we divide the right hand side of \eqref{XbarR1pLNewReduced} by $\left(1-F_1(\sigma \delta_1^{k},\delta_1^{k(k+1)},y_1)\right)\approx \alpha$ to obtain 
 \begin{align}
   \dot y_1 &=\left(1-F_1(\sigma \delta_1^{k},\delta_1^{k(k+1)},y_1)\right)^{-1}\left(F_1(\sigma \delta_1^{k},\delta_1^{k(k+1)},y_1)+\delta_1^k \eta-\mu (\delta_1^k y^L +\sigma^k y_1)\right)+y_1,\eqlab{XbarR1pLNewReduced2}\\
 \dot \delta_1 &=\frac{1}{k}\delta_1,\nonumber
 \end{align}
Now, \eqref{XbarR1pLNewReduced2} is linearizable by a local transformation $(y_1,\delta_1)\mapsto (\tilde y_1,\delta_1)$ given by the following equation
\begin{align}
y_1 = H(\tilde y_1,\delta_1^k,\sigma),\eqlab{Hlocal1}
\end{align}
where $H$ is $C^1$ in all its arguments. Indeed any planar system is $C^1$ linearizable near a hyperbolic equilibrium,  see e.g. \cite{sell1985a,sternberg1958a}. The fact that $H$ depends explicitly on $\delta_1^k$ (rather than just $\delta_1$) follows from the fact that we can write \eqref{XbarR1pLNewReduced2} as a smooth system in terms of $(y_1,\delta_1^k)$. The linearized system is
\begin{align*}
 \dot{\tilde y}_1 &=-\alpha^{-1}(1-\alpha)\tilde y_1,\\
 \dot \delta_1 &=\frac{1}{k}\delta_1.
\end{align*}
Solving this system with $\delta_1(T) = \nu$ gives $T=k\log (\nu \delta_1(0)^{-1})$ and hence $$\tilde y_1(T) = \left(\delta_1(0)\nu^{-1}\right)^{k\alpha^{-1}(1-\alpha)}\tilde y_1(0)=\left(\delta_1(0)\nu^{-1}\right)^{k\vert B\vert}\tilde y_1(0).$$ Using \eqref{Hlocal1} and its $C^1$ inverse $\tilde y_1=\tilde H(y_1,\delta_1^k,\sigma)$ we transform this result back to the original variables and obtain the desired expression for $P_1$. The proof of the remaining claims are either straightforward or follow almost directly.
\section{Proof of \lemmaref{Plocal2}}\applab{lemma6}
In the corresponding chart, we have
\begin{align}
 \dot \rho_2 &= \frac{1}{k+1}\rho_2 \left(r_2^{k+1}-F_2(\rho_2^{k+1},y_2)\right),\eqlab{pLdelta1New}\\
 \dot r_2 &=-r_2 \left(r_2^{k+1}-F_2(\rho_2^{k+1},y_2)\right),\nonumber\\
 \dot y_2 &=r_2^{k+1} \left(F_2(\rho_2^{k+1},y_2)+r_2^k \left(\eta-\mu \left(y^L+\rho_2^{k(k+1)} y_2\right)\right)\right)-k y_2 \left(r_2^{k+1}-F_2(\rho_2^{k+1},y_2)\right),\nonumber
\end{align}
see also \eqref{pLdelta1},
where $F_2$ is given in \eqref{pLF2}.
We work close to $q_{f,2}^L$ where $r_2^{k+1}-F_2(\rho_2^{k+1},y_2)\approx \frac{\beta}{\alpha}\phi^L(0)>0$. We therefore divide the right hand side of \eqref{pLdelta1New} by this quantity to obtain
\begin{align*}
 \dot \rho_2 &=\frac{1}{k+1}\rho_2,\\
 \dot r_2 &=-r_2,\\
 \dot y_2 &=r_2^{k+1} \left(r_2^{k+1}-F_2(\rho_2^{k+1},y_2)\right)^{-1}\left(F_2(\rho_2^{k+1},y_2)+r_2^k \left(\eta-\mu \left(y^L+\rho_2^{k(k+1)} y_2\right)\right)\right)-ky_2.
\end{align*}
We then follow the approach in \cite[Prop. 2.11]{krupa_extending_2001} and perform a partial linearization. For this, note that  
$q_{f,2}^L$ is a hyperbolic stable node within the $\rho_2=0$ subsystem. Therefore there exists a smooth $H$ such that the transformation $(r_2,y_2)\mapsto (r_2,\tilde y_2)$ of the form 
\begin{align}
 y_2 = H(\tilde y_2,r_2),\eqlab{Hlocal2}
\end{align}
brings the nonlinear system into its linearized form
\begin{align}
 \dot r_2 &=-r_2,\\
 \dot{\tilde y}_2 &=-k \tilde y_2,
\end{align}
within $\rho_2=0$.
But then by applying \eqref{Hlocal2} to the full system we obtain the following
\begin{align*}
 \dot \rho_2 &=\frac{1}{k+1}\rho_2,\\
 \dot r_2 &=-r_2,\\
 \dot{\tilde y}_2 &=-k\tilde y_2+r_2^{k+1}\rho_2^{k+1} G(\rho_2,r_2,y_2),
\end{align*}
for some smooth function $G$, uniformly bounded on the relevant local domain. But then straightforward estimation gives the desired result upon using \eqref{Hlocal2} and its inverse $\tilde y_2=\tilde H(y_2,r_2)$ to transform the estimates back to the $(\rho_2,r_2,y_2)$-variables. 
\section{Proof of \lemmaref{Plocal3}}\applab{lemma7}

We consider \eqref{X3} near $E^R_3\cap \{y_3=y^R\}$. Here $F_3(\delta_3^{k+1},y_3) = 1-(\alpha+\beta(1-\delta_3^{k(k+1)} \phi^R(\delta_3^{k+1})))y_3\approx -\frac{\beta}{\alpha}<0$. We therefore divide the right hand sides by $-\left(r_3^{k+1} +F_3(\delta_3^{k+1},y_3)\right)$ and obtain the following system:
\begin{align}
\dot r_3 &=\frac{1}{1+k}r_3 ,\nonumber\\
\dot y_3 &=- r_3^{k+1} \left(r_3^{k+1} +F_3(\delta_3^{k+1},y_3)\right)^{-1}\left(F_3(\delta_3^{k+1},y_3) +r_3^k\delta_3^k (\eta-\mu y_3)\right),\nonumber\\
\dot \delta_3 &=-\frac{1}{1+k}\delta_3,\nonumber
\end{align}
We then straighten out the unstable fibers of $E^R_3$ by applying a transformation of the form $(r_3,y_3,\delta_3)\mapsto (r_3,\tilde y_3,\delta_3)$ given by
 \begin{align}
  y_3 =  H(\tilde y_3,r_3),\eqlab{Hlocal3}
 \end{align}
where $H$ is smooth and satisfies $H(\tilde y_3,r_3)=\tilde y_3 +\mathcal O(r_3^{k+1})$. In these variables
\begin{align*}
 \dot{\tilde y}_3&=r_3^{k+1}\delta_3^k G(r_3,\tilde y_3,\delta_3)=\sigma^k r_3 G(r_3,\tilde y_3,\delta_3),
\end{align*}
for some smooth $G$ that is uniformly bounded in the relevant neighborhood. Integrating the $r_3$ and $\delta_3$-equations we can subsequently integrate the $\tilde y_3$-equation and obtain the desired result by using \eqref{Hlocal3} and its inverse $\tilde y_3 =\tilde H(y_3,r_3)$ to transform back into the original variables $(r_3,y_3,\delta_3)$. 
\bibliography{refs}
\bibliographystyle{plain}

 \end{document}